\documentclass[
a4paper,
reqno,
11pt,
]{article}
\usepackage{
amssymb,
amsmath,
amsthm,
eucal,
dsfont,
}

\makeatletter
 
  \@addtoreset{equation}{section}
 \makeatother
\theoremstyle{plain}

\newtheorem{theorem}{Theorem}[section]
\newtheorem{lemma}[theorem]{Lemma}
\newtheorem{proposition}[theorem]{Proposition}
\newtheorem{corollary}[theorem]{Corollary}
\theoremstyle{definition}

\newtheorem{remark}[theorem]{Remark}

\newcommand{\bignorm}[1]{{\left\|#1\right\|}}
\newcommand{\norm}[1]{{\|#1\|}}

\newcommand{\wtilde}[1]{{\widetilde{#1}}}
\def\supp{\mathop{\mathrm{supp}}\nolimits}

\def\Id{\mathop{\mathrm{Id}}\nolimits}

\def\Op{\mathop{\mathrm{Op}}\nolimits}
\def\IK{\mathop{\mathrm{IK}}\nolimits}
\def\Range{\mathop{\mathrm{Range}}\nolimits}

\def\sgn{\mathop{\mathrm{sgn}}\nolimits}

\def\Re{\mathop{\mathrm{Re}}\nolimits}
\def\Im{\mathop{\mathrm{Im}}\nolimits}
\def\loc{\mathop{\mathrm{loc}}\nolimits}

\def\com{\mathop{\mathrm{com}}\nolimits}
\def\R{{\mathbb{R}}}
\def\Z{{\mathbb{Z}}}
\def\N{{\mathbb{N}}}
\def\C{{\mathbb{C}}}

\def\S{{\mathcal{S}}}
\def\F{{\mathcal{F}}}
\def\H{{\mathcal{H}}}

\def\D{{\mathcal{D}}}

\def\<{{\langle}}
\def\>{{\rangle}}
\def\ep{{\varepsilon}}

\DeclareMathOperator*{\slim}{s-lim}

\title{Strichartz estimates for Schr\"odinger equations with slowly decaying potentials}
\author{Haruya Mizutani\footnote{Department of Mathematics, Graduate School of Science, Osaka University, Toyonaka, Osaka 560-0043, Japan. E-mail address: \texttt{haruya@math.sci.osaka-u.ac.jp} }}
\date{\empty}

\begin{document}
\maketitle


\begin{abstract}
For Schr\"odinger equations with a class of slowly decaying repulsive potentials, we show that the solution satisfies global-in-time Strichartz estimates for any admissible pairs. Our admissible class of potentials includes the positive homogeneous potential $Z|x|^{-\mu}$ with $Z>0$ and $0<\mu<2$ in three and higher dimensions, especially the repulsive Coulomb potential. The proof employs several techniques from scattering theory such as the long time parametrix construction of Isozaki-Kitada type, propagation estimates and local decay estimates. 
\end{abstract}

%%%%%%%%%%%%%%%%%%%%%%%%%%%%%%%%%%%%%%%%%%%%%%%%%%%%%%%%%%%%%			%Introduction			%%%%%%%%%%%%%%%%%%%%%%%%%%%%%%%%%%%%%%%%%%%%%%%%%%%%%%%%%%%%%%%%%%%
\section{Introduction}
\subsection{Main results}
\label{main_results}

Let $H=-\Delta+V(x)$ be the  Schr\"odinger operator  on $\R^n$ with a real-valued potential $V(x)$ decaying at infinity. Consider the Cauchy problem of Schr\"odinger equation
\begin{align}
\label{equation_1}
i\partial_t u=Hu+F,\quad (t,x)\in \R\times\R^n;\quad u|_{t=0}=u_0,
\end{align}
where $u_0:\R^n\to \C$ and $F:\R\times \R^n\to \C$ are given data. The present paper is concerned with the so-called {\it Strichartz estimates}, which is a family of space--time inequalities of the form
\begin{align}
\label{Strichartz}
\norm{u}_{L^p(\R;L^q(\R^n))}\le C\norm{u_0}_{L^2(\R^n)}+C\norm{F}_{L^{\tilde p'}(\R;L^{\tilde q'}(\R^n))}
\end{align}
where $(p,q)$ and $(\tilde p,\tilde q)$ satisfy the admissible condition
\begin{align}
\label{admissible}
p\ge2,\quad
2/p=n(1/2-1/q),\quad
(n,p,q)\neq(2,2,\infty). 
\end{align}
There is a vast literature on Strichartz estimates for Schr\"odinger equations with potentials (see the   discussion below). However, the case when $V$ is slowly decaying in the sense that 
\begin{align}
\label{slowly_decaying}
|V(x)|\sim |x|^{-\mu},\quad |x|\to \infty,
\end{align}
for some $\mu\in (0,2)$, is less understood. To the author's best knowledge, there is only one negative result in this case given by Goldberg-Vega-Visciglia \cite{GVV} where they showed that if 
$$
V\in C^3(\R^n\setminus\{0\};\R),\quad V(x)=|x|^{-\mu}\wtilde V(\theta),\quad \theta=x/|x|,\quad \mu\in [0,2),
$$
and  $\wtilde V$ has a non-degenerate minimum point so that $\min\limits\wtilde V=0$ then, for any admissible pair, (global-in-time) Strichartz estimates cannot hold in general. Note that radially symmetric potentials clearly do not satisfy the above condition for the counterexample. 

In light of those observations, the purpose of this paper to prove Strichartz estimates for a class of slowly decaying potentials which particularly includes radially symmetric positive potentials $|x|^{-\mu}$ with $\mu\in (0,2)$. More precisely, we first consider the following condition. 

\begin{itemize}
\item[(H1)] For all $\alpha\in \Z_+^n$, there exists $C_\alpha>0$ such that
$$
|\partial_x^\alpha V(x)|\le C_\alpha (1+|x|)^{-\mu-|\alpha|},\quad x\in\R^n.$$
\item[{(H2)}] There exists $C_1>0$ such that $$V(x)\ge C_1(1+|x|)^{-\mu},\quad x\in\R^n.$$  
\item[{(H3)}] There exists $R_0,C_2>0$ such that $$-x\cdot\nabla V(x)\ge C_2(1+|x|)^{-\mu},\quad |x|\ge R_0.$$ 
\end{itemize}

Under condition (H1), $-\Delta+V$ is essentially self-adjoint on $C_0^\infty(\R^n)$ and we denote by $H$ its unique self-adjoint extension on $L^2(\R^n)$ with domain $D(H)=D(-\Delta)$. The solution $u$ to \eqref{equation_1} is given by Duhamel's formula
$$
u=e^{-itH}u_0-i\int_0^te^{-i(t-s)H}F(s)ds.
$$
Moreover, $H$ is purely absolutely continuous: $\sigma(H)=\sigma_{\mathrm{ac}}(H)=[0,\infty)$. In particular, $H$ has no eigenvalue.  The main result for smooth potentials then is as follows.

%theorem
\begin{theorem}
\label{theorem_1}
Let $n\ge 2$, $\mu\in (0,2)$ and $V\in C^\infty(\R^n;\R)$ satisfy \emph{(H1)}--\emph{(H3)}. Then, for any $(p,q)$ and $(\tilde p,\tilde q)$ satisfying \eqref{admissible}, there exists $C>0$ such that
\begin{align}
\label{theorem_1_1}
\norm{e^{-itH}u_0}_{L^p(\R;L^q(\R^n))}&\le C\norm{u_0}_{L^2(\R^n)},\\
\label{theorem_1_2}
\bignorm{\int_0^t e^{-i(t-s)H}F(s)ds}_{L^p(\R;L^q(\R^n))}&\le C\norm{F}_{L^{\tilde p'}(\R;L^{\tilde q'}(\R^n))},
\end{align}
for all $u_0\in L^2(\R^n)$ and $F\in L^{\tilde p'}(\R;L^{\tilde q'}(\R^n))\cap L^1_{\mathrm{loc}}(\R;L^2(\R^n))$. 
\end{theorem}

Here we have taken the condition $F\in L^1_{\mathrm{loc}}(\R;L^2(\R^n))$ to make sure  that the map $F\mapsto \int_0^t e^{-i(t-s)H}F(s)ds$ has clear sense. Of course, \eqref{theorem_1_2} implies that it extends to a bounded operator from $L^{\tilde p'}(\R;L^{\tilde q'}(\R^n))$ to $L^p(\R;L^q(\R^n))$. The same remark also applies to the next theorem. 

We next consider potentials with local singularity. Combining with Theorem \ref{theorem_1} and weighted space-time $L^2$-estimates proved by \cite{BoMi2}, we have the following. 

%theorem
\begin{theorem}	
\label{theorem_2}
Let $n\ge3$, $Z>0$ and $\mu\in (0,2)$. Suppose that $V_S\in C^\infty(\R^n;\R)$ satisfies
$$
|\partial_x^\alpha V_S(x)|\le C_\alpha(1+|x|)^{-1-\mu-|\alpha|},\quad x\in \R^n,
$$ 
for all $\alpha\in \Z^n_+$. Let $H_1=-\Delta+Z|x|^{-\mu}+\ep V_S(x)$. Then there exists $\ep_*=\ep_*(Z,\mu,V_S)>0$ such that for all $\ep\in [0,\ep_*)$ and $(p,q)$ and $(\tilde p,\tilde q)$ satisfying \eqref{admissible}, there exists $C>0$ such that
\begin{align*}
\norm{e^{-itH_1}u_0}_{L^p(\R;L^q(\R^n))}&\le C\norm{u_0}_{L^2(\R^n)},\\
\bignorm{\int_0^t e^{-i(t-s)H_1}F(s)ds}_{L^p(\R;L^q(\R^n))}&\le C\norm{F}_{L^{\tilde p'}(\R;L^{\tilde q'}(\R^n))}
\end{align*}
for all $u_0\in L^2(\R^n)$ and $F\in L^{\tilde p'}(\R;L^{\tilde q'}(\R^n))\cap L^1_{\mathrm{loc}}(\R;L^2(\R^n))$. 
\end{theorem}

%remark
\begin{remark}
Under the conditions in Theorem \ref{theorem_2}, $H_1$ is defined as the unique self-adjoint operator generated by the lower semi-bounded sesquilinear form $\<(-\Delta+Z|x|^{-\mu}+\ep V_S)f,g\>$ on $C_0^\infty(\R^n)$. Moreover, $H_1$ is purely absolutely continuous and has no eigenvalue. 
\end{remark}

%remark
\begin{remark}
It will be seen in the proof that one can choose $\ep_*=\min(Z/M_0,\ \mu Z/M_1)$ in Theorem \ref{theorem_2}, where $M_\ell =\norm{|x|^\mu (x\cdot \nabla)^\ell V_S}_\infty$. 
\end{remark}

\begin{remark}
Both of conditions (H1)--(H3) and conditions in Theorem \ref{theorem_2} do not intersect with one for the above counterexample by \cite{GVV}. 
\end{remark}

%remark
\begin{remark}
The restriction $n\ge3$ on the space dimension in Theorem \ref{theorem_2} is due to the use of the following space-time $L^2$-estimate with a singular weight
\begin{align*}
\norm{\chi(x)|x|^{-\mu/2}e^{-it(-\Delta+V)}u_0}_{L^2(\R^{1+n})}\le C \norm{u_0}_{L^2(\R^n)}
\end{align*}
 for $V$ satisfying (H1)--(H3), where $\chi\in C_0^\infty(\R^n)$. This estimate immediately follows from the endpoint Strichartz estimate  if $n\ge3$ since $\chi(x)|x|^{-\mu/2}\in L^{n}$, but this is not the case if $n=2$ since the endpoint Strichartz estimate cannot hold in two space dimensions. It might be possible to obtain this estimate by revisiting the argument of Subsection \ref{subsection_smoothing} in this paper with a more careful analysis. However, we do not pursue this issue for the sake of simplicity. 
\end{remark}

We here recall some known results. When $V\equiv0$, {\it i.e.},  the free case $H=-\Delta$, Strichartz estimates \eqref{Strichartz} were found by Strichartz \cite{Str} in case of $p=q$ and then extended by \cite{GiVe,Yaj} for all admissible pairs, except for the endpoint cases: $p$ or $\tilde p=2$. The endpoint cases were settled by \cite{KeTa}. While the original proof of Strichartz \cite{Str} relied on a Fourier restriction theorem, the proof by \cite{GiVe,Yaj,KeTa} employed the so-called $TT^*$-argument which is based on the $L^2$-boundedness of the unitary group $e^{it\Delta}$ and the following dispersive estimate
$$
\norm{e^{it\Delta}}_{1\to \infty}\le C|t|^{-n/2},\quad  t\in \R\setminus\{0\}.
$$
In fact, Keel-Tao \cite{KeTa} showed in a quite abstract setting that these two estimates imply Strichartz estimates for all admissible pairs. In case with decaying potentials $V$, there is also a huge literature on Strichartz estimates. In particular, in a seminal paper  \cite{RoSc}, Rodnianski-Schlag showed that if $V$ is of very-short range in the sense that there exists $\ep>0$ such that
\begin{align*}
|V(x)|\le C(1+|x|)^{-2-\ep},
\end{align*}
then Strichartz estimates holds for all non-endpoint admissible pairs, provided that initial data belong to the continuous subspace of $H$ and the zero energy is neither an eigenvalue nor a resonance of $H$. Moreover, in the proof, they introduced a simple but very useful perturbation method which states briefly that, for two self-adjoint operators $H_0$ and $H=H_0+V$, if $e^{-itH_0}$ satisfies the Strichartz estimate for a non-endpoint admissible pair $(p,q)$, and there exists a decomposition $V=A^*B$ such that $A$ is $H_0$-smooth and $BP_{\mathrm{ac}}(H)$ is $H$-smooth in the sense of Kato \cite{Kat}, then $e^{-itH}P_{\mathrm{ac}}(H)$ satisfies the Strichartz estimate for the same pair $(p,q)$. For further results on Strichartz estimates with very short-range potentials in a more general sense, we refer to \cite{Gol,Bec,Miz2}. There are also many results in the case when $V$ is of the inverse-square type in the sense that $|x|^2V\in L^\infty(\R^n)$ (see \cite{BPST2,BoMi2,Miz1}). In the slowly decaying case, there are several results on weighted $L^2$ estimates, or equivalently, uniform resolvent estimates \cite{Nak2,FoSk,Ric,BoMi2}. However, as explained above, there is no known positive result on global-in-time Strichartz estimates. Note that, even if $1<\mu<2$, the property \eqref{slowly_decaying} is too weak to apply Rodnianski-Schlag's method since $\<x\>^{-\alpha}$ is $\Delta$-smooth if and only if $\alpha\ge1$ for $n\ge3$ and $\alpha>1$ for $n=2$ (see \cite{Wal}). 

We conclude this subsection by stating a simple application of the above theorems. As is well known, Strichartz estimates are very useful for studying scattering theory for nonlinear dispersive and wave equations. It is hence very likely that Theorems \ref{theorem_1} and \ref{theorem_2} could be used to study the following nonlinear Schr\"odinger equation
\begin{align}
\label{NLS}
(i\partial_t+\Delta-V)v=\sigma |v|^{p}v,\quad (t,x)\in \R\times\R^n;\quad v(0,x)=v_0(x),\quad x\in \R^n,
\end{align}
for $V$ satisfying (H1)--(H3) or the assumption in Theorem \ref{theorem_2} and a suitable range of $p$, where $\sigma=\pm1$.  For instance, we let $n\ge3$, $H=-\Delta+V$ satisfy (H1)--(H3) and consider the mass-critical case $p=4/n$. Then  one can show in both cases $\sigma=\pm1$ the small data scattering, namely  if $\norm{v_0}_{L^2}$ is sufficiently small then there exists a unique (mild) solution $v\in C(\R;L^2(\R^n))\cap L^{2+4/n}(\R^{1+n})$ to \eqref{NLS} such that $v$ scatters in the sense that there exist $v_\pm\in L^2$ such that $$\lim_{t\to\pm\infty}\norm{v(t)-e^{-itH}v_\pm}_{L^2}=0.$$ This is a straightforward consequence of Theorem \ref{theorem_1} and the proof is analogous to the free case  (see, for instance, \cite[Section 3]{KiVi}). Moreover, by linear scattering theory (see, {\it e.g.}, \cite{DeGe}), we see that, for the short-range case $1<\mu<2$, the wave operators 
$$
W^\pm:=\slim_{t\to\pm\infty}e^{itH}e^{it\Delta}
$$
exist and are asymptotically complete: $\Range W_\pm=\H_{\mathrm{ac}}(H)$, the absolutely continuous subspace of $H$, while in the long-range case $0<\mu\le1$, the modified wave operators 
$$
W^\pm_S:=\slim_{t\to\pm\infty}e^{itH}e^{-iS(t,D)}
$$
exist and are asymptotically complete: $\Range W_S^\pm=\H_{\mathrm{ac}}(H)$, where $S(t,D)=\F^{-1}S(t,\xi)\F$ is a Fourier multiplier by an approximate solution to the Hamilton-Jacobi equation
$$
\partial_t S(t,\xi)=|\xi|^2+V(\nabla_\xi S(t,\xi)). 
$$
Since $\H_{\mathrm{ac}}(H)=L^2(\R^n)$ in the present case, we have proved the following corollary. 
\begin{corollary}
\label{corollary_1}
Let $n\ge 3$, $p=4/n$ and $\sigma=\pm1$. If $\norm{v_0}_{L^2}$ is small enough, then \eqref{NLS} admits a unique global mild solution $v\in C(\R;L^2(\R^n))\cap L^{2+4/n}(\R^{1+n})$. Moreover, there exist $\tilde v_\pm\in L^2$ such that
$$
\lim_{t\to\pm\infty}\norm{v(t)-e^{it\Delta}\tilde v_\pm}_{L^2}=0\quad\text{if}\quad 1<\mu<2
$$
and that
$$
\lim_{t\to\pm\infty}\norm{v(t)-e^{-iS(t,D)}\tilde v_\pm}_{L^2}=0\quad\text{if}\quad 0<\mu\le1. 
$$
\end{corollary}

%subsection
\subsection{Notation}
Throughout the paper we use the following notation. $\<x\>$ stands for $\sqrt{1+|x|^2}$. $\S(\R^n)$ denotes the Schwartz space. For Banach spaces $X$ and $Y$, $\mathbb B(X,Y)$ denotes the Banach space of bounded linear operators from $X$ to $Y$,  $\mathbb B(X)=\mathbb B(X,X)$ and $\norm{\cdot}_{X\to Y}:=\norm{\cdot}_{\mathbb B(X,Y)}$. We write $L^q=L^q(\R^n)$, $L^p_tL^q_x=L^p(\R;L^q(\R^n))$, $$\norm{\cdot}_q:=\norm{\cdot}_{L^q},\quad \norm{\cdot}_{p\to q}:=\norm{\cdot}_{L^p\to L^q},\quad \norm{\cdot}:=\norm{\cdot}_{2\to2}.$$ $\<f,g\>$ stands for the inner product $\int f\overline g$ on $L^2$. $\H^\gamma(\R^\ell)$ denotes the $L^2$-Sobolev space on $\R^\ell$ of order $\gamma$ with norm $\norm{f}_{\H^\gamma}=\norm{\<D\>^\gamma f}_2$. For $p\in [1,\infty]$, $p'=p/(p-1)$ denotes its H\"older conjugate exponent. We denote the Sobolev critical exponent and its dual by
$$
2^*=\begin{cases}\frac{2n}{n-2},&n\ge3,\\ \infty,&n=2,\end{cases},\quad 2_*=(2^*)'=\frac{2n}{n+2}.
$$  
For positive constants or operators $A,B$, $A\lesssim B$ (resp. $A\gtrsim B$) means that there exists a non-essential constant $C>0$ such that $A\le CB$ (resp, $A\ge CB$). In particular, the symbols $\lesssim$ and $\gtrsim$ will be often used when inequalities are uniform with respect to the frequency parameters $\lambda,h$ or to the temporal parameter $t$. 

\subsection{Outline of the proof}
Here we briefly explain the ideas of the proof of Theorem \ref{theorem_1}. We employ the scheme developed by \cite{BGT,BoTz1,BoTz2,BoMi1} in the study of Strichartz estimates on manifolds. 

For simplicity, we may consider the homogeneous estimate \eqref{theorem_1_1} only. At first, since $V$ is non-negative, an abstract argument shows that the square function estimates for the homogeneous Littlewood-Paley decomposition associated with $H$ hold true. Then the proof can be reduced to proving following energy localized estimate
$$
\norm{\varphi(\lambda^{-2}H)e^{-itH}u_0}_{L^p_tL^q_x}\lesssim \norm{u_0}_{L^2},\quad \lambda>0,
$$
where $\varphi\in C_0^\infty((0,\infty))$ and the implicit constant should be independent of $\lambda$. Since the high energy case $\lambda\ge1$ can be handled similarly, we here focus on the low energy case $\lambda\in (0,1]$. 

We next decompose the energy localized solution $\varphi(\lambda^{-2}H)e^{-itH}u_0$ into two regions $\{\lambda |x|\le 1\}$ and $\{\lambda |x|\ge1\}$.  For the part $\{\lambda |x|\le 1\}$, by virtue of Bernstein's inequality $\norm{\varphi(\lambda^{-2}H)}_{2\to 2^*}\lesssim\lambda$, the desired Strichartz estimate can be deduced from a weighted $L^2$-estimate
$$
\norm{\<x\>^{-1}e^{-itH}u_0}_{L^2_tL^2_x}\lesssim\norm{u_0}_{L^2}
$$
which follows from the uniform resolvent estimate proved by Nakamura \cite{Nak2} under the conditions (H1)--(H3) and Kato's smooth perturbation theory \cite{Kat}. For the non-compact part $\{\lambda |x|\ge1\}$, we approximate the energy localization $\varphi(\lambda^{-2}H)$ by a suitable rescaled pseudodifferential operator $\D\Op(a^\lambda)^*\D^*$ modulo an error term, where $\D f(x)=\lambda^{n/2}f(\lambda x)$ is the usual dilation and $a^\lambda(x,\xi)$ is a smooth bounded symbol supported in the region $\{|x|\ge1,\ |\xi|^2+V^\lambda(x)\sim1\}$ and $V^\lambda(x)=\lambda^{-2}V(\lambda^{-1}x)$. Then the error term can be handled by a similar argument as that for the compact part. By rescaling, the main term $\D\Op(a^\lambda)^*\D^*e^{-itH}$ can be written in the form
\begin{align}
\label{outline_1}
\D\Op(a^\lambda)^*e^{-it\lambda^2 H^\lambda }\D^*
\end{align}
where $H^\lambda=-\Delta+V^\lambda(x)$. Here note that $V^\lambda(x)$ only satisfies 
$$
|V^\lambda(x)|\lesssim \lambda^{-2}(1+\lambda^{-1}|x|)^{-\mu}=\lambda^{-2+\mu}(\lambda+|x|)^{-\mu}
$$ and, thus, may blow up as $\lambda\to0$. Hence it seems to be difficult to handle the operator \eqref{outline_1} uniformly in small $\lambda$. To overcome this difficulty, we introduce another small parameter 
$$
h=\lambda^{2/\mu-1}.
$$
By a scaling argument with respect to $h$, the problem then is reduced to treat the operator
\begin{align}
\label{outline_2}
\Op_h(a_h)^*e^{-itH_h/h},
\end{align}
where $H_h=-h^2\Delta+V_h(x)$, $V_h(x)=\lambda^{-2}V(\lambda^{-2/\mu}x)$, and $\Op_h(a_h)$ is a semiclassical pseudodifferential operator with the symbol $a_h$ supported in 
$\{|x|\gtrsim 1,\ |\xi|^2+V_h(x)\sim1\}$. The main advantage to introduce the parameter $h$ is that, 
since $V_h$ obeys $$|V_h(x)|\lesssim (\lambda^{2/\mu}+|x|)^{-\mu}\lesssim \<x\>^{-\mu}$$ for $|x|\gtrsim1$ uniformly in $h$, we can use the semiclassical analysis to handle the operator \eqref{outline_2} uniformly in $h$. Then, we further decompose $\Op_h(a_h)^*e^{-itH_h}$ into two regions $\{|x|\sim1\}$ and $\{|x|\gg1\}$. For the former part, we employ a similar idea as in Staffilani-Tataru \cite{StTa} which yields that the desired Strichartz estimate can be deduced from the semiclassical WKB parametrix construction for sufficiently small $t$, and the local smoothing estimate of the form
$$
\norm{\<x\>^{-1}\varphi(H_h)e^{-itH_h/h}u_0}_{L^2_tL^2}\lesssim \norm{u_0}_{L^2}
$$
which can be obtained by using a semiclassical version of Mourre's theory. It follows from this step that we may assume that $a_h$ is supported in $\{|x|\gg1, |\xi|\sim1\}$. We then decompose $a_h$ into the outgoing part $a_h^+$ and incoming part $a_h^-$. Thanks to the abstract $TT^*$-argument by Keel-Tao \cite{KeTa}, it suffices to show the following dispersive estimate
$$
\norm{\Op_h(a_h^\pm)^*\varphi(H_h)e^{-itH_h/h}\Op_h(a_h^\pm)}_{1\to \infty}\lesssim |th|^{-n/2},\quad t\neq0,\ h\in (0,1].
$$
To this end, we essentially follow the idea of Bouclet-Tzvetkov \cite{BoTz2}. The main ingredient for the proof of this dispersive estimate is the construction of the semiclassical Isozaki-Kitada (IK) parametrix of $e^{-itH_h/h}\Op_h(a_h^+)$ whose main term is of the form
$$
J_h^+(c^+)e^{ith\Delta}J_h^+(d^+)^*,
$$
where $J_h^+(w)$, which is called the IK modifier, is a semiclassical Fourier integral operator with a time-independent phase function $S^+(x,\xi)=x\cdot\xi+O(\<x\>^{1-\mu})$. The dispersive estimate for the main part of the IK parametrix is a simple consequence of the standard stationary phase theorem, while several propagation estimates will be used in order to deal with the error term. To prove such propagation estimates, we employ the local decay estimate for the propagator which can be obtained by means of the semiclassical version studied by Nakamura \cite{Nak1} of the multiple commutator method by Jensen-Mourre-Perry \cite{JMP}. 

The paper is organized as follows. We collect several preliminary materials in Section 2 which include a brief review on the semiclassical pseudodifferential operator calculus, the dilation and rescaled Hamiltonians, basic properties of the spectral multiplier $\varphi(\lambda^{-2}H)$ and an associated Littlewood-Paley decomposition, the approximations of the spectral multiplier in terms of rescaled pseudodifferential operators, and local smoothing and local decay estimates for the propagator. The proofs of homogeneous estimates \eqref{theorem_1_1} and inhomogeneous estimates \eqref{theorem_1_2} of Theorem \ref{theorem_1} are given in Sections 3 and 4, respectively. Section 5 concerns with the proof of Theorem \ref{theorem_2} which is based on Theorem \ref{theorem_1} and Rodnianski-Schlag's method. Appendix A is devoted to  the proof of a  uniform weighted bound for the power of semiclassical resolvent which will be used in the proof of local decay estimates. \\\\
\noindent{\bf Acknowledgements.} The author is partially supported by JSPS KAKENHI Grant-in-Aid for Young Scientists (B) \#JP17K14218 and Grant-in-Aid for Scientific Research (B) \#JP17H02854.

%section
\section{Preliminaries}
We start with collecting several preliminary results which will be used in the sequel. 

%subsection
\subsection{Semiclassical pseudodifferential calculus}
Let $S^{\mu,m}=S(\<x\>^\mu\<\xi\>^m,\<x\>^{-2}dx^2+\<\xi\>^{-2}d\xi^2)$ be a family of symbols $a\in C^\infty(\R^{2n})$ satisfying
\begin{align}
\label{PDO_0}
|\partial_x^\alpha\partial_\xi^\beta a(x,\xi)|\le C_{\alpha\beta}\<x\>^{\mu-|\alpha|}\<\xi\>^{m-|\beta|},\quad x,\xi\in\R^n. 
\end{align}
As usual, we set $S^{\mu,-\infty}=\bigcap_{m\ge0}S^{\mu,-m}$ and $
S^{-\infty,m}:=\bigcap_{\mu\ge0}S^{-\mu,m}$. Henceforth, for a given $h$-dependent symbol $a_h\in C^\infty(\R^{2n})$ with a small parameter $h\in (0,1]$, we also say that $a_h\in S^{\mu,m}$ if $\{a_h\}_{h\in (0,1]}$ is bounded in $S^{\mu,m}$, namely $C_{\alpha\beta}$ in \eqref{PDO_0} may be taken uniformly in $h\in (0,1]$.

For $a\in S^{\mu,m}$, the semiclassical pseudodifferential operator ($h$-PDO) $\Op_h(a)$ is defined by
$$
\Op_h(a)f(x)=a(x,hD)f(x):=\frac{1}{(2\pi h)^n}\int e^{i(x-y)\cdot\xi/h}a(x,\xi)f(y)dyd\xi.
$$
We also use the notation $\Op(a):=\Op_1(a)$ when $h=1$. $\Op_h(a)$ maps from $\S(\R^n)$ to $\S'(\R^n)$ for any $\mu,m\in \R$. If $a\in S^{0,0}$ then Calder\'on-Vaillancourt's theorem shows that $$\sup_{h\in (0,1]}\norm{\Op_h(a)}\lesssim \sum_{|\alpha+\beta|\le N}\norm{\partial_x^\alpha\partial_\xi^\beta a}_{L^\infty(\R^{2n})}$$ with some $N\in \N$ depending only on $n$. If $a\in S^{\mu,m}$ with $\mu\le0$ and $m<-n$, then $\Op_h(a)$ extends to a bounded operator from $L^p$ to $L^q$ for any $1\le p\le q\le\infty$, satisfying
\begin{align}
\label{PDO_1}
\sup_{h\in (0,1]}h^{n(\frac1p-\frac1q)}\norm{\Op_h(a)}_{p\to q}\le C_{npq}<\infty
\end{align}
(see \cite[Proposition 2.4]{BoTz1}). The following symbolic calculus is also well known (see \cite{Hor}).

%proposition
\begin{proposition}
\label{proposition_PDO_1}
Let $\mu,\mu',m,m',\in\R$, $a\in S^{\mu,m}$ and $b\in S^{\mu',m'}$. Then there exists $a\# b\in S^{\mu+\mu',m+m'}$ such that $\Op_h(a)\Op_h(b)=\Op_h(a\# b)$ and
\begin{align}
\label{proposition_PDO_1_1}
a\# b=\sum_{|\alpha|<N}\frac{h^{|\alpha|}}{\alpha!}\partial_\xi^\alpha a D_x^\alpha b +h^N r_N,\quad
r_N\in S^{\mu+\mu'-N,m+m'-N}.
\end{align}
Moreover, there exists $a^*\in S^{\mu,m}$ such that $\Op_h(a)^*=\Op_h(a^*)$ and
\begin{align}
\label{proposition_PDO_1_2}
a^*=\sum_{|\alpha|<N}\frac{h^{|\alpha|}}{\alpha!}\partial_\xi^\alpha D_x^\alpha \overline a +h^N r_N^*,\quad
r_N^*\in S^{\mu-N,m-N},
\end{align}
where $\Op_h(a)^*$ is the formal adjoint of $\Op_h(a)$. 
\end{proposition}

%remark
\begin{remark}
\label{remark_PDO_1}
The remainder $r_N$ in the expansion \eqref{proposition_PDO_1_1} can be written explicitly as
$$
r_N(x,\xi)=\frac{N}{(2\pi h)^n}\sum_{|\alpha|=N}\frac{1}{\alpha!}\iiint_0^1e^{-y\cdot \eta/h}\frac{\eta^\alpha}{h^\alpha}(1-t)^{N-1} (\partial_\xi^\alpha a)(x,\xi+t\eta)b(x+y,\xi)dtdyd\eta. 
$$
In particular, if $\partial_\xi^\alpha a\equiv0$ on $\R^{2n}$ for all $|\alpha|\ge N$ then $r_N\equiv0$. This is the case when $N=3$ and $a=|\xi|^2+V(x)$, the symbol of $H$. This fact will be used in the proof of Proposition \ref{proposition_functional_2} below. 
\end{remark}

%subsection
\subsection{Dilation and rescaled Hamiltonians}
\label{subsection_dilation}
The dilation group $\D(\theta)$ defined by $\D(\theta) f(x):=\theta^{n/2}f(\theta x)$ for $\theta>0$  plays a crucial role throughout the paper. $\D(\theta)$ is unitary on $L^2(\R^n)$ and its dual is given by $\D(\theta)^*=\D(\theta^{-1})$. Moreover, $\D(\theta)$ is bounded on $L^q$ satisfying
\begin{align}
\label{dilation_1}
\norm{\D(\theta) f}_{q}=\theta^{n(\frac12-\frac1q)}\norm{f}_q,\quad 1\le q\le\infty.
\end{align}
For a Borel measurable function $\varphi$ on $\R^n$, the actions of $\D(\theta)$ on $\varphi(x)$ and $\varphi(D)$ are given by 
\begin{align}
\label{dilation_2}
\varphi(x)=\D(\theta)\varphi(\theta^{-1}x)\D(\theta)^*,\ 
\varphi(D)=\D(\theta)\varphi(\theta D)\D(\theta)^*,
\end{align}
respectively, where $D=-i\nabla$ and $\varphi(D):=\F^{-1}\varphi(\xi)\F$ is the Fourier multiplier. Given a parameter $\lambda>0$, we will use dilations with two different scaling parameters $$\D:=\D(\lambda),\quad \D_\mu:=\D(\lambda^{2/\mu}).$$ The actions of $\D $ and $\D_\mu $ on $H$ are given by 
\begin{align}
\label{dilation_3_0}
H=\lambda^2\D  H^\lambda \D ^*=\lambda^2\D_\mu  H_{h_\lambda}\D_\mu ^*,\quad 
\end{align}
where the rescaled Hamiltonians $H^\lambda$ and $H_{h_\lambda}$ are defined by 
\begin{align*}
&H^\lambda:=-\Delta +V^\lambda(x),\quad \ \ V^\lambda(x):=\lambda^{-2}V(\lambda^{-1}x),\\
&H_{h_\lambda}:=-{h_\lambda^2}\Delta +V_{h_\lambda}(x),\quad V_{h_\lambda}(x):=\lambda^{-2}V(\lambda^{-2/\mu}x)=h_\lambda^{-\frac{2\mu}{2-\mu}}V(h_\lambda^{-\frac{2}{2-\mu}}x)
\end{align*}
with $h_\lambda=\lambda^{2/\mu-1}$. In what follows we drop the subscript $\lambda$ and use the notation 
$$
h=\lambda^{2/\mu-1}
$$
everywhere. It follows from \eqref{dilation_3_0} that the actions of $\D $ and $\D_\mu $ on $\psi(H)$ are given by
\begin{align}
\label{dilation_3}
\psi (\lambda^{-2}H)=\D \psi (H^\lambda)\D ^*=\D_\mu \psi (H_h )\D_\mu ^*. 
\end{align}
Here, for a Borel measurable function $\psi$ on $\R$, the spectral multiplier $\psi(H)$ is defined by 
$$
\psi(H)=\int\psi(\rho)\, dE_H(\rho), $$
where $E_H$ is the spectral measure  associated with $H$. 

%{lemma}
\begin{lemma}
\label{lemma_dilation_1}
Assume {\rm (H1)--(H3)}. Then $V^\lambda$ and $V_h$ satisfy the following estimates:
\begin{align}
\label{lemma_dilation_1_1}
|\partial_x^\alpha V^\lambda(x)|&\le C_\alpha\lambda^{-2+\mu}(\lambda+|x|)^{-\mu-|\alpha|},\quad x\in \R^n,\\
\label{lemma_dilation_1_2}
V^\lambda(x)&\ge C_1\lambda^{-2+\mu}(\lambda+|x|)^{-\mu},\quad x\in \R^n,\\
\label{lemma_dilation_1_3}
x\cdot (\nabla V^\lambda)(x)&\le -C_2\lambda^{-2+\mu}(\lambda+|x|)^{-\mu},\quad |x|\ge \lambda R_0,\\
\label{lemma_dilation_1_4}
|\partial_x^\alpha V_h(x)|&\le C_\alpha (\lambda^{2/\mu}+|x|)^{-\mu-|\alpha|},\quad x\in \R^n,\\
\label{lemma_dilation_1_5}
V_h(x)&\ge C_1(\lambda^{2/\mu}+|x|)^{-\mu},\quad x\in \R^n,\\
\label{lemma_dilation_1_6}
x\cdot (\nabla V_h)(x)&\le -C_2(\lambda^{2/\mu}+|x|)^{-\mu},\quad |x|\ge \lambda^{2/\mu}R_0.
\end{align}
\end{lemma}

%proof
\begin{proof}
The lemma follows immediately from (H1)--(H3). 
\end{proof}

%remark
\begin{remark}
By virtue of this lemma, the action of $\D $ is well adapted to an analysis in the high energy regime $\lambda\ge1$, while the action of $\D_\mu $ is well adapted  to an analysis in the low energy regime $0<\lambda\le1$. Moreover $H_h $ can be regarded as a semiclassical Hamiltonian with the semiclassical parameter $h=\lambda^{2/\mu-1}$ if $\lambda\le1$. 
\end{remark}

\subsection{Spectral multiplier and Littlewood-Paley decomposition}
\label{subsection_multiplier}
This subsection concerns with  mapping properties of the spectral multiplier and square function estimates for the Littlewood-Paley decomposition associated with $H$. We also show here that Theorem \ref{theorem_1} follows from corresponding energy localized estimates (see Theorem \ref{theorem_LP_2}). Throughout this subsection, we assume that $V\ge0$ and $V\in L^1_{\loc}$ so that $H=-\Delta+V$ is defined as the unique self-adjoint operator associated with the non-negative quadratic form $\<(-\Delta+V) u,u\>$ on $C_0^\infty(\R^n)$. We begin with stating H\"ormander's type multiplier theorem. 

%{lemma}
\begin{lemma}
\label{lemma_LP_1} 
Let $m\in L^\infty(\R)$ satisfy, with some $\gamma>(n+1)/2$ and non-trivial $\varphi\in C_0^\infty((0,\infty))$, 
$$
|m|_\gamma:=\sup_{t>0}\norm{\varphi(\cdot)m(t\cdot)}_{\H^\gamma(\R)}<\infty.
$$ 
Then all of $m(\lambda^{-2}H),m(H^\lambda)$ and $m(H_h)$ are bounded on $L^p$ for any $1<p<\infty$ and $\lambda>0$. Moreover, their operator norms on $L^p$ are uniformly bounded with respect to $\lambda>0$.
\end{lemma}

%proof
\begin{proof}
Since $V\ge0$, the kernel $e^{-tH}(x,y)$ of $e^{-tH}$ satisfies the upper Gaussian bound (see \cite{Sim})
$$
0\le e^{-tH}(x,y)\le e^{t\Delta}(x,y)=(4\pi t)^{-\frac n2}e^{-\frac{|x-y|^2}{4t}},\quad t>0,\ x,y\in\R^n,
$$
which, together with the formula $e^{-tH^\lambda }=\D ^*e^{-\lambda^{-2}tH}\D$, implies
\begin{equation}
\begin{aligned}
\label{proof_lemma_LP_1_1}
e^{-tH^\lambda }(x,y)&=\lambda^{-n}e^{-\lambda^{-2}tH}(\lambda^{-1}x,\lambda^{-1}y)
\le (4\pi t)^{-\frac n2}e^{-\frac{|x-y|^2}{4t}}.
\end{aligned}
\end{equation}
Then an abstract theorem \cite[Theorem 1.2]{CaDa} applies to $H$, yielding 
$$
\norm{m(H^\lambda )f}_{p}\lesssim (p+(p-1)^{-1})(1+|m|_\gamma+\norm{m}_{L^\infty}^2)\norm{f}_{L^p},\quad f\in \S(\R^n),
$$
for all $1<p<\infty$ uniformly in  $\lambda>0$. This estimate, combined with the fact $$
\norm{m(\lambda^{-2}H)}_{p\to p}=\norm{m(H^\lambda )}_{p\to p}=\norm{m(H_h)}_{p\to p}
$$
which follows from \eqref{dilation_1} and \eqref{dilation_3}, implies the desired assertion.
\end{proof}

\begin{lemma}\label{lemma_LP_2}Let  $m\in \mathcal S(\R)$, $1\le p\le q\le\infty$ and $(p,q)\neq(1,1),(\infty,\infty)$. Then\begin{align}\label{lemma_LP_2_1}\norm{m(\lambda^{-2}H)}_{p\to q}\le C_{pq}\lambda^{n(\frac1p-\frac1q)},\quad \lambda>0.\end{align}\end{lemma}

%proof
\begin{proof} Since $\norm{m(\lambda^{-2}H)}_{p\to q}=\lambda^{n(1/p-1/q)}\norm{m(H^\lambda )}_{p\to q}$ by  \eqref{dilation_1} and \eqref{dilation_3},  it suffices to show $m(H^\lambda )\in \mathbb B(L^p,L^q)$ with uniform bounds in $\lambda$. We also may assume $q=\infty$ since other cases follow from the duality argument, Riesz-Thorin's theorem and Lemma \ref{lemma_LP_1}. Taking a large constant $M$ specified later, we write 
$m(H^\lambda )=(H^\lambda +1)^{-M}\wtilde m(H^\lambda )(H^\lambda +1)^{-M}$,  
where $\wtilde m(t)=(t+1)^{2M}m(t)$. Since $\wtilde m$ satisfies the condition in Lemma \ref{lemma_LP_1}, $\wtilde m(H^\lambda )$ is bounded on $L^r$ uniformly in $\lambda$ for any $1<r<\infty$. Let $r>p$. It remains to show $(H^\lambda +1)^{-M}\in \mathbb B(L^p,L^r)\cap \mathbb B(L^r,L^\infty)$ with uniform bounds in $\lambda$. To this end, we use the well-known formula
$$(H^\lambda +1)^{-M}=\frac{1}{\Gamma(M)}\int_0^\infty t^{M-1}e^{-t}e^{-tH^\lambda }dt$$
which, together with the following decay estimates of the semi-group
$$
\norm{e^{-tH^\lambda }}_{r_1\to r_2}\lesssim t^{-\frac n2(\frac{1}{r_1}-\frac{1}{r_2})},\quad t>0,\quad 1\le r_1\le r_2\le \infty,\ (r_1,r_2)\neq (1,1),(\infty,\infty),
$$
obtained from \eqref{proof_lemma_LP_1_1} and Lemma \ref{lemma_LP_1}, shows that $(H^\lambda +1)^{-M}\in \mathbb B(L^p,L^r)\cap \mathbb B(L^r,L^\infty)$ with operator norms being independent of $\lambda$, provided that $M>n+1$. 
\end{proof}

%remark
\begin{remark}Let $V_\pm\ge0$ be such that $V=V_+-V_-$. 
It was proved by \cite{JeNa} that if $V_+\in \mathcal K_n^{\loc}$, $V_-\in \mathcal K_n$ and $m^{(k)}(t)=O(\<t\>^{-\ep-k})$ with some $\ep>0$ and all $k\ge0$, then $m(H)\in \mathbb B(L^p)$ for any $1\le p\le\infty$. Here $\mathcal K_n$ and $\mathcal K_n^{\loc}$ are Kato and local Kato classes, respectively \cite{Sim}. We do not know whether $m(H)$ is bounded on $L^p$ at $p=1$ or $\infty$ under the condition $0\le V\in L^{1}_{\loc}$ only, even if $m\in \S(\R)$. However, the above lemma is sufficient for the purpose of this paper. 
\end{remark}

%remark
\begin{remark}
In contrast to Sobolev's inequality $\norm{f}_{{2^*}}\le C\norm{\nabla f}_{2}$ which fails if $n=2$, \eqref{lemma_LP_2_1} with $(p,q)=(2,2^*)$ holds for all $n\ge 2$. This fact will play a crucial role in case of $n=2$. 
\end{remark}

We next recall the square function estimates for the Littlewood-Paley decomposition associated with $H$. Consider a homogeneous dyadic partition of unity on $(0,\infty)$: 

%proposition
\begin{align}
\label{proposition_LP_3_1}
\sum\limits_{j\in\Z}f(2^{-j}s)=1,\quad s>0,
\end{align}
where $f\in C_0^\infty(\R)$, $\supp f\subset [1/2,2]$, $f(s)=1$ for $3/4<s<3/2$ and $0\le f\le 1$. 
\begin{proposition}
\label{proposition_LP_3}Suppose that $H=-\Delta+V$ has no eigenvalue and $1<q<\infty$. Then
\begin{align}
\label{proposition_LP_3_2}
C_q^{-1}\norm{v}_{q}\le \Big\|\Big(\sum_{j\in\Z}|f(2^{-j}H)v|^2\Big)^{1/2}\Big\|_{q}\le C_q\norm{v}_{q}.
\end{align}
with some $C_q>1$. In particular, we have
\begin{align}
\label{proposition_LP_3_3}
&\norm{v}_{q}\le C_q\Big(\sum_{j\in\Z}\norm{f(2^{-j}H)v}_{q}^2\Big)^{1/2}\quad\text{if}\quad 2\le q<\infty,\\
\label{proposition_LP_3_4}
&\Big(\sum_{j\in\Z}\norm{f(2^{-j}H)v}_{q}^2\Big)^{1/2}\le C_q\norm{v}_{q}\quad\text{if}\quad 1<q\le2.
\end{align}
\end{proposition}

%proof
\begin{proof}
The proof follows essentially the same argument as that in \cite[Section 0.2]{Sog}. In order to prove \eqref{proposition_LP_3_2} we consider the following square function associated with $H$:
$$
Sv(x)=\Big(\sum_{j\in \Z}|f(2^{-j}H)v(x)|^2\Big)^{1/2},\quad v\in \S(\R^n).
$$
Let $u,v\in \S(\R^n)$. Since $H$ has no eigenvalue and hence $\mathds1_{\{0\}}(H)=0$, it follows from \eqref{proposition_LP_3_1} that $$\<u,v\>=\sum_{j,k\in\Z}\<f(2^{-j}H)u,f(2^{-k}H)v\>.$$ Moreover, the support property 
\begin{align}
\label{proposition_LP_3_proof_1}
f(2^{-j}s)f(2^{-k}s)=0,\quad |j-k|>2,\quad s>0,
\end{align}
implies that $f(2^{-j}H)$ satisfies the following almost orthogonality: 
$$
f(2^{-j}H)f(2^{-k}H)=0,\quad |j-k|>2.
$$
Hence H\"older's inequality implies
\begin{align*}
|\<u,v\>|\le \sum_{|j-k|\le 2}|\<f(2^{-j}H)u,f(2^{-k}H)v\>|\lesssim\norm{Su}_{q}\norm{Sv}_{{q'}}.
\end{align*}
By duality, it remains to show that $S$ is bounded on $L^q$. To this end, we set $$S_\pm v(x)=\Big(\sum_{\pm j\ge0}|f(2^{-j}H)v(x)|^2\Big)^{1/2}$$ and may prove the $L^q$-boundedness of $S_+$, the proof for $S_-$ being analogous. Let $r_j(t)$, $j=0,1,...$, be Rademacher functions on $[0,1]$ and consider the spectral multiplier $m_t(H)$, where
$$
m_t(s)=\sum_{j=0}^\infty r_j(t)f(2^{-j}s).
$$
Using \eqref{proposition_LP_3_proof_1}, we see that for any $k=0,1,2,...$ there exists $C_k$ independent of $t\in [0,1]$ such that
$$
|\partial_s^k m_t(s)|\le C_k s^{-k},\quad s>0.
$$
In particular, $m_t$ satisfies the conditions in Lemma \ref{lemma_LP_1} and hence $\norm{m_t(H)}_{q\to q}\lesssim1$ uniformly in $t\in [0,1]$ by Lemma \ref{lemma_LP_1} with $\lambda=1$. Combining with Khintchine's inequality, we have 
$$
\norm{S_+v}_q\lesssim \bignorm{\norm{m_t(H)v}_{L^q([0,1]_t)}}_{L^q(\R^n_x)}\lesssim \norm{m_t(H)v}_{L^q([0,1]_t;L^q(\R^n_x))}\lesssim \norm{v}_q.
$$
This completes the proof of \eqref{proposition_LP_3_2}. \eqref{proposition_LP_3_3} and \eqref{proposition_LP_3_4} follow immediately  from \eqref{proposition_LP_3_2}.
\end{proof}

We conclude this subsection by observing that, by virtue of Proposition \ref{proposition_LP_3}, Theorem \ref{theorem_1} follows from the following energy localized version. 

%corollary
\begin{theorem}
\label{theorem_LP_2}
Assume that $n\ge2$ and $V$ satisfies {\rm (H1)--(H3)}. Let $(p,q)$ and $(\tilde p,\tilde q)$ satisfy \eqref{admissible}. Then, for any $\varphi\in C_0^\infty(\R)$ with $\supp \varphi \Subset (0,\infty)$ one has, uniformly in $\lambda>0$, 
\begin{align}
\label{theorem_LP_2_1}
\norm{\varphi(\lambda^{-2}H)e^{-itH}u_0}_{L^p(\R;L^q(\R^n))}&\lesssim \norm{u_0}_{L^2(\R^n)},\\
\label{theorem_LP_2_2}
\bignorm{\int_0^t \varphi(\lambda^{-2}H)e^{-i(t-s)H}F(s)ds}_{L^p(\R;L^q(\R^n))}&\lesssim \norm{F}_{L^{\tilde p'}(\R;L^{\tilde q'}(\R^n))}. 
\end{align}
\end{theorem}

%proof
\begin{proof}[Proof of Theorem \ref{theorem_1}, assuming Theorem \ref{theorem_LP_2}]
Let $f$ be as that in Proposition \ref{proposition_LP_3} and $\varphi\in C_0^\infty((0,\infty))$ so that $\varphi\equiv1$ on $\supp f$. Consider the inhomogeneous estimate \eqref{theorem_1_2}. We may assume without loss of generality that $n\ge3$ and $(p,q)=(\tilde p,\tilde q)=(2,2^*)$ since other cases follow from complex interpolation or the homogeneous estimates \eqref{theorem_1_1} and Christ-Kiselev's lemma \cite{ChKi}. Since $2_*<2<2^*$, \eqref{proposition_LP_3_3}, \eqref{proposition_LP_3_4}  and \eqref{theorem_LP_2_2} with $\lambda=2^{j/2}$ imply
\begin{align*}
\bignorm{\int_0^t e^{-i(t-s)H}F(s)ds}_{L^2_tL^{2^*}_x  }^2
&\lesssim \sum_{j\in \Z}\bignorm{\int_0^t \varphi(2^{-j}H) e^{-i(t-s)H} f(2^{-j}H)F(s)ds}_{L^2_tL^{2^*}_x  }^2\\
&\lesssim \sum_{j\in \Z}\norm{f(2^{-j}H)F}_{L^2_tL^{2_*}_x  }^2\\&\lesssim \norm{F}_{L^2_tL^{2_*}_x  }^2.
\end{align*}
The homogeneous estimate \eqref{theorem_1_1} is verified similarly. 
\end{proof}

%subsection
\subsection{Functional calculus}
\label{subsection_functional}
Let $\varphi\in C_0^\infty(\R)$ and $\chi\in C_0^\infty(\R^n)$ be such that $0\le \chi\le1$, $\chi(x)=1$ for $|x|<1$, $\chi(x)=0$ for $|x|>2$ and set $\chi_R(x)=\chi(x/R)$ with $R>0$. This subsection is devoted to the construction of approximations of the operators $(1-\chi_R(x))\varphi(H^\lambda )$ and  $(1-\chi_R(x))\varphi(H_h )$ in terms of suitable PDOs. We begin with rough weighted bounds of $(H^\lambda -z)^{-1}$ and $\varphi(H^\lambda )$. 

%{lemma}
\begin{lemma}
\label{lemma_functional_1}
For any $\alpha\in \R$ there exists $C_\alpha>0$ such that, for any $\lambda>0$ and $z\in \C\setminus\R$, 
\begin{align}
\label{lemma_functional_1_1}
\norm{\<x\>^{\alpha}(H^\lambda -z)^{-1}\<x\>^{-\alpha}}\le C_\alpha\<z\>^{|\alpha|/2}|\Im z|^{-|\alpha|-1}.
\end{align}
Moreover, $\<x\>^\alpha \varphi(H^\lambda )\<x\>^{-\alpha}$ is bounded on $L^2$ uniformly in $\lambda>0$.
\end{lemma}

%proof
\begin{proof}
By Stein's complex interpolation \cite{Ste}, we may assume that $\alpha=N\in \N\cup\{0\}$ and prove \eqref{lemma_functional_1_1} by induction in $N$. The case when $N=0$ is trivial. Next, a direct calculation yields
\begin{align*}
\<x\>^N(H^\lambda-z)^{-1}\<x\>^{-N}
=(H^\lambda-z)^{-1}+(H^\lambda-z)^{-1}[-\Delta,\<x\>^N](H^\lambda-z)^{-1}\<x\>^{-N},
\end{align*}
where the second term of the right hand side is written in the form
\begin{align*}
(H^\lambda -z)^{-1}\<H^\lambda \>^{1/2}\cdot \<H^\lambda \>^{-1/2}\<D\>\cdot \<D\>^{-1}[-\Delta,\<x\>^N]\<x\>^{-N+1}\cdot\<x\>^{N-1}(H^\lambda -z)^{-1}\<x\>^{-N}.
\end{align*}
By the spectral theorem, $\norm{(H^\lambda -z)^{-1}\<H^\lambda \>^{1/2}}\le C\<z\>^{1/2}|\Im z|^{-1}$. Since $H^\lambda \ge-\Delta$, we also obtain $\norm{\<H^\lambda \>^{-1/2}\<D\>}\le 1$. Moreover, since $\<\xi\>^{-1}\{|\xi|^2,\<x\>^N\}\<x\>^{-N+1}\in S^{0,0}$, Calder\'on-Vaillancourt's theorem shows $\<D\>^{-1}[-\Delta,\<x\>^N]\<x\>^{-N+1}\in  \mathbb B(L^2)$, where $\{f,g\}=\nabla_xf\cdot\nabla_\xi g-\nabla_\xi f\cdot\nabla_x g$ is the Poisson bracket. Hence
\begin{align*}
\norm{(H^\lambda -z)^{-1}[-\Delta,\<x\>^N](H^\lambda -z)^{-1}\<x\>^{-N}}
\le C_N\<z\>^{N/2}|\Im z|^{-N-1}
\end{align*}
for all $N$ by the hypothesis of induction. This completes the proof of \eqref{lemma_functional_1_1}. The assertion for $\<x\>^N\varphi(H^\lambda )\<x\>^{-N}$ is a consequence of \eqref{lemma_functional_1_1} and Helffer-Sj\"ostrand's formula
\begin{align}
\label{HS}
\varphi(H^\lambda )=-\frac{1}{\pi }\int_\C \frac{\partial \Phi}{\partial \overline z}(z)(H^\lambda -z)^{-1}dL(z), 
\end{align}
where $\Phi(z)$ is an almost analytic extension of $\varphi$ such that $\Phi$ is independent of $\lambda$, $\Phi|_{\R}=\varphi$, $\supp \Phi$ is compact, and $|\partial_{\overline z}\Phi(z)|\le C_N|\Im z|^N$ for any $N\ge0$ (see \cite{HeSj,Hor}). \end{proof}

%proposition
\begin{proposition}
\label{proposition_functional_2}
Let $N\in \N\cup\{0\}$ and $R>0$. Then there exist a bounded set $\{a^\lambda\}_{\lambda>0}\subset S^{0,-\infty}$ satisfying $\supp a^{\lambda}\subset \supp[(1-\chi_R)\varphi\circ p^\lambda]$ and  $Q^\lambda\in \mathbb B(L^2)$ such that
\begin{align}
\label{proposition_functional_2_1}
(1-\chi_R)(x)\varphi(H^\lambda )=\Op(a^\lambda)^*+Q^\lambda,\quad \sup_{\lambda>0}\norm{\<D\>^{N}\<x\>^{N}Q^\lambda\<x\>^{N}}\le C_N<\infty,
\end{align}
where $p^\lambda(x,\xi)=|\xi|^2+V^\lambda(x)$ is the symbol of $H^\lambda$. 
\end{proposition}

%proof
\begin{proof}
The proof is based on a standard argument using Helffer-Sj\"ostrand's formula and a microlocal parametrix of the resolvent (see, e.g., \cite{DiSj} or \cite[Proposition 2.1]{BGT}). However, these previous literatures cannot be applied directly to the present case since $V^\lambda$ may not uniformly bounded as $\lambda\to0$, so we give a complete proof. We shall show that 
$$
\varphi(H^\lambda )(1-\chi_R)(x)=\Op(a^\lambda)+\wtilde Q^\lambda,\quad \norm{\<x\>^NQ^\lambda\<x\>^N\<D\>^N}\lesssim1. 
$$ 
Taking $d_0>0$ with $\supp \Phi\subset \{|z|\le d_0\}$, we construct the parametrix of $(H^\lambda -z)^{-1}(1-\chi_R)$ for $|z|\le d_0$. By \eqref{proposition_PDO_1_1} and Remark \ref{remark_PDO_1}, for any symbol $q$, the symbol of $(H^\lambda -z)\Op(q)$ is given by
$$
(p^\lambda-z)\# q|_{h=1}=(p^\lambda-z)q-2i\xi\cdot\nabla_x q-\Delta_x q.
$$
With this expansion at hand, we define $q_k^\lambda=q_k^\lambda(z,x,\xi)$ inductively by
\begin{align}
\label{proof_proposition_functional_2_1}
q_0^\lambda=\frac{1-\chi_R}{p^\lambda-z},\quad
q_1^\lambda=\frac{2i\xi\cdot\nabla_xq_0^\lambda}{p^\lambda-z},\quad
 q_k^\lambda=\frac{2i\xi\cdot\nabla_xq_{k-1}^\lambda+\Delta_xq^\lambda_{k-2}}{p^\lambda-z},\quad k\ge2. 
\end{align}
We shall show that $q_k^\lambda$ satisfy
\begin{align}
\label{proof_proposition_functional_2_2}
|\partial_x^\alpha\partial_\xi^\beta q_k^\lambda(z,x,\xi)|\le C_{k\alpha\beta}\<x\>^{-k-|\alpha|}\<\xi\>^{-2-k-|\beta|}|\Im z|^{-1-2k-|\alpha|-|\beta|},\quad x,\xi\in \R^n,
\end{align}
uniformly in $\lambda>0$. To this end we observe from Lemma \ref{lemma_dilation_1} that
\begin{align}
\label{proof_proposition_functional_2_3}
V^\lambda(x)\ge C_1 \lambda^{-2+\mu}(\lambda+|x|)^{-\mu},\quad  
|\partial_x^\alpha V^\lambda (x)|\le C_\alpha\<x\>^{-|\alpha|}V^\lambda (x) 
\end{align}
uniformly in $\lambda>0$ and  $|x|\ge R$. Also note that, by Leibniz's rule, $\partial_x^\alpha\partial_\xi^\beta q_0^\lambda$ is of the form
\begin{align}
\label{proof_proposition_functional_2_4}
\partial_x^\alpha\partial_\xi^\beta q_0^\lambda=\sum_{j=0}^{|\alpha|}\sum_{\ell=0}^{|\beta|}\frac{d_{j\ell}^\lambda (x,\xi)}{(p^\lambda(x,\xi)-z)^{1+j+\ell}},
\end{align}
where $d_{j\ell}^\lambda$ is independent of $z$, a polynomial in $\xi$ of degree at most $\ell$, supported in $\supp(1-\chi_R)$ in the $x$-variable. Moreover, it follows from an induction argument in $j$ that $d_{j\ell}^\lambda$ is of the form
$$
d_{j\ell}^\lambda(x,\xi)=Q_\ell(\xi)\sum\partial_x^\gamma\left((1-\chi_R(x))\partial_{x}^{\alpha_1}V^\lambda(x)\cdots \partial_{x}^{\alpha_j}V^\lambda(x)\right)
$$
where $|\gamma|=|\alpha|-j$, $|\alpha_1|=\cdots=|\alpha_j|=1$, $Q_\ell(\xi)$ is a polynomial in $\xi$ of degree at most $\ell$ and independent of $x,z$ and $\lambda$. Therefore, we know by \eqref{proof_proposition_functional_2_4} and Leibniz's rule that
\begin{align}
\label{proof_proposition_functional_2_5}
|d^\lambda_{j\ell}(x,\xi)|
\le C_{j\alpha\ell}(V^\lambda (x))^{j}\<x\>^{-|\alpha|}\<\xi\>^\ell. 
\end{align}
The proof of \eqref{proof_proposition_functional_2_2} is then divided into two cases $p^\lambda\lessgtr2d_0$. In case of $p^\lambda\le 2d_0$, we have
$
V^\lambda (x)\le p^\lambda(x,\xi)\le 2d_0
$
which, together with  \eqref{proof_proposition_functional_2_3}, yields that, for all $\gamma\in \Z^n_+$, 
\begin{align}
\label{proof_proposition_functional_2_5_1}
|\partial_x^\gamma V^\lambda (x)|\le C_\gamma \<x\>^{-|\gamma|}
\end{align}
uniformly in $x\in \supp(1-\chi)$ and $\lambda>0$. Moreover, since $V^\lambda$ is positive we have
\begin{align}
\label{proof_proposition_functional_2_5_2}
|p^\lambda-z|\ge C |\Im z|\<\xi\>^2.
\end{align}
\eqref{proof_proposition_functional_2_3}--\eqref{proof_proposition_functional_2_5_2} imply \eqref{proof_proposition_functional_2_2} for $k=0$. On the other hand, if $p^\lambda(x,\xi)\ge2d_0$ then we have
$$
|p^\lambda-z|\ge p^\lambda/2\gtrsim |\xi|^2+V^\lambda(x)+1
$$
which, together with \eqref{proof_proposition_functional_2_3}--\eqref{proof_proposition_functional_2_5}, implies \eqref{proof_proposition_functional_2_2} for $k=0$. \eqref{proof_proposition_functional_2_2} for $k\ge1$ follows from an induction argument in $k$. Moreover, we learn from \eqref{proof_proposition_functional_2_1} and \eqref{proof_proposition_functional_2_4} that $q_k^\lambda$ is of the finite sum
$$
q_k^\lambda=\sum_{j=k}^{2k+1}\frac{\tilde d^\lambda_{kj}(x,\xi)}{(p^\lambda(x,\xi)-z)^{j+1}},
$$
where $\tilde d^\lambda_{kj}$ are independent of $z$, polynomials in $\xi$ of degree less than $j$. Moreover, $\tilde d^\lambda_{kj}$ are supported in $\supp(1-\chi)$ in the $x$-variable for all $\xi$ and satisfies 
\begin{align}
\label{proof_proposition_functional_2_6}
|\partial_x^\alpha \partial_\xi^\beta \tilde d^\lambda_{kj}(x,\xi)|\le C_{\alpha\beta}(V^\lambda (x))^j
\<x\>^{-j-|\alpha|}\<\xi\>^{j-|\beta|}. 
\end{align}
By the construction of $q_k^\lambda$ and Proposition \ref{proposition_PDO_1}, we obtain 
$$
(H^\lambda -z)^{-1}(1-\chi_R)(x)=\sum_{k=0}^{N-1}\Op(q_k^\lambda)+(H^\lambda -z)^{-1}\Op(r_N^\lambda),
$$
where $r_N^\lambda:=\Delta_xq^\lambda_{N-2}+2i\xi\cdot\nabla_x q^\lambda_{N-1}+\Delta q^\lambda_{N-1}$. By \eqref{proof_proposition_functional_2_2}, $\{r_N^\lambda\}_{\lambda>0}$ is a bounded set in $S^{-N,-N-2}$. Plugging this formula into \eqref{HS} and applying Cauchy-Pompeiu's formula give us the formula
$$
\varphi(H^\lambda )(1-\chi_R)(x)=\Op(a^\lambda)+\wtilde Q^\lambda,
$$
where $a^{\lambda}=b_0^\lambda+\cdots+b_{N-1}^\lambda$ and $\wtilde Q^\lambda$ are given by
\begin{align*}
b_{0}^\lambda(x,\xi)&=(1-\chi_R)(x)(\varphi\circ p^\lambda)(x,\xi),\quad\\
b_{k}^\lambda(x,\xi)&=\sum_{j=k}^{2k+1}\frac{(-1)^{j}}{j!}\tilde d^\lambda_{kj}(x,\xi)(\varphi^{(j)}\circ p^\lambda)(x,\xi),\quad 1\le k\le N-1,\\
\wtilde Q^\lambda &=-\frac{1}{\pi }\int_\C \frac{\partial \Phi}{\partial \bar z}(z)( H^\lambda -z)^{-1}\Op(r_N^\lambda)dL(z). 
\end{align*}
Note that $\Phi$ is an almost analytic extension of $\varphi$ given in \eqref{HS}. It is easy to see that $\supp a^{\lambda}\subset\supp(1-\chi)\cap\supp(\varphi\circ p^\lambda)
$. In particular, $p^\lambda\le 2d_0$ on $\supp a^\lambda$ and, thus, the same argument as above implies that $\partial_x^\alpha V^\lambda =O(\<x\>^{-|\alpha|})$ uniformly in  $\lambda>0$ on $\supp a^\lambda$. This bound and \eqref{proof_proposition_functional_2_6} show that $\{a^\lambda\}_{\lambda>0}$ is bounded in $ S^{0,-\infty}$. Finally, \eqref{proof_proposition_functional_2_2} implies
$$
\sup_{\lambda>0}\norm{\<x\>^{N/2}\Op(r_N^\lambda)\<x\>^{N/2}\<D\>^{N/2}}\le C_N|\Im z|^{-n(N)}
$$
with some $n(N)$ depending only on $N$, which, combined with Lemma \ref{lemma_functional_1},  implies
\begin{align*}
&\sup_{\lambda>0}\norm{\<x\>^{N/2}\wtilde Q^\lambda\<x\>^{N/2}\<D\>^{N/2}}\\
&\le C_N\int_\C\left|\frac{\partial \Phi}{\partial \bar z}(z)\right|\sup_{\lambda>0}\norm{\<x\>^{N/2}( H^\lambda -z)^{-1}\<x\>^{-N/2}}\sup_{\lambda>0}\norm{\<x\>^{N/2}\Op(r_N^\lambda)\<x\>^{N/2}\<D\>^{N/2}}|dL(z)|\\
&\le C_{N,M}\int_{\supp\Phi}\<z\>^{N/4}|\Im z|^{M-N/2-1-n(N)}|dL(z)|<\infty. 
\end{align*}
provided $M\ge N/2+1+n(N)$. Replacing $N$ by $2N$, we  complete the proof. \end{proof}

The following analogous results for $\varphi(H_h )$ will also be needed later. 

%proposition
\begin{proposition}
\label{proposition_functional_3}
Let $\alpha\in \R$, $N\in \N\cup\{0\}$ and $h=\lambda^{2/\mu-1}$. Then $\<x\>^{\alpha}\varphi(H_h  )\<x\>^{-\alpha}$ is bounded on $L^2$ uniformly in $\lambda\in (0,1]$. Moreover, for any $R>0$, there exists $b_h\in S^{0,-\infty}$ supported in $\supp[(1-\chi_R)\varphi\circ p_h]$, where $p_h(x,\xi)=|\xi|^2+V_h$, such that
$$
(1-\chi_R)(x)\varphi(H_h  )=\Op_{h}(b_{h})^*+Q_h,\quad 
\sup_{h\in (0,1]}(h^{-N}\norm{\<hD\>^N\<x\>^NQ_h\<x\>^N})\le C_N.
$$
\end{proposition}

%proof
\begin{proof}
The proof is similar to that for $\varphi(H^\lambda )$. Indeed, since $\<hD\>^{-1}[{h}^2\Delta,\<x\>^N]\<x\>^{-N+1}=O_{L^2}(1)$, the uniform $L^2$-boundedness of $\<x\>^{\alpha}\varphi(H_h  )\<x\>^{-\alpha}$ in $h\in (0,1]$ can be  verified analogously. On the other hand, since the semiclassical symbol of $(H_h  -z)\Op_{h}(q)$ has the expansion 
$$
(p_h-z)\# q=(p_h-z)q-2i{h}\xi\cdot\nabla_x q-h^2\Delta_x q,
$$
if we define $q_{k}^\lambda$ by the same manner as in \eqref{proof_proposition_functional_2_1} with $p^\lambda$ replaced by $p_h$, then, thanks to Lemma \ref{lemma_dilation_1}, $q_{k}^\lambda$ satisfies the same estimate \eqref{proof_proposition_functional_2_2} and we have
$$
(H_h  -z)^{-1}(1-\chi_R)(x)=\sum_{k=0}^{N-1}h^k\Op_{h}(q_{k}^\lambda)+h^N(H_h  -z)^{-1}\Op(r_{N}^\lambda)
$$
with $r_N^\lambda=\Delta_xq_{N-2}^\lambda+2i\xi\cdot\nabla_x q_{N-1}^\lambda+h\Delta q_{N-1}^\lambda\in S^{-N,-N-2}$. The rest of the proof is completely analogous to that of Proposition \ref{proposition_functional_2} and we omit it.
\end{proof}

\begin{remark}
\label{remark_functional_1}
Under the hypothesis (H1) and (H2), it follows by essentially the same proof as that of Proposition \ref{proposition_functional_2} that, for each $N\in \N$, $\varphi(H)$ can be decomposed as a sum of $\Op(a_N)$ with some $a_N\in S^{0,-\infty}$ and an error term $Q_N$ satisfying $\<D\>^N\<x\>^NQ_N\<x\>^N\in \mathbb B(L^2)$. This remark will be used in the proof of Proposition \ref{proposition_smoothing_2} below. 
\end{remark}

%subsection
\subsection{Weighted space-time and local decay estimates}
\label{subsection_smoothing}
Here we collect several results on wighted space-time and local decay estimates, which will play a crucial role in the proof of Theorem \ref{theorem_LP_2}. 

%{lemma}
\begin{proposition}
\label{proposition_smoothing_1}
We have
\begin{align}
\label{proposition_smoothing_1_1}
\norm{\<x\>^{-1}e^{-itH}u_0}_{L^2(\R^{1+n})}&\lesssim \norm{u_0}_{L^2(\R^n)},\\\label{proposition_smoothing_1_2}\bignorm{\<x\>^{-1}\int_0^t e^{-i(t-s)H}\<x\>^{-1}F(s)ds}_{L^2(\R^{1+n})}&\lesssim \norm{F}_{L^2(\R^{1+n})}.
\end{align}
\end{proposition}

%proof
\begin{proof}
By Kato's smooth perturbation theory \cite[Theorem 5.1]{Kat} (see also \cite{Dan}), \eqref{proposition_smoothing_1_1} and \eqref{proposition_smoothing_1_2} are consequences of the following uniform resolvent estimate
$$
\sup_{z\in \C\setminus\R}\norm{\<x\>^{-1}(H-z)^{-1}\<x\>^{-1}}<\infty
$$
which was proved by Nakamura \cite[Theorem 1.8]{Nak2}. \end{proof}

%remark
\begin{remark}
It was actually proved in \cite{Nak2} that $\<x\>^{-\rho}(H-z)^{-1}\<x\>^{-\rho}$ is bounded on $L^2$ uniformly in $z\in \C\setminus\R$ if $\rho>1/2+\mu/4$, which implies \eqref{proposition_smoothing_1_1} with $\<x\>^{-1}$ replaced by $\<x\>^{-\rho}$.  However, Proposition \ref{proposition_smoothing_1} is sufficient for our purpose. 
\end{remark}

In the proof of Theorem \ref{theorem_LP_2}, the energy localized version of Proposition \ref{proposition_smoothing_1} will be also required. The following proposition concerns with low energy estimates. 

\begin{proposition}
\label{proposition_smoothing_2}
Let $\varphi\in C_0^\infty((0,\infty))$. Then we have
\begin{align}
\label{proposition_smoothing_2_1}
\norm{\<x\>^{-1}\varphi(H_h  )e^{-itH_h  /h}u_0}_{L^2(\R^{1+n}) }&\lesssim\norm{u_0}_{L^2(\R^n)},\\
\label{proposition_smoothing_2_2}
\bignorm{\<x\>^{-1}\varphi(H_h  )\int_0^t e^{-i(t-s)H_h  /h}\<x\>^{-1}F(s)ds}_{L^2(\R^{1+n})}&\lesssim \norm{F}_{L^2(\R^{1+n})}
\end{align}
uniformly in $h=\lambda^{1/2-\mu}\in (0,1]$. Moreover, for any $s>s'>0$ and $\ep>0$,
\begin{align}
\label{proposition_smoothing_2_3}
\norm{\<x\>^{-s}e^{-itH_h  /h}\varphi(H_h)\<x\>^{-s}}&\lesssim h^{-\ep}\<t\>^{-s'},\quad t\in \R,\quad \lambda\in (0,1].
\end{align}
\end{proposition}

%proof
\begin{proof}
We may replace $\varphi$ by $\varphi^2$. It suffices to show that, for any $N\in \N$ and $\gamma>N-1/2$, 
\begin{align}
\label{proof_proposition_smoothing_2_1}
\sup_{z\in \C\setminus\R} \norm{\<x\>^{-\gamma}\varphi(H_h)(H_h  -z)^{-N}\varphi(H_h  )\<x\>^{-\gamma}}\lesssim h^{-N}
\end{align}
holds uniformly in $h\in (0,1]$. Indeed, as before, \eqref{proposition_smoothing_2_1} and \eqref{proposition_smoothing_2_2} follow from \eqref{proof_proposition_smoothing_2_1} with $N=1$.  \eqref{proposition_smoothing_2_3} is a consequence of \eqref{proof_proposition_smoothing_2_1} and an abstract method by Jensen-Mourre-Perry \cite{JMP} (see also \cite[Theorem 2]{Nak1}  for its semiclassical version). A sketch of the proof of the bound \eqref{proof_proposition_smoothing_2_1} can be found in \cite[Lemma 2.1]{Nak2}. We give its details in Appendix \ref{appendix_A} for the sake of completeness.  
\end{proof}

In the high energy regime $\lambda\ge1$, we also have the following similar bounds. 
%{lemma}
\begin{proposition}
\label{proposition_smoothing_3}
For any $\varphi\in C_0^\infty((0,\infty))$  one has
\begin{align}
\label{proposition_smoothing_3_1}
\norm{\<x\>^{-1}\varphi(H^\lambda )e^{-itH^\lambda }u_0}_{L^2(\R^{1+n})}&\lesssim  \norm{u_0}_{L^2(\R^n)},\\
\label{proposition_smoothing_3_2}
\bignorm{\<x\>^{-1}\varphi(H^\lambda )\int_0^t e^{-i(t-s)H^\lambda}\<x\>^{-1}F(s)ds}_{L^2(\R^{1+n})}&\lesssim \norm{F}_{L^2(\R^{1+n})}
\end{align} 
uniformly in $\lambda\ge1$. Moreover, for all $s>s'>0$, one has
\begin{align}
\label{proposition_smoothing_3_3}
\norm{\<x\>^{-s}\varphi(H^\lambda )e^{-itH^\lambda }\<x\>^{-s}}&\lesssim \<t\>^{-s'},\quad\lambda\ge1,\quad  t\in \R.
\end{align} 
\end{proposition}

%proof
\begin{proof}
As for the previous proposition, the results follow from the uniform bound
\begin{align}
\label{proof_proposition_smoothing_3_1}
\sup_{\lambda\ge1}\sup_{z\in \C\setminus\R}\norm{\<x\>^{-\gamma}\varphi(H^\lambda )(H^\lambda -z)^{-N}\varphi(H^\lambda )\<x\>^{-\gamma}}<\infty
\end{align}
which can be proved by the same argument as that of \eqref{proof_proposition_smoothing_2_1}. We thus omit it.
\end{proof}

%%%%%%%%%%%%%%%%%%%%%%%%
\section{Homogeneous estimates}
\label{section_homogeneous}

In this section we prove \eqref{theorem_LP_2_1}, completing the proof of \eqref{theorem_1_1}. Let $\varphi$ and $\chi_R$ be as in the beginning of Subsection \ref{subsection_functional} and $\wtilde\varphi\in C_0^\infty((0,\infty))$ with $\wtilde\varphi\equiv1$ on $\supp \varphi$. Let 
$\varphi^\lambda(H)=\varphi(\lambda^{-2}H)$, $\wtilde \varphi^\lambda(H)=\wtilde \varphi(\lambda^{-2}H)$, $\chi^\lambda(x)=\chi_R(\lambda x)
$ and
\begin{align}
\label{equation_reduction_0}
A^\lambda:=\D \Op(a^\lambda)^*\D ^*,\quad
R^\lambda:=\D Q^\lambda\D ^*+\chi^\lambda\wtilde \varphi^\lambda(H),
\end{align}
Proposition \ref{proposition_functional_2} then yields that 
\begin{align}
\label{equation_reduction_1}
\varphi^\lambda(H)
=\wtilde \varphi^\lambda(H) \wtilde \varphi^\lambda(H)\varphi^\lambda(H)
=\wtilde \varphi^\lambda(H)A^\lambda \varphi^\lambda(H)+\wtilde \varphi^\lambda(H)R^\lambda \varphi^\lambda(H).
\end{align}
The following proposition provides desired Strichartz estimates for the remainder term. 

%proposition
\begin{proposition}
\label{proposition_reduction_1}
We have, uniform in $\lambda>0$,
\begin{align}
\label{proposition_reduction_1_1}
\norm{\wtilde \varphi^\lambda(H) R^\lambda \varphi^\lambda(H) e^{-itH}u_0}_{L^2(\R;L^{2^*}(\R^n)) }
\lesssim \norm{u_0}_{L^2(\R^n)}.
\end{align}
\end{proposition}

%remark
\begin{remark}Note that \eqref{proposition_reduction_1_1} holds for all $n\ge2$. Interpolating between \eqref{proposition_reduction_1_1} and the trivial $L^2_x$-$L^\infty_tL^2_x$ estimate also implies \eqref{proposition_reduction_1_1} with $(2,2^*)$ replaced by any admissible pair $(p,q)$. 
\end{remark}

%proof
\begin{proof}
Since $\norm{\wtilde \varphi^\lambda(H)}_{2\to 2^*}\lesssim \lambda$ by Lemma \ref{lemma_LP_2}, it suffices to show the following uniform bound
\begin{align}
\label{proof_proposition_reduction_1_1}
\lambda\norm{R^\lambda \varphi^\lambda(H) e^{-itH}u_0}_{L^2_tL^2_x }\lesssim \norm{u_0}_{L^2},\quad \lambda>0.
\end{align}
When $0<\lambda\le1$, the support property of $\chi$ implies $\lambda \chi^\lambda(x)\lesssim \<x\>^{-1}.$
Proposition \ref{proposition_functional_2}, the  formula $\<x\>^{-1}\D ^*=\D ^*\<\lambda x\>^{-1}$ and the unitarity of $\D $ also imply
$$
\lambda\norm{\D Q^\lambda\D ^*f}_2\lesssim \lambda\norm{\<\lambda x\>^{-1}f}_2\lesssim \norm{\<x\>^{-1}f}_2.
$$
These two bounds and Proposition \ref{proposition_smoothing_1} gives us \eqref{proof_proposition_reduction_1_1} for $\lambda\le1$. When $\lambda\ge1$, we use \eqref{dilation_1}--\eqref{dilation_3},  the change of variable $s=t\lambda^{2}$ and Proposition \ref{proposition_smoothing_3} to obtain
\begin{align*}
\lambda \norm{R^\lambda \varphi^\lambda(H) e^{-itH}u_0}_{L^2_tL^2_x }
&=\lambda\bignorm{\Big(Q^\lambda +\chi(x)\wtilde \varphi(H^\lambda )\Big)\varphi(H^\lambda )e^{-it\lambda ^2H^\lambda }\D ^*u_0}_{L^2_tL^2_x }\\
&\lesssim \norm{\<x\>^{-1}\varphi(H^\lambda )e^{-isH^\lambda }\D ^*u_0}_{L^2_tL^2_x }\\
&\lesssim \norm{u_0}_{L^2}
\end{align*}
which completes the proof. 
\end{proof}

By virtue of this proposition and Lemma \ref{lemma_LP_2}, it remains to show that the estimate
\begin{align}
\label{equation_reduction_2}
\norm{A^\lambda \varphi^\lambda(H) e^{-itH}u_0}_{L^p_tL^q_x}\lesssim  \norm{u_0}_{L^2}
\end{align}
holds uniformly in $\lambda>0$. The proof of \eqref{equation_reduction_2} is divided into the high energy $\lambda\ge1$ and the low energy $0<\lambda\le1$ cases. In the high energy regime, using \eqref{dilation_2} we have the equality 
$$
A^\lambda \varphi^\lambda(H) e^{-itH}=\D \Op(a^\lambda)^*\varphi(H^\lambda )e^{-it\lambda^{2}H^\lambda }\D ^*
$$
which, together with  \eqref{dilation_1}, implies that \eqref{equation_reduction_2} with $\lambda\ge1$ is equivalent to the estimate
\begin{align}
\label{IK_1}
\norm{\Op(a^\lambda)^*\varphi(H^\lambda )e^{-itH^\lambda }u_0}_{L^p_tL^q_x}\lesssim  \norm{u_0}_{L^2},\quad \lambda\ge1.
\end{align}
On the other hand, when $\lambda\le1$, we use $\D_\mu=\D(\lambda^{2/\mu}) $ to write
\begin{align*}
A^\lambda \varphi^\lambda(H) e^{-itH}
=\D_\mu \Op_{h}( a_{h})^*\varphi(H_{h})e^{-it\lambda^{2}H_{h}}\D_\mu ^*,
\end{align*}
where $a_{h}(x,\xi)=a^\lambda(h^{-1}x,\xi)$. By \eqref{dilation_1}, \eqref{equation_reduction_2} with $0<\lambda\le1$ then is equivalent to 
\begin{align}
\label{IK_2}
\norm{\Op_{h}( a_{h})^*\varphi(H_{h})e^{-i t H_{h}/h }u_0}_{L^p_tL^q_x}\lesssim  h ^{-1/p}\norm{u_0}_{L^2},\quad h=\lambda^{2/\mu-1}\in(0,1].
\end{align}

The proof of the high energy estimate \eqref{IK_1} is simpler than that of the low energy estimate \eqref{IK_2}. Therefore, we first give the proof of \eqref{IK_2} in detail and, then, explain necessary modifications for the high energy estimate \eqref{IK_1}.  

\subsection{The low energy case}
\label{subsection_low _energy}

We begin with observing that $a_h$ belongs to $S^{0,-\infty}$ and satisfies the support property
\begin{align}
\label{a_h}
\supp a_h\subset \{(x,\xi)\in \R^{2n}\ |\ |x|>c_0,\ |\xi|<c_1\}
\end{align}
with some $c_0,c_1>0$ being independent of $h\in (0,1]$. Indeed, since $V^\lambda(h^{-1}x)=V_h(x)$, $a_h$ is supported in $\supp[(1-\chi_R)(h^{-1}\cdot)\varphi \circ p_h]$ and hence
\begin{align*}
\supp a_h\subset \{(x,\xi)\in \R^{2n}\ |\ |x|\ge \lambda^{2/\mu-1} R,\ d^{-1}<|\xi|^2+V_h(x)<d\}
\end{align*}
with some $d>0$. This, together with  Lemma \ref{lemma_dilation_1}, yields that $C_1(\lambda^{2/\mu}+|x|)^{-\mu}<d$ and hence
$$
|x|\ge \max\{(C_1/d)^{1/\mu}-\lambda^{2/\mu},\lambda^{2/\mu-1} R\}\ge (C_1/d)^{1/\mu}/2
$$
provided $R\ge1$. Since $\{a^\lambda\}_{\lambda>0}$ is bounded in $S^{0,-\infty}$, this support property of $a_h$ implies
$$
|\partial_x^\alpha\partial_\xi^\beta a_{h}(x,\xi)|
\le C_{\alpha\beta N} h^{-|\alpha|}\<h ^{-1}x\>^{-|\alpha|}\<\xi\>^{-N-|\beta|}\le C_{\alpha\beta N}\<x\>^{-|\alpha|}\<\xi\>^{-N-|\beta|}
$$
for any $N\ge0$. This shows $a_h\in S^{0,-\infty}$. Note that $c_0$ in \eqref{a_h} is not necessarily large and, thus, \eqref{a_h} is not enough to construct a long-time parametrix of $\Op_{h}( a_{h})^*\varphi(H_{h})e^{-itH_{h}/h}$. Therefore, we further decompose $a_h$ into compact and non-compact parts as follows:
$$
a_h=a_h^{\mathop{\mathrm{com}}}+a_h^{\mathop{\mathrm{\infty}}},\quad a_h^{\mathop{\mathrm{com}}}=\chi_R(x)a_h,\quad a_h^{\mathop{\mathrm{\infty}}}=(1-\chi_R)(x)a_h
$$
with $\chi_R(x)=\chi(x/R)$ and $\chi\in C_0^\infty(\R^n)$ satisfying $\chi\equiv1$ near the origin. The following proposition provides desired Strichartz  estimates for the compact part. 

%proposition
\begin{proposition}
\label{proposition_dispersive_1}
Let $c_1>c_0>0$ and $b_h\in S^{0,-\infty}$ be supported in $\{c_0<|x|<c_1,\,|\xi|\le c_1\}$. Then for any admissible pair $(p,q)$ one has
$$
\norm{\Op_{h}(b_{h})^*e^{-i t H_{h}/h }u_0}_{L^p(\R;L^q(\R^n))}\lesssim  h ^{-1/p}\norm{u_0}_{L^2(\R^n)},\quad h\in (0,1]. 
$$
\end{proposition}

%remark
\begin{remark}
$a_h^{\mathop{\mathrm{com}}}$ satisfies the condition on $b_h$ in this proposition. 
\end{remark}

The following is a key ingredient in  the proof of this proposition. 

%{lemma}
\begin{lemma}
\label{lemma_dispersive_2}
Let $b_h$ be as in Proposition \ref{proposition_dispersive_1} and $\wtilde b\in S^{0,-\infty}$ such that $\wtilde b\equiv1$ on $\supp b_h$ and $\supp\wtilde b\subset \{c_0/2<|x|<2c_1,\ |\xi|\le 2c_1\}$. Then there exists $\ep_0>0$ such that, for any interval $I$ with $|I|\le 2\ep_0$ and admissible pair $(p,q)$,
\begin{align}
\label{lemma_dispersive_2_1}
\norm{\Op_h(\wtilde b)^*e^{-itH_h/h}f}_{L^p(I;L^{q}(\R^n))}&\lesssim h^{-1/p}\norm{f}_{L^2(\R^n)},\quad h\in (0,1]. 
\end{align}
\end{lemma}

%proof
\begin{proof}
By the standard $TT^*$-argument in \cite{KeTa}, it suffices to show the following dispersive estimate
\begin{align}
\label{proof_lemma_dispersive_2_1}
\norm{\Op_h(\wtilde b)^*e^{-itH_h/h}\Op_h(\wtilde b)}_{1\to \infty}\lesssim |th|^{-n/2},\quad 0<|t|\le 2\ep_0,
\end{align}
whose proof is based on a semiclassical parametrix of $e^{-itH_h/h}\Op_h(\wtilde b)$ and  the stationary phase method. Since such a parametrix construction is well known (see \cite{Rob1}), we only outline it. 

Note that, since $|x|>c_0$ on $\supp \wtilde b$, $V_h$ satisfies $\partial_x^\alpha V_h(x)=O(\<x\>^{-\mu-|\alpha|})$ on $\supp \wtilde b$ uniformly in $h\in (0,1]$. Using this fact, for all $N\in \N$ and sufficiently small $\ep_0>0$, one can construct smooth functions $\Psi_h,d_h\in C^\infty( (-3\ep_0,3\ep_0)\times\R^{2n})$ satisfying the following properties: First, $\Psi_h$ satisfies the Hamilton-Jacobi equation 
$$
\partial_t \Psi_h+p_h(x,\nabla_x\Psi_h)=0;\quad \Psi_h|_{t=0}=x\cdot\xi,
$$
on a small neighborhood of $\supp \wtilde b$. Second, $\{d_h(t)\}_{|t|\le 3\ep_0}$ is bounded in $S^{-\infty,-\infty}$. Moreover, $d_h$ approximately solves the following transport equation in such a way that
\begin{align}
\label{proof_lemma_dispersive_2_1_0}
\partial_t d_h+\mathcal X_h\cdot\nabla_x d_h+\mathcal Y_hd_h=h^Nr_{h};\quad d_h|_{t=0}=\wtilde b
\end{align}
with some bounded set $\{r_h(t)\}_{|t|\le 3\ep_0}\subset S^{-\infty,-\infty}$, where $\mathcal X_h=2\nabla_x\Psi_h$ and $\mathcal Y_h=\Delta_x\Psi_h$. Finally, for all $(t,x,\xi)\in (-3\ep_0,3\ep_0)\times\R^{2n}$, $\Psi_h$ satisfies
\begin{align}
\label{proof_lemma_dispersive_2_2}
|\partial_x^\alpha\partial_x^\beta\Psi_h(t,x,\xi)|&\le C_{\alpha\beta},\quad |\alpha|+|\beta|\ge2,\\
\label{proof_lemma_dispersive_2_3}
|\nabla_x\otimes\nabla_\xi\Psi_h(t,x,\xi)-\Id|&\le C|t|,\\
\label{proof_lemma_dispersive_2_4}
|\nabla_\xi^2\Psi_h(t,x,\xi)-2t\Id|&\le C|t|^2.
\end{align}
To $\Psi_h$ and $a\in S^{0,-\infty}$, we associate with a semiclassical Fourier integral operator ($h$-FIO)
$$
J_{\Psi_h(t)}(a)f(x)=(2\pi h)^{-n}\iint e^{i(\Psi_h(t,x,\xi)-y\cdot\xi)/h}a(x,\xi)f(y)dyd\xi.
$$
If we compute $\int_0^t \frac{d}{ds}e^{isH_h/h}J_{\Psi_h(s)}(d_h(s))ds$, then by virtue of \eqref{proof_lemma_dispersive_2_1_0}, we have Duhamel's formula
\begin{align}
\label{proof_lemma_dispersive_2_5}
e^{-itH_h/h}\Op_h(\wtilde b)=J_{\Psi_h(t)}(d_h(t))+ih^{N-1}\int_0^te^{-i(t-s)H_h/h}J_{\Psi_h(s)}(r_h(s))ds.
\end{align}
With \eqref{proof_lemma_dispersive_2_2} and \eqref{proof_lemma_dispersive_2_3} at hand, we learn from a standard theory of $h$-FIO (see \cite{Rob2}) that
\begin{align}
\label{proof_lemma_dispersive_2_6}
\norm{J_{\Psi_h(t)}(d_h(t))}\lesssim 1,\quad \norm{J_{\Psi_h(t)}(r_h(t))}_{\mathbb B(\H^{-s},\H^s)}\lesssim h^{-2s}
\end{align}
for all $s\in \R$, uniformly in $|t|\le 2\ep_0$ and $h\in (0,1]$. Moreover, \eqref{proof_lemma_dispersive_2_4} and the stationary phase method yield the following decay estimate
\begin{align}
\label{proof_lemma_dispersive_2_7}
\norm{J_{\Psi_h(t)}(d_h(t))}_{1\to\infty}\lesssim \min\{h^{-n},|th|^{-n/2}\},\quad |t|\le 2\ep_0,\ h\in (0,1].
\end{align}
\eqref{proof_lemma_dispersive_2_1} follows from \eqref{proof_lemma_dispersive_2_5}--\eqref{proof_lemma_dispersive_2_7}, \eqref{PDO_1} and Sobolev's embedding. 
\end{proof}

To prove Proposition \ref{proposition_dispersive_1}, one more technical lemma will be needed.

%{lemma}
\begin{lemma}
\label{lemma_dispersive_3}
For any $N\in \N$ there exist $c_h,r_{1,h},r_{2,h}\in S^{-\infty,-\infty}$ such that
\begin{align}
\label{lemma_dispersive_3_1}
\Op_h(b_h)^*&=\Op_h(\wtilde b)^*\Op_h(b_h)^*+h^N\Op_h(r_{1,h}),\\
\label{lemma_dispersive_3_2}
[\Op_h(b_h)^*,H_h  ]&=h\Op_h(c_h)+h^N\Op_h(r_{3,h})(1+V_h).
\end{align}
\end{lemma}

%proof
\begin{proof}
\eqref{lemma_dispersive_3_1} follows from Proposition \ref{proposition_PDO_1} since $b_h\wtilde b\equiv b_h$. To show \eqref{lemma_dispersive_3_2}, we first observe from Proposition \ref{proposition_PDO_1} and \eqref{lemma_dispersive_3_1} that there exist $c_h'\in S^{-\infty,-\infty}$ and $r'_h\in S^{-\infty,\infty}$ such that
\begin{align*}
[\Op_h(b_h)^*,-h^2\Delta]=h\Op(c_h')+h^N\Op_h(r'_h). 
\end{align*}
To deal with $[\Op_h(b_h)^*,V_h]$, we take $\chi\in C_0^\infty(\R^n)$ satisfying $\chi\equiv1$ on $\supp b_h$ and $\supp \chi\Subset \pi_x(\supp\wtilde b)$, where $\pi_x$ is the projection onto the $x$-space. Then we obtain by Proposition \ref{proposition_PDO_1} that
\begin{align*}
\Op_h(b_h)^*=\Op_h(b_h)^*\chi=\chi\Op_h(b_h)^*\chi+h^N\Op_h(r''_h),\quad r''_h\in S^{-\infty,-\infty},
\end{align*}
and hence $[\Op_h(b_h)^*,V_h]=[\Op_h(b_h)^*,\chi V_h]+h^N\Op_h(r''_h)V_h$. 
Since $|\partial_x^\alpha(\chi V_h)|=O(\<x\>^{-\mu-|\alpha|})$ uniformly in $h$ (note that $\supp\chi V_h\subset \{|x|>c_0/2\}$), the symbol of $[\Op_h(b_h)^*,\chi V_h]$ belongs to $hS^{-\infty,-\infty}$ and, thus, \eqref{lemma_dispersive_3_2} follows.  
\end{proof}

%proof
\begin{proof}[Proof of Proposition \ref{proposition_dispersive_1}]
The proof is based on a similar argument as that in \cite{StTa,BoTz1}. Let $u_h=\varphi(H_h  )e^{-itH_h/h}u_0$ and $N\gg1$ be large enough (specified later). It follows from  \eqref{lemma_dispersive_3_1} that $$\Op_h(b_h)^*=\Op_h(\wtilde b)^*\Op_h(b_h)^*+h^N\Op_h(r_{h})^*$$ with some $r_h\in S^{-\infty,-\infty}$. By \eqref{PDO_1} and \eqref{proposition_smoothing_2_1}, the remainder  $h^N\Op_h(r_{h})^*u_h$ satisfies
\begin{align}
\label{proof_proposition_dispersive_1_0}
\norm{h^N\Op_h(r_{h})^*u_h}_{L^2_tL^{2^*}_x  }\lesssim h^{N-1}\norm{\<x\>^{-1}u_h}_{L^2_tL^2_x }\lesssim h^{N-3/2}\norm{u_0}_{L^2}\lesssim \norm{u_0}_{L^2}
\end{align}
if $N\ge 3/2$. 
To deal with the main term we consider a decomposition of the time interval 
\begin{align}
\label{proof_proposition_dispersive_1_00}
\R=\bigcup_{j\in \Z} I_j,\quad I_j=[(j-1/2)\ep_0,(j+1/2)\ep_0],
\end{align}
and choose intervals $\wtilde I_j$ centered at $j\ep_0$ satisfying $I_j\Subset \wtilde I_j$, $|\wtilde I_j|\le 3\ep_0/2$. Let $\theta_j\in C_0^\infty(\R)$ be such that $\theta_j\equiv1$ on $I_j$, $I_j\Subset \supp\theta_j\Subset \wtilde I_j$. Set $v_j(t)=\theta_j(t)\Op_h(b_h)^*u_h$. Then
\begin{align*}
\norm{\Op_h(\wtilde b)^*\Op_h(b_h)^*u_h}_{L^p_tL^q_x}^p\le \sum_{j\in \Z}\norm{\Op_h(\wtilde b)^*v_j}_{L^p(\wtilde I_j;L^q)}^p
\end{align*}
by the almost orthogonality of $\theta_j$. Now we claim that $v_j$ satisfies
\begin{equation}
\norm{\Op_h(\wtilde b)^*v_j}_{L^p(\wtilde I_j;L^q)}\lesssim
\left\{\begin{aligned}
\label{proof_proposition_dispersive_1_1}
&h^{-1/p}\norm{u_0}_{L^2},\quad j=0,\\
&h^{-1/p}\norm{\<x\>^{-1}u_h}_{L^2(\wtilde I_j;L^2)},\ j\neq0.
\end{aligned}
\right.
\end{equation}
Since $p\ge2$, this claim, together with \eqref{proposition_smoothing_2_1} and Minkowski's inequality, implies
\begin{align*}
\sum_{j\in \Z}\norm{\Op_h(\wtilde b)^* v_j}_{L^p(\wtilde I_j;L^q)}^p
\lesssim h^{-1}(\norm{u_0}_{L^2}^p+\norm{\<x\>^{-1}e^{-itH_h  /h}u_0}_{L^2_tL^2_x }^p)
\lesssim h^{-1}\norm{u_0}_{L^2}^p
\end{align*}
which, combined with \eqref{proof_proposition_dispersive_1_0}, shows \eqref{lemma_dispersive_2_1}. It thus remains to show \eqref{proof_proposition_dispersive_1_1}. The case $j=0$ follows  directly from \eqref{lemma_dispersive_2_1}. For $j\neq0$ we observe that $v_j$ satisfies
$$
(ih\partial_t-H_h  )v_j=-ih\theta_j'(t)\Op_h(b_h)^*u_h(t)+\theta_j(t)[H_h  ,\Op_h(b_h)^*]u_h(t);\quad v_j|_{t=0}=0,
$$
which leads to Duhamel's formula for $\Op_h(\wtilde b)^*v_j$:
$$
\Op_h(\wtilde b)^*v_j(t)=- i\int_0^t\Op_h(\wtilde b)^*e^{-i(t-s)H_h  /h}w_h(s)ds,
$$
where $w_h(s)=\{-i \theta_j'(s)\Op_h(b_h)^*+h^{-1}\theta_j(s)[\Op_h(b_h)^*,H_h  ]\}u_h(s)$. Note that $w_h(s)$ is supported in $\wtilde I_j$ with respect to $s$. 
Now we observe from \eqref{lemma_dispersive_2_1} and Lemma \ref{lemma_dispersive_3} that
\begin{align}
\nonumber
&\bignorm{\int_0^\infty \Op_h(\wtilde b)^*e^{-i(t-s)H_h  /h}w_h(s)ds}_{L^p(\wtilde I_j;L^q)}\\
\nonumber
&\lesssim h^{-1/p}\bignorm{\int_0^\infty e^{isH_h  /h}w_h(s)ds}_{L^2}\\
\nonumber 
&\lesssim h^{-1/p}\norm{w_h}_{L^1(\wtilde I_j;L^2)}\\
\label{proof_proposition_dispersive_1_2}
&\lesssim h^{-1/p}(1+h^{N-2/(2-\mu)})\norm{\<x\>^{-1}u_h}_{L^1(\wtilde I_j;L^2)}
\end{align}
for any $N\ge0$, where we used the fact $b_h\<x\>\in C_0^\infty(\R^{2n})$ and $|V_h|\lesssim h^{-2/(2-\mu)}$ in the last line. Since $p>1$, Christ-Kiselev's lemma \cite{ChKi} shows that, in the left hand side of \eqref{proof_proposition_dispersive_1_2}, the integral over $[0,\infty)$ can be replaced by an integral over $[0,t]$. Since $|\wtilde I_j|\le 2\ep_0$, this implies that
\begin{align*}
\norm{\Op_h(\wtilde b)^*v_j(t)}_{L^p(\wtilde I_j;L^q)}
&\lesssim h^{-1/p}(1+h^{N-2/(2-\mu)})\norm{\<x\>^{-1}u_h}_{L^1(\wtilde I_j;L^2)}\\
&\lesssim h^{-1/p}(1+h^{N-2/(2-\mu)})\norm{\<x\>^{-1}u_h}_{L^2(\wtilde I_j;L^2)}
\end{align*}
Choosing $N\ge 2/(2-\mu)$, we complete the proof of \eqref{proof_proposition_dispersive_1_1}. 
\end{proof}

By virtue of Proposition \ref{proposition_dispersive_1}, in order to obtain \eqref{IK_2}, it remains to show that
\begin{align}
\label{IK_3}
\norm{\Op_{h}( a_{h}^\infty)^*\varphi(H_{h})e^{-i t H_{h}/h }u_0}_{L^p_tL^q_x}&\lesssim h ^{-1/p}\norm{u_0}_{L^2}
\end{align}
holds uniformly in $h=\lambda^{2/\mu-1}\in(0,1]$. Note that, for $R\ge1$ large enough, there exists a relatively compact open interval $I\Subset(0,\infty)$ such that $a_h^\infty$ is supported in the region
$$
\{(x,\xi)\in \R^{2n}\ |\ |x|> R,\ |\xi|^2\in I\}.
$$
We first introduce some notation. For  $R\ge 1$, a relatively compact interval $I\Subset (0,\infty)$ and $\sigma\in (-1,1)$, the outgoing/incoming regions are defined respectively by
\begin{align}
\label{outgoing}
\Gamma^\pm (R,I,\sigma):=\{(x,\xi)\in \R^{2n}\ |\ |x|>R,\ |\xi|^2\in I,\ \pm \cos(x,\xi)>-\sigma\}
\end{align}
where $\cos(x,\xi)=x\cdot\xi/(|x||\xi|)$. 
Note that 
$\{|x|>R,|\xi|^2\in I\}\subset \Gamma^+(R,I,\sigma)\cup \Gamma^-(R,I,\sigma)$ if $\sigma>0$. Also note that $\Gamma^\pm(R,I,\sigma)$ are decreasing in $R$ and increasing in $I,\sigma$, namely
$$\Gamma^\pm(R_1,I_1,\sigma_1)\subset \Gamma^\pm(R_2,I_2,\sigma_2)\quad \text{if}\quad R_2<R_1,\ I_1\Subset I_2,\ \sigma_1<\sigma_2.$$ 
Given $R\ge 2$ and $I_1\Subset I_2,\sigma_1<\sigma_2$, one can construct $\chi^\pm_{1\to 2}\in S^{0,-\infty}$ which is of the form
$$
\chi^\pm_{1\to 2}(x,\xi)=\rho_R(x)\rho_{I_1,I_2}(\xi)\rho_{\sigma_1,\sigma_2}(\pm\cos(x,\xi))
$$
with some $\rho_R\in C^\infty(\R^n)$, $\rho_{I_1,I_2}\in C_0^\infty(\R^n)$ and $\rho_{\sigma_1,\sigma_2}\in C^\infty(\R)$ such that $\chi^\pm_{1\to2}\equiv1$ on $\Gamma^\pm(R,I_1,\sigma_1)$ and $\supp \chi^\pm_{1\to2}\subset \Gamma^\pm(R^{1/2},I_2,\sigma_2)$ (see \cite[Proposition 3.2]{BoTz1}). Now we decompose $a^\infty_h$ into outgoing and incoming parts
$$
a^\infty_h=a^++a^-,\quad a^\pm \in S^{0,-\infty},\quad \supp a^\pm \subset \Gamma^\pm(R,I,1/2).
$$
As in Lemma \ref{lemma_dispersive_2}, \eqref{IK_3} is a consequence of the following dispersive estimate
\begin{align}
\label{dispersive}
\norm{\Op_{h}(a^\pm)^*\varphi(H_h)e^{-i t H_h/h }\Op_{h}(a^\pm)}_{1\to\infty}\lesssim |th|^{-n/2},\quad t\neq0,\ h\in (0,1].
\end{align}
We actually show a slightly more general statement as follows: 

%theorem
\begin{theorem}	
\label{theorem_IK_1}
Let $I\Subset (0,\infty)$ be a relatively compact interval, $\sigma\in (-1,1)$, $\varphi\in C_0^\infty((0,\infty)$ and $a^\pm,b^\pm\in S^{0,-\infty}$ be supported in $\Gamma^\pm(R,I,\sigma)$. Then, for sufficiently large $R$, 
\begin{align}
\label{theorem_IK_1_1}
\norm{\Op_{h}(b^\pm)^*\varphi(H_h)e^{-i t H_h/h }\Op_{h}(a^\pm)}_{1\to\infty}\lesssim |th|^{-n/2}\end{align}
uniformly in $\pm t\le0$ and $h\in (0,1]$.
\end{theorem}

%proof
\begin{proof}[Proof of \eqref{dispersive}, assuming Theorem \ref{theorem_IK_1}]
We use a the same argument as that in \cite{BoTz1}. Let $$
U^\pm(t)=\Op_{h}(a^\pm)^*\varphi(H_h)e^{-i t H_h/h }\Op_{h}(a^\pm).
$$
and $K^\pm (t,x,y)$ denote the Schwartz kernel of $U^\pm (t)$. Since $U^\pm(t)^*=U^\pm(-t)$, $K^\pm$ satisfies
$$
K^\pm(t,x,y)=\overline{K^\pm(-t,y,x)}.
$$
Therefore, if \eqref{theorem_IK_1_1} with $b^\pm=a^\pm$ holds for $\pm t\le0$, then  so does it for $\pm t\ge0$ and \eqref{dispersive} follows. 
\end{proof}

%remark
\begin{remark}
Although the case with $b^\pm=a^\pm$ is sufficient to obtain the homogeneous estimate \eqref{theorem_LP_2_1}, the general case will be used in the proof of the inhomogeneous estimate \eqref{theorem_LP_2_2}.
\end{remark}

The proof of  Theorem \ref{theorem_IK_1} basically follows the same line as that in Bouclet-Tzvetkov \cite[Sections 4 and 5]{BoTz2} in which long-range metric perturbations of the Laplacian (without potentials) were considered. However, since several propositions in the proof will be also used in the proof of inhomogeneous estimates \eqref{theorem_LP_2_2} which was not considered in \cite{BoTz2}, we give a complete proof. The proof is based on the so-called semiclassical Isozaki-Kitada (IK for short) parametrix for $e^{-i t H_h/h }\Op_h(a^\pm)$. Such a parametrix was originally introduced by Isozaki-Kitada \cite{IsKi1} in the non-semiclassical regime,  in order to show the  existence and asymptotic completeness of modified wave operators for the pair $(-\Delta,H)$ under the condition (H1). Since then, the IK parametrix has been extensively used in the study of long-range scattering theory in both non-semiclassical and semiclassical settings (see \cite{IsKi2,Rob2,DeSk} and references therein), and more recently, used in the proof of Strichartz estimates (see \cite{BoTz1,BoTz2,BoMi1} and references therein). To construct the IK parametrix, we basically follow the argument in \cite[Section 4]{Rob2}, \cite[Section 3]{BoTz1} and \cite[Section 4]{BoTz2}. For a basic theory of semiclassical FIOs, we refer to \cite{Rob1}.

Henceforth, although all the statements will be stated for both the outgoing ($+$) and incoming ($-$) cases, we will give the proofs for the outgoing case only, the incoming case being analogous. 

We begin with constructing the phase function. 

%{lemma}
\begin{lemma}
\label{lemma_IK_3}
Let $I\Subset (0,\infty)$ and $\sigma\in (-1,1)$. Then there exist $R_{\IK}\ge1$ and a family of smooth functions $S^\pm_{h,R}\in C^\infty(\R^{2n})$ with parameters $R\ge R_{\IK}$ and  $h\in (0,1]$ such that, for all $\alpha,\beta\in \Z_+^n$,
\begin{align*}
|\partial_x^\alpha\partial_\xi^\beta(S^\pm_{h,R}(x,\xi)-x\cdot\xi)|&\le C_{\alpha\beta}\<x\>^{1-\mu-|\alpha|},\\
|\partial_x^\alpha\partial_\xi^\beta(S^\pm_{h,R}(x,\xi)-x\cdot\xi)|&\le C_{\alpha\beta}\min\{\<x\>^{1-\mu-|\alpha|},R^{1-\mu-|\alpha|}\},\quad |\alpha|\ge1,
\end{align*}
where $C_{\alpha\beta}$ is independent of $h,x,\xi$ and $R$.  Moreover, $S^\pm_{h,R}$ satisfies
\begin{align}
\label{lemma_IK_3_2}
|\nabla_x S^\pm_{h,R}(x,\xi)|^2+V_h(x)=|\xi|^2,\quad (x,\xi)\in \Gamma^\pm(R,I,\sigma).
\end{align}
\end{lemma}

%proof
\begin{proof}
Let $\wtilde V_h\in C^\infty(\R^n;\R)$ be such that $\wtilde V_h\equiv V_h$ for $|x|\ge1$ and $\partial_x^\alpha \tilde V_h(x)=O(\<x\>^{-\mu-|\alpha|})$ on $\R^n$ uniformly in $h$. It then follows from  \cite[Proposition 4.1]{Rob2} (see also \cite[Proposition 3.2]{BoTz1}) that there exists $R_{\IK}\ge1$ such that, for any  $R\ge R_{\IK}$, one can construct $\wtilde S^+_{h,R}\in C^\infty(\R^{2n})$ satisfying $$|\partial_x^\alpha\partial_\xi^\beta(\wtilde S^+_{h,R}(x,\xi)-x\cdot\xi)|\le C_{\alpha}\<x\>^{-\mu-|\alpha|},\quad (x,\xi)\in \R^{2n}$$ and $|\nabla_x \wtilde S^+_{h,R}(x,\xi))|^2+\tilde V_h(x)=|\xi|^2$ on $\Gamma^+(R,I,\sigma)$. Let $R_1=R, I_1=I$ and $\sigma_1=\sigma$ and take  $I_1\Subset I_2$ and $-1<\sigma_1<\sigma_2<1$. Using the cut-off function $\chi_{1\to2}^+$ introduced above, we define $$S^+_{h,R}(x,\xi):=x\cdot\xi +\chi_{1\to2}^+(x,\xi)(\wtilde S^+_{h,R}(x,\xi)-x\cdot \xi).$$ Then it is not hard to check that $S^+_{h,R}$ satisfies the desired properties. 
\end{proof}

In what follows we drop the subscript $R$ and write $S^\pm_h=S_{h,R}^\pm$ for simplicity. 
To a symbol $a\in S^{0,-\infty}$ we associate $h$-FIOs $J^\pm_h(a):\S(\R^n)\to \S'(\R^n)$ defined by 
$$
J^\pm_h(a)f(x)=\frac{1}{(2\pi h)^n}\iint e^{i(S^\pm_h(x,\xi)-y\cdot\xi)/h}a(x,\xi)f(y)dyd\xi.
$$
Here we record a few basic properties of $J^+_h$ without proofs (see \cite{Rob1,Bouclet1} for details). For $a,b\in S^{0,-\infty}$, the kernel of $J^+_h(a)(J^+_h(b))^*$ is given by
\begin{align*}
\frac{1}{(2\pi h)^n}\int e^{i(S^+_h(x,\xi)-S^+_h(y,\xi))/h}a(x,\xi)\overline{b(y,\xi)}d\xi,
\end{align*}
where, by Lemma \ref{lemma_IK_3}, the phase function obeys
$$
S^+_h(x,\xi)-S^+_h(y,\xi)=(x-y)\cdot\xi_h(x,y,\xi),\ \partial_x^\alpha\partial_\xi^\beta \partial_y^\gamma(\xi_h(x,y,\xi)-\xi)=O(R^{-\mu}).
$$
Hence the standard Kuranishi's trick implies that, for sufficiently large $R$
$$J^+_h(a)(J^+_h(b))^*=\Op_h(c_h),\quad c_h\in S^{0,-\infty}.$$ In particular, $J^+_h(a)$ is bounded on $L^2$ uniformly in $h\in (0,1]$. Using the following property 
$$
J^+_h(a)x_j=J^+_h(a\partial_{\xi_j}S_h^+)-ih J^+_h(\partial_{\xi_j}a)
$$
and the fact $\<x\>^{-1}a\partial_{\xi_j}S_h^+\in S^{0,-\infty}$ and a duality argument, we also see that,  for any $\alpha\in \R$, $$\sup_{h\in (0,1]}\norm{\<x\>^\alpha J^+_h(a)\<x\>^{-\alpha }}\le C_\alpha<\infty.$$
Next we recall a microlocal parametrix construction of an elliptic $h$-FIO. Take $N\in \N$, $I_1\Subset I_2\Subset I_3\Subset (0,\infty)$ and $-1<\sigma_1<\sigma_2<\sigma_3<1$ arbitrarily. Let $R\ge R_{\IK}^3$ be large enough and consider a symbol $b=b_0+hb_1+\cdots+h^Nb_N$ with $b_j\in S^{-j,-\infty}$ supported in $\Gamma^+(R^{1/3},I_3,\sigma_3)$. Assume that $b_0$ is elliptic in such a way that
$$
b_0\gtrsim1\quad\text{on}\quad \Gamma^+(R^{1/2},I_2,\sigma_2).
$$
Then, for any $a\in S^{0,-\infty}$ supported in $\Gamma^+(R,I_1,\sigma_1)$, there exist a sequence $c_j\in S^{-j,-\infty}$ supported in $\Gamma^+(R^{1/2},I_2,\sigma_2)$ and $r_N\in S^{-N,-\infty}_N$ such that, with $c=c_0+hc_1+\cdots+h^Nc_N$,
$$
\Op_h(a)=J^+_h(b)J^+_h(c)^*+\Op_h(r_N). 
$$
Finally, we consider the composition $\Op_h(a)J_h^+(b)$ for $a,b\in S^{0,-\infty}$. By a direct calculation by means of Taylor's expansion, $\Op_h(a)J_h^+(b)$ is also an $h$-FIO $J_h^+(c_h)$ with the amplitude
$$
c_h(x,\xi)=\frac{1}{(2\pi h)^n}\iint e^{i[S^+_h(y,\xi)-S^+_h(x,\xi)+(x-y)\cdot\eta]/h}a(x,\eta)b(y,\xi)dyd\eta
$$
having the following asymptotic expansion
\begin{align}
\label{composition}
c_h(x,\xi)-\sum_{|\alpha|<N}\frac{i^{-|\alpha|}h^{|\alpha|}}{\alpha!}(\partial_\xi^\alpha a)(x,\nabla_xS^+_h(x,\xi))\partial_y^\alpha \left[e^{i\Psi_h(x,y,\xi)}b(y,\xi)\right]\Big|_{y=x}\in h^NS^{-N,-\infty}
\end{align}
for any $N\ge0$, where $\Psi_h(x,y,\xi)=S^+_h(y,\xi)-S^+_h(x,\xi)+(x-y)\cdot\nabla_xS^+_h(x,\xi)$. 
 In particular, since $\nabla_xS^+_h(x,\xi)=\xi+O(R^{-\mu})$, for any neighborhood $U$ of $\supp a\cap \supp b$, there exists $R'>0$ such that, modulo an error term in $S^{-N,-\infty}_N$,  $c_h$ is supported in $U$ provided that $R\ge R'$. 

%proposition
\begin{proposition}[Semiclassical Isozaki-Kitada parametrix]
\label{proposition_IK_3}
Let $I_1\Subset I_{2}\Subset  I_3\Subset I_4\Subset (0,\infty)$ be relatively compact open intervals and $-1<\sigma_1<\sigma_{2}<\sigma_3<\sigma_4<1$ so that $$\Gamma^\pm(R^{1/j},I_j,\sigma_j)\subset \Gamma^\pm(R^{1/(j+1)},I_{j+1},\sigma_{j+1}),\quad 1\le j\le4.$$ 
For $R\ge R_{\IK}^4$ large enough, any $a^\pm\in S^{0,-\infty}$ supported in $ \Gamma^\pm(R,I_1,\sigma_1)$ and any $N\ge0$, we have
\begin{align*}
e^{-itH_h/h}\Op_h(a^\pm)
&=J^\pm_h(c^\pm)e^{ith\Delta}J^\pm_h(d^\pm)^*+\sum_{j=1}^4Q_j^\pm(t,h),
\end{align*}
where the remainder terms $Q_j^\pm(t,h)$, $1\le j\le 4$, are given by 
\begin{align*}
Q_1^\pm(t,h)&=h^Ne^{-itH_h/h}\Op_h(r_{1}^\pm),\\
Q^\pm_2(t,h)&=-ih^{N-1}\int_0^t e^{-i(t-s)H_h/h}J^\pm_h(r_{2}^\pm)e^{ish\Delta}J^\pm_h(d^\pm)^*ds,\\
Q^\pm_3(t,h)&=-\frac ih\int_0^t e^{-i(t-s)H_h/h}\Op_h(\wtilde e_{\com}^\pm)J^\pm_h(e_{\com}^\pm)e^{ish\Delta}J^\pm_h(d^\pm)^*ds,\\
Q^\pm_4(t,h)&=-\frac ih\int_0^t e^{-i(t-s)H_h/h}\Op_h(\wtilde e^\pm)J^\pm_h(e^\pm)e^{ish\Delta}J^\pm_h(d^\pm)^*ds
\end{align*}
and $S^\pm_h$ is given by Lemma \ref{lemma_IK_3} with $R,I,\sigma$ replaced by $R^{1/4},I_4,\sigma_4$. Moreover, the amplitudes satisfy the following properties:
\begin{itemize}
\item $c^\pm,d^\pm,e^\pm,\wtilde e^\pm\in S^{0,-\infty}$, $r_{1}^\pm,r_{2}^\pm\in S^{-N,\infty}$ and $e_{\com}^\pm\in C_0^\infty(\R^{2n})$;
\item $\supp c^\pm\subset \Gamma^\pm(R^{1/3},I_3,\sigma_3)$, $\supp d^\pm \subset \Gamma^\pm(R^{1/2},I_2,\sigma_2)$;
\item $\supp e_{\com}^\pm,\supp \wtilde e_{\com}^\pm\subset\Gamma^\pm(R^{1/4},I_4,\sigma_4)\cap\{|x|\le R^{4/9}\}$;
\item $\supp e^\pm,\supp \wtilde e^\pm\subset \Gamma^\pm(R^{1/4},I_4,\sigma_4)\cap \Gamma^\mp(R^{1/4},I_4,-\wtilde\sigma)$ with some $\wtilde \sigma\in (\sigma_2,\sigma_3)$.
\end{itemize}
\end{proposition}

%proof
\begin{proof}
This proposition is basically known (see \cite{Rob2} and also \cite[Proposition 4.2]{BoTz2}). Hence, we only give a brief outline of the proof. Given a symbol $c\in S^{0,-\infty}$, we obtain, by computing $\int_0^t\frac{d}{ds}(e^{isH_h/h}J^+_h(c)e^{ish\Delta})ds$ in two ways, the following Duhamel formula
\begin{align}
\label{proof_proposition_IK_3_1}
e^{-itH_h/h}J^+_h(c)=J^+_h(c)e^{ith\Delta}-\frac ih\int_0^t e^{-i(t-s)H_h/h}\Big(H_hJ^+_h(c)+ J^+_h(c)h^2\Delta\Big)e^{ish\Delta}ds.
\end{align}
Here $H_hJ^+_h(c)+ J^+_h(c)h^2\Delta$ is an $h$-FIO with the phase $S^+_h$ and the amplitude
\begin{align*}
&e^{-\frac ih S^+_h}H_h(e^{\frac ih S^+_h}c)-\xi^2\\
&=(|\nabla_x S^+_h|^2+V_h-|\xi|^2)c+ih\{-2(\nabla_xS^+_h)\cdot \nabla_xc+(\Delta_x S^+_h)c\}-h^2\Delta_x c.
\end{align*}
With Lemma \ref{lemma_IK_3} at hand,  one then can construct, by means of Hamilton-Jacobi theory,  $\tilde c_j^+\in S^{-j,\infty}$ supported in $\Gamma^+(R^{1/3},I_3,\sigma_3)$ such that $\wtilde c^+=\tilde c_0^++h\tilde c_1^++\cdots+h^{N}\tilde c_N^+$ satisfies
\begin{align}
\label{proof_proposition_IK_3_2}
r_0^+:=ih\{-2(\nabla_xS^+_h)\cdot \nabla_x\tilde c^++(\Delta_x S^+_h)\tilde c^+\}-h^2\Delta_x \tilde c^+\in h^N S^{-N,-\infty}.
\end{align}
Let $\chi^+\in S^{0,-\infty}$ be such that $\chi^+\equiv1$ on $\Gamma^+(R^{1/3},I_3,\sigma_3)$, $\chi^+$ is supported in a sufficiently small neighborhood of $\Gamma^+(R^{1/3},I_3,\sigma_3)$, and is of the form
$$
\chi^+(x,\xi)=\rho_1(x/R^{1/3})\rho_2(\xi)\rho_3(\cos(x,\xi))
$$ 
with some $\rho_1\in C^\infty(\R^n)$, $\rho_2\in C_0^\infty(\R^{n})$ and $\rho_3\in C^\infty(\R)$ such that $\rho_1(x)=1$ if $|x|\ge1$, $\rho_2(\xi)=1$ if $|\xi|^2\in I_3$ and $\rho_3(s)=1$ if $s>-\sigma_3$. 
We set $
c^+:=\chi^+\tilde c^+
\in S^{0,-\infty}$ which is well defined on $\R^{2n}$ and, by virtue of \eqref{lemma_IK_3_2}, satisfies
\begin{align}
\label{proof_proposition_IK_3_4}
(|\nabla_x S^+_h|^2+V_h-|\xi|^2)c^+\equiv0,\quad x,\xi\in \R^{2n}. 
\end{align}
Moreover, by \eqref{proof_proposition_IK_3_2} and \eqref{proof_proposition_IK_3_4}, we have
\begin{align*}
ih(-2(\nabla_xS^+_h)\cdot \nabla_xc^++(\Delta_x S^+_h)c^+)-h^2\Delta_x c^+=h^Nr_2'+e_{\com}^++e^+,
\end{align*}
where $r_2'=\chi^+r_0^+$, and $e_{\com}^+$ (resp. $e^+$) consists of the parts for which at least one derivative in $x$ falls on the factor $\rho_1$ (resp. $\rho_3$). Then, it is easy to see that $|x|\le R^{1/3}$ on $\supp e^+_{\com}$ and that $\cos(x,\xi)\le -\sigma_3$ on $\supp e^+$ which show that $e^+$ and $e_{\com}^+$ satisfy desired properties in the statement. By using the formula \eqref{composition}, one can also construct $\wtilde e^+_{\com},\wtilde e^+\in S^{0,-\infty}$ such that $\wtilde e^+_{\com}\equiv1$ (resp. $\wtilde e^+\equiv1$) on a sufficiently small neighborhood of $e^+_{\com}$ (resp. $\supp e^+$) and that 
\begin{align*}
\supp \wtilde e^+_{\com}&\subset \Gamma^+(R^{1/4},I_4,\sigma_4)\cap\{|x|<R^{4/9}\},\\
\supp \wtilde e^+&\subset \Gamma^+(R^{1/4},I_4,\sigma_4)\cap \Gamma^-(R^{1/4},I_4,-\wtilde \sigma)
\end{align*} 
with some $\sigma_2<\wtilde \sigma<\sigma_3$ and that
\begin{align*}
J_h^+(e^+_{\com})&=\Op_h(\wtilde e^+_{\com})J_h^+(e^+_{\com})+h^NJ_h^+(r_2''),\\
J_h^+(e^+)&=\Op_h(\wtilde e^+)J_h^+(e^+)+h^NJ_h^+(r_2''')
\end{align*}
with some $r_2'',r_2'''\in S^{-N,-\infty}$. If we define $r_2^+=r_2'+r_2''+r_2'''$ then $r_2^+$ satisfies the desired properties. By construction, we have
\begin{align*}
H_hJ^+_h(c^+)+ J^+_h(c^+) h^2\Delta=h^N J_h^+(r_2^+)+\Op_h(\wtilde e^+_{\com})J_h^+(e_{\com}^+)+\Op_h(\wtilde e^+)J_h^+(e^+).
\end{align*}
Finally, since  $ c_0\gtrsim 1$ on $\Gamma^+(R^{1/3},I_3,\sigma_3)$ by construction, one can also construct $d^+\in S^{0,-\infty}$ and $r_1^+\in S^{-N,-\infty}$ such that $d^+$ is supported in $\Gamma^+(R^{1/2},I_2,\sigma_2)$ and $$\Op_h(a^+)=J_h^+(c^+)J_h^+(d^+)^*-h^N\Op_h(r^+_1).$$
Multiplying \eqref{proof_proposition_IK_3_1} with $c=c^+$ by $J_h^+(d^+)^*$ from the right hand side, we complete the proof. 
\end{proof}

Now we give the proof of  Theorem \ref{theorem_IK_1} which is divided into a series of propositions. 
%{lemma}
\begin{proposition}
\label{proposition_IK_4}
Let $I\Subset (0,\infty)$ be a relatively compact interval. For sufficiently large $R\ge R_{\IK}$ and any symbols $a,b\in S^{0,-\infty}$ satisfying $|\xi|^2\in I$ on $\supp a$ and on $\supp b$, we have
\begin{align}
\label{proposition_IK_4_1}
\norm{J^\pm_h(a)e^{ith\Delta}J^\pm_h(b)^*}_{1\to \infty}\lesssim \min\{|th|^{-n/2},h^{-n}\},\quad t\in \R,\ h\in (0,1].
\end{align}
\end{proposition}

%proof
\begin{proof}
The kernel $I_{a,b}(t,x,y)$ of $J^+_h(a)e^{ith\Delta}J^+_h(b)^*$ is given by
$$
I_{a,b}(t,x,y)=\frac{1}{(2\pi h)^n}\int e^{i\Phi^+(t,x,\xi)/h}a(x,\xi)\overline{b(y,\xi)}d\xi
$$
where
$
\Phi^+(t,x,\xi)=-t|\xi|^2+{S^+_h(x,\xi)-S^+_h(y,\xi)}
$. By virtue of Lemma \ref{lemma_IK_3}, $\Phi^+$ satisfies
\begin{align}
\label{proof_proposition_IK_4_1}
{\nabla_\xi\Phi^+(t,x,\xi)}=-2t\xi+\Big(1+Q^+(x,y,\xi)\Big)({x-y})
\end{align}
with some $Q^+$ satisfying, for all $\alpha,\beta,\gamma\in \Z_+^n$,
$$
|\partial_x^\alpha\partial_y^\beta\partial_\xi^\gamma Q^+(x,y,\xi)|\le C_{\alpha\beta\gamma}R^{-\mu},\quad x,y,\xi\in \R^n,\ h\in (0,1].
$$
Since \eqref{proposition_IK_4_1} is trivial if $|t|\lesssim h$, we may assume $|t|\gtrsim h$. Then, for sufficiently large $R$, there exists $C_0>0$ such that if $|(x-y)/t|\ge C_0$ then
$$
\left|\frac{\nabla_\xi\Phi^+(t,x,\xi)}{t}\right|\gtrsim \left|\frac{x-y}{t}\right|\ge C_0
$$
and we have, by integration by parts in $\xi$, that
$$
|I(t,x,y)|\le C_Nh^{-n}|th^{-1}|^{-N}\lesssim |th|^{-n/2}
$$
for all $|t|\gtrsim h$ provided $N\ge n/2$. Otherwise, we have that
$$
\frac{\nabla_\xi^2\Phi^+(t,x,\xi)}{t}=-2\Id+O(R^{-\mu}),\quad \frac{\partial_x^\alpha \partial_y^\beta \partial_\xi^\gamma \Phi^+(t,x,\xi)}{t}=O(1)
$$
for all $\alpha,\beta,\gamma\in \Z_+^n$ and $x,y\in \R^n$, $h\in (0,1]$, $|\xi|^2\in I$. Therefore, choosing $R\ge R_{\IK}$ sufficiently large if necessary, we can apply the stationary phase theorem to obtain
$$
|I_{a,b}(t,x,y)|\lesssim h^{-n}|th^{-1}|^{-n/2}=|th|^{-n/2}
$$
for $|t|\gtrsim h$, which completes the proof. 
\end{proof}

%{lemma}
\begin{proposition}
\label{proposition_IK_5}
Let $d^\pm$, $r_2^\pm$, $e_{\com}^\pm$ and $e^\pm$  be as in Proposition \ref{proposition_IK_3}. Let $R\ge1$ be large enough. Then for all $s\in \R$ and sufficiently large $N\ge1$, we have
\begin{align}
\label{proposition_IK_5_1}
\norm{\<D\>^{s}\<x\>^{N/4}J^\pm_h(r_2^\pm)e^{ith\Delta}J^\pm_h(d^\pm)^*\<x\>^{N/8}\<D\>^{s}}\lesssim h^{-n-2s}\<t\>^{-N/8}
\end{align}
uniformly in $\pm t\ge0$ and $h\in (0,1]$. Moreover, for any $M\ge0$ and $s\in \R$,
\begin{align}
\label{proposition_IK_5_2}
\norm{\<D\>^{s}\<x\>^{M}J^\pm_h(\chi^\pm)e^{ith\Delta}J^\pm_h(d^\pm)^*\<x\>^{M}\<D\>^{s}}\lesssim h^{M-2s}\<t\>^{-M},
\end{align}
uniformly in $\pm t\ge0$ and $h\in (0,1]$, where $\chi^\pm=e_{\com}^\pm$ or $e^\pm$. 
\end{proposition}

To prove this proposition, we need the following elementary fact (see \cite[Lemma 4.1]{BoTz2}).

%{lemma}
\begin{lemma}
\label{lemma_IK_6}Let $x,y,\xi\in \R^n\setminus\{0\}$, $-1<\sigma<1$ and $-1<\sigma_2<\sigma_3<1$. 
\begin{itemize}
\item[{\rm (1)}] If $\pm \cos(x,\xi)>-\sigma$ and $\pm t\ge0$, then 
$\pm\cos(x+2t\xi,\xi)>-\sigma $ and $|x+2t\xi|\gtrsim |x|+t|\xi|$. 
\item[{\rm (2)}] If $\pm\cos(x,\xi)\le-\sigma_3$ and $\pm \cos(y,\xi)>-\sigma_2$ then $|x-y|\gtrsim |x|+|y|$. 
\end{itemize}
Here implicit constants are independent of $x,y$ and $t$.
\end{lemma}

%proof
\begin{proof}[Proof of Proposition \ref{proposition_IK_5}]
Let $I_{b,d^+}$ be the kernel of $J^+_h(b)e^{ith\Delta}J^+_h(d^+)^*$. 
By \eqref{lemma_IK_6} (1), one has
\begin{align}
\label{proof_proposition_IK_5_1}
|y|\ge R^{1/2},\quad |y+2t\xi|\gtrsim |y|+t
\end{align}
for $t\ge0$ on the support of $d^+(y,\xi)$. We first consider \eqref{proposition_IK_5_1} with $s=0$. Take $\chi\in C_0^\infty(\R^n)$ supported in a unit ball $\{x\ |\ |x|\le1\}$ and decompose $I_{r_2^+,d^+}$ as
\begin{align*}
I_{r_2^+,d^+}
&=\frac{1}{(2\pi h)^n}\int e^{i\Phi^+(t,x,\xi)/h}\chi\Big(\nabla_\xi \Phi^+(t,x,\xi)\Big)r_2^+(x,\xi)\overline{d^+(y,\xi)}d\xi\\
&+\frac{1}{(2\pi h)^n}\int e^{i\Phi^+(t,x,\xi)/h}(1-\chi)\Big(\nabla_\xi \Phi^+(t,x,\xi)\Big)r_2^+(x,\xi)\overline{d^+(y,\xi)}d\xi\\
&=:I^{(1)}+I^{(2)}. 
\end{align*}
For the first part $I^{(1)}$, we know by \eqref{proof_proposition_IK_4_1} that
$$
1\ge |\nabla_\xi \Phi^+|\ge |y+2t\xi|-|x|-C(1+|x|+|y|)^{1-\mu}
$$
with some universal $C>0$, which, together with \eqref{proof_proposition_IK_5_1} implies
$
\<x\>^{-1}\lesssim (1+|y|+|t|)^{-1}
$
on the support of the amplitude of $I^{(1)}$ provided $R$ is large enough. This bound implies
\begin{align}
\label{proof_proposition_IK_5_2}
|I^{(1)}(t,x,y)|\lesssim h^{-n}\<x\>^{-N/2}\<y\>^{-N/4}\<t\>^{-N/4}
\end{align}
since $r^+_2\in S^{-N,-\infty}$. On the other hand, when $|\nabla_\xi \Phi^+|\ge1$, we have
$$
|y|+t\lesssim |y+2t\xi|=\left|\nabla_\xi \Phi^+(t,x,\xi)-x+O((1+|x|+|y|)^{1-\mu})\right|\lesssim |\nabla_\xi \Phi^+|\<x\>+C|y|^{1-\mu}
$$
which yields $|\nabla_\xi \Phi^+|^{-1}\lesssim \<x\>\<y\>^{-1/2}\<t\>^{-1/2}$ if $R$ is large enough. Integrating by parts with respect to the operator
$$
L_{\Phi^+}=\frac{h}{i|\nabla_\xi \Phi^+|^2}(\nabla_\xi\Phi^+)\cdot\nabla_\xi,
$$ we thus obtain that
\begin{align}
\label{proof_proposition_IK_5_3}
|I^{(2)}(t,x,y)|\lesssim h^{N/2}\<x\>^{-N/2}\<y\>^{-N/4}\<t\>^{-N/4},\quad t\ge0.
\end{align}
Then \eqref{proposition_IK_5_1} with $s=0$ and $b=r_2^+$ easily follows from \eqref{proof_proposition_IK_5_2} and \eqref{proof_proposition_IK_5_3} provided $N$ is large enough. When $s\in \N$, taking into account the bound
$$
\nabla_x\Phi^+(t,x,\xi)=\xi+O(\<x\>^{-\mu})
$$
and the formula $J_h^+(a)\partial_{x_j}=ih^{-1}J_h^+(\xi_ja)$, we obtain \eqref{proposition_IK_5_1} by the same argument. For general $s\ge0$, \eqref{proposition_IK_5_1} with $b=r_2^+$ then follows from the case with $s\in \N\cup\{0\}$ and  complex interpolation. 

In case of \eqref{proposition_IK_5_2} with $\chi^+=e_{\com}^+$, since $|x|\le R^{1/3}$ on $\supp e_{\com}^+$ and $|y|\ge R^{1/2}$ on $\supp d^+$, we know by Lemma \ref{lemma_IK_6} (1) that $|x-y-2t\xi|\gtrsim 1+|x|+|y|+t$ and thus
$$
|\nabla_\xi\Phi^+|=|x-y-2t\xi+O((1+|x|+|y|)^{1-\mu})|\gtrsim 1+|x|+|y|+t
$$
for $t\ge0$ on the support of the amplitude of $I_{e_{\com}^+,d^+}$ provided that $R$ is large enough. Therefore, integrating by parts with respect to $L_{\Phi^+}$ shows
$$
|I_{e_{\com}^+,d^+}(t,x,y)|\le h^{M}(1+|x|+|y|+t)^{-M},\quad t\ge0,
$$
for all $M\ge0$ and \eqref{proposition_IK_5_2} follows. Finally, to obtain \eqref{proposition_IK_5_2} with $\chi^+=e^+$, we use the support properties of $e^+$ and $d^+$ which yield that
$$
\cos(x,\xi)< -\wtilde \sigma<-\sigma_2<\cos(y+2t\xi,\xi)
$$
for all $t\ge0$. Lemma \ref{lemma_IK_6} (2) then implies
$$
|x-y-2t\xi|\gtrsim 1+|x|+|y|+|t|
$$
for all $t\ge0$ on the support of the amplitude of $I_{e^+,d^+}$ which gives us \eqref{proposition_IK_5_1} as above. \end{proof}

Using Propositions \ref{proposition_smoothing_2_3}, \ref{proposition_IK_4} and \ref{proposition_IK_5}  we have the following propagation estimates. 

%{lemma}
\begin{proposition}
\label{proposition_IK_7}
Let $R\ge1$ be large enough and $b^\pm\in S^{0,-\infty}$ supported in $\Gamma^\pm(R,I,\sigma)$. Then the following statements are satisfied with implicit constants independent of $t$ and $h\in (0,1]$. 
\begin{itemize}
\item For sufficiently large integer $N$ and all $s\ge0$, 
\begin{align}
\label{proposition_IK_7_1}
\norm{\<x\>^{-N}\varphi(H_h)e^{-itH_h/h}\Op_h(b^\pm)\<D\>^s}\lesssim h^{-n-s-1}\<t\>^{-N/8},\quad \pm t\ge 0. 
\end{align}
\item For any $\chi\in C_0^\infty (\R^n)$ supported in a unit ball $\{|x|<1\}$ and all $M,s\ge0$, 
\begin{align}
\label{proposition_IK_7_2}
\norm{\chi(x/R^{1/3})\varphi(H_h)e^{-itH_h/h}\Op_h(b^\pm)\<D\>^s}\lesssim h^{M-s}\<t\>^{-M},\quad \pm t\ge 0. 
\end{align}
\item If $\chi^\mp\in S^{0,-\infty}$ is supported in $\Gamma^\mp(R^{1/4},I_4,-\wtilde\sigma)$ with some $\wtilde \sigma>\sigma$, then for all $M,s\ge0$
\begin{align}
\label{proposition_IK_7_3}
\norm{\<D\>^s\Op_h(\chi^\mp)^*\varphi(H_h)e^{-itH_h/h}\Op_h(b^\pm)\<x\>^M\<D\>^s}\lesssim h^{M-2s}\<t\>^{-M},\quad \pm t\ge 0. 
\end{align}
\end{itemize}
\end{proposition}

%proof
\begin{proof}We may prove the proposition for the case $s=0$  only. The case $s>0$ follows from the case $s=0$ since $b^\pm\<\xi\>^s,\chi^\mp\<\xi\>^s\in S^{0,-\infty}$ and $\norm{\<hD\>^{-s}\<D\>^s}\lesssim h^{-s}$. 

In order to prove \eqref{proposition_IK_7_1}, choosing $N\gg M$ large enough , we decompose the operator in \eqref{proposition_IK_7_1} into corresponding five parts
$$
\<x\>^{-N}\varphi(H_h)J_h^+(c^+)e^{ith\Delta}J_h^+(d^+)^*,\quad
\<x\>^{-N}\varphi(H_h)Q_j^+(t,h),\quad 1\le j\le 4,
$$
where $c^+,d^+\in S^{0,-\infty}$ are given by Proposition \ref{proposition_IK_3}.  We shall show that
\begin{align}
\label{proof_proposition_IK_7_1}
\norm{\<x\>^{-N}\varphi(H_h)J_h^+(c^+)e^{ith\Delta}J_h^+(d^+)^*}&\lesssim h^{-n-1}\<t\>^{-N/8},\\
\label{proof_proposition_IK_7_3}
\norm{\<x\>^{-N}\varphi(H_h)Q_j^+(t,h)}&\lesssim h^{N-n-2}\<t\>^{-N/8},
\end{align}
for $t\ge0$ and $h\in (0,1]$ and $1\le j\le4$ . \eqref{proof_proposition_IK_7_1} is a direct consequence of \eqref{proposition_IK_5_1} since $\<x\>^{-N}\varphi(H_h)\<x\>^N$ is bounded on $L^2$ uniformly in $h\in (0,1]$ by Proposition \ref{proposition_functional_3}. For the part 
$$
\<x\>^{-N}\varphi(H_h)Q_1^+(t,h)=h^N\<x\>^{-N}\varphi(H_h)e^{-itH_h/h}\<x\>^{-N}\<x\>^N\Op_h(r_1^+),
$$ 
we use \eqref{proposition_smoothing_2_3} and the fact $\<x\>^Nr_1^+\in S^{0,-\infty}$ to obtain \eqref{proof_proposition_IK_7_3}. To deal with the term
$$
\<x\>^{-N}\varphi(H_h)Q_2^+(t,h)=-ih^{N-1}\int_0^t \<x\>^{-N}\varphi(H_h)e^{-i(t-s)H_h/h}J^+_h(r_2^+)e^{irh\Delta}J^+_h(d^+)^*ds,
$$
we again use \eqref{proposition_smoothing_2_3} and \eqref{proposition_IK_5_1} to obtain for $s\le t$ that
\begin{align*}
\norm{\<x\>^{-N}\varphi(H_h)e^{-i(t-s)H_h/h}\<x\>^{-N/4}}&\lesssim h^{-1}\<t-s\>^{-N/8},\\
\norm{\<x\>^{N/4}J^+_h(r_2^+)e^{ish\Delta}J^+_h(d^+)^*}&\lesssim h^{-n}\<s\>^{-N/8},
\end{align*}
which imply \eqref{proof_proposition_IK_7_3} since
\begin{equation}
\begin{aligned}
\label{proof_proposition_IK_7_4}
\int_0^t\<t-s\>^{-\frac N8}\<s\>^{-\frac N8}ds
\lesssim \int_0^{t/2}\<t\>^{-\frac N8}\<s\>^{-\frac N8}ds+\int_{t/2}^t\<t-s\>^{-\frac N8}\<t\>^{-\frac N8}ds
\lesssim \<t\>^{-\frac N8}.
\end{aligned}
\end{equation}
The estimate \eqref{proof_proposition_IK_7_3} for the terms $\<x\>^{-N}\varphi(H_h)Q_j^+(t,h)$, $j=3,4$, can be also verified  similarly by means of \eqref{proposition_smoothing_2_3} and \eqref{proposition_IK_5_2} (instead of \eqref{proposition_IK_5_1}). This completes the proof of \eqref{proposition_IK_7_1}. 

We next show  \eqref{proposition_IK_7_2} which, as above, follows from the following estimates
\begin{align*}
\norm{\chi(x/R^{1/3})\varphi(H_h)J_h^+(c^+)e^{ith\Delta}J_h^+(d^+)^*\<x\>^M}&\lesssim h^M\<t\>^{-M},\\
\norm{\chi(x/R^{1/3})\varphi(H_h)Q_j(t,h)\<x\>^M}&\lesssim h^M\<t\>^{-M},\ 1\le j\le 4.
\end{align*}
The latter bounds for the remainder terms follow from the same argument as that for \eqref{proof_proposition_IK_7_3}. To deal with the main term, we observe from the support property $|x|>R^{1/4}$ on $\supp c^+$ that
$$
\varphi(H_h)J_h^+(c^+)=\varphi(H_h)(1-\wtilde \chi)(x/R^{1/4})J_h^+(c^+)
$$
with some $\wtilde \chi\in C_0^\infty(\R^n)$. By Proposition \ref{proposition_functional_3}, we then see that $\chi(x/R^{1/3})\varphi(H_h)(1-\wtilde \chi)(x/R^{1/4})$ is a sum of $\Op(\wtilde c^+)$ and an error term $O_{L^2}(h^N\<x\>^{-N})$, where $\wtilde c^+\in S^{0,-\infty}$ is supported in $\Gamma^+(R^{1/4},I_4,\sigma_4)\cap\{|x|<R^{1/3}\}$. Then the error term can be estimated by using \eqref{proposition_IK_5_1}. Hence it suffices to deal with the operator
$
\Op(\wtilde c^+)J_h^+(c^+)e^{ith\Delta}J_h^+(d^+)^*
$. 
Moreover, by virtue of \eqref{composition}, we may replace $\Op(\wtilde c^+)J_h^+(c^+)$ by $J_h^+(c^+_0)$ with some $c^+_0\in S^{-\infty,-\infty}$ supported in $\Gamma^+(R^{1/4},I_4,\sigma_4)\cap\{|x|<R^{1/3}\}$ without loss of generality. Then, since $|x|<R^{1/3}$ on $\supp c_0^+$ and $|x|>R^{1/2}$ on $\supp d^+$, the same argument as that in the proof of \eqref{proposition_IK_5_2} shows 
$$
\norm{J_h^+(c^+_0)e^{ith\Delta}J_h^+(d^+)^*}\lesssim h^M\<t\>^{-M}
$$
for all $M\ge0$ uniformly in $t\ge0$, which completes the proof of \eqref{proposition_IK_7_2}. 

In order to prove \eqref{proposition_IK_7_3}, by a similar argument as above, it suffices to show that
\begin{align}
\label{proof_proposition_IK_7_5}
\norm{\Op_h(\chi^-)^*\varphi(H_h)J_h^+(c^+)e^{ith\Delta}J_h^+(d^+)^*\<x\>^M}&\lesssim h^M\<t\>^{-M},\\
\label{proof_proposition_IK_7_6}
\norm{\Op_h(\chi^-)^*\varphi(H_h)Q_j(t,h)\<x\>^M}&\lesssim h^M\<t\>^{-M}.
\end{align}
Using \eqref{proposition_IK_7_1} with $b^-=\chi^-$, we see that
\begin{align*}
\norm{\Op_h(\chi^-)^*\varphi(H_h)e^{-itH_h/h}\<x\>^{-N}}
=\norm{\<x\>^{-N}\varphi(H_h)e^{itH_h/h}\Op_h(\chi^-)}
\lesssim h^{-n}\<t\>^{N/8}
\end{align*}
for $-t\le0$, that is $t\ge0$. This bound, together with \eqref{proposition_IK_5_1}, \eqref{proposition_IK_5_2} and \eqref{proof_proposition_IK_7_3}, implies \eqref{proof_proposition_IK_7_6}. To deal with the main term, we recall that $d^+$ is supported in $\Gamma^+(R^{1/2},I_2,\sigma_2)$, where one can choose $\sigma_2$ sufficiently close to $\sigma$ so that $\sigma<\sigma_2<\wtilde \sigma$. Then \eqref{proof_proposition_IK_7_5} can be verified by essentially the same argument as above, as follows. At first, by means of Proposition \ref{proposition_functional_3} and \eqref{composition}, we may replace without loss of generality $\Op_h(\chi^-)^*\varphi(H_h)J_h^+(c^+)$ by $J_h^+(\wtilde \chi^-)$ with some $\wtilde \chi^-\in S^{0,-\infty}$ supported in $\supp \chi^-$. Next, by means of Lemma \ref{lemma_IK_6} (1), we have the following property
$$\cos(x,\xi)<-\wtilde \sigma<-\sigma_2<\cos (y+2t\xi,\xi)$$
for all $t\ge0$ on the support  of the amplitude of $J_h^+(\wtilde \chi^-)e^{ith\Delta}J_h^+(d^+)^*\<x\>^{M}$. By virtue of Lemma \ref{lemma_IK_6} (2), this support property implies $|x-y-2t\xi|\gtrsim |x|+|y|+t$ for all $t\ge0$. Then the same argument as that in the proof of Proposition \ref{proposition_IK_5} yields 
$$
\norm{J_h^+(\wtilde \chi^-)e^{ith\Delta}J_h^+(d^+)^*\<x\>^{M}}\lesssim h^M\<t\>^{-M},\quad t\ge0,
$$
for all $M\ge0$, which completes the proof. 
\end{proof}

We are now ready to show Theorem \ref{theorem_IK_1}. 
%proof
\begin{proof}[Proof of Theorem \ref{theorem_IK_1}]Let $t\le0$. As before, by means of  Proposition \ref{proposition_IK_3} we decompose the operator $\Op_h(b^+)^*\varphi(H_h)e^{-itH_h/h}\Op_h(a^+)$ into five parts
\begin{align*}
U_0(t)&:=\Op_h(a^+)^*\varphi(H_h) J^+_h(c^+)e^{ith\Delta}J^+_h(d^+)^*,\\
U_j(t)&:=\Op_h(a^+)^*\varphi(H_h)Q_j^+(t,h),\quad 1\le j\le4. 
\end{align*}

To deal with the main term $U_0(t)$, thanks to Proposition \ref{proposition_IK_4}, it is enough to check that $\Op_h(a^+)^*\varphi(H_h)$ is bounded on $L^\infty$ uniformly in $h$. To this end, taking into account the fact that $\Op_h(a^+)=\Op_h(a^+)^*\rho(x/R)$ with some $\rho\in C_0^\infty$ supported away from the origin, we use Proposition \ref{proposition_functional_3} to write $\rho(x/R)\varphi(H_h)=\Op_h(b_h)^*+Q_h$ with some $b_h\in S^{0,-\infty}$ and $Q_h$ satisfying
$
\norm{Q_h\<x\>^N}\lesssim h^N
$
for $N>n/2$. Then $\norm{\Op_h(a^+)^*\Op_h(b_h)^*}_{\infty\to\infty}\lesssim 1$ by \eqref{PDO_1}. Moreover, \eqref{PDO_1} and the embedding $L^\infty(\R^n)\subset \<x\>^{N}L^2(\R^n)$ show that
$$
\norm{\Op_h(a^+)^*Q_hf}_{\infty}\le \norm{\Op_h(a^+)^*}_{2\to \infty}\norm{Q_h\<x\>^N}\norm{\<x\>^{-N}f}_{L^2}\lesssim h^{N-n/2}\norm{f}_\infty\lesssim \norm{f}_\infty.
$$ 

To deal with $U_1(t)$, we take $N\ge 4n$ and use \eqref{PDO_1}, the dual estimate of \eqref{proposition_IK_7_1} to see that
\begin{align*}
\norm{U_1(t)^*f}_{\infty}
&\lesssim h^{N-n/2}\norm{\Op_h(a^+)^*\varphi(H_h)e^{-it H_h/h}\Op_h(r_1^+)f}_{2}\\
&\lesssim h^{N-n}\<t\>^{-N/8}\norm{\<x\>^N\Op_h(r_1^+)f}_{2}\\
&\lesssim \<t\>^{-n/2}\norm{f}_{1}.
\end{align*}

To deal with $U_{2}(t)$ which is of the form
$$
-ih^{N-1}\int_0^t \Op_h(a^+)^*\varphi(H_h)e^{-i(t-s)H_h/h}J^+_h(r_2^+)e^{ish\Delta}J^+_h(d^+)^*ds
$$
we also use the dual estimate of \eqref{proposition_IK_7_1}, Proposition \ref{proposition_IK_5} and Sobolev's embedding to see that
\begin{align*}
\norm{\Op_h(a^+)^*\varphi(H_h)e^{-i(t-s)H_h/h}\<x\>^{-N}}_{\infty\to \infty}
&\lesssim h^{-3n/2-1}\<t-s\>^{-N/8},\\
\norm{\<x\>^{N}J_h^+(r_2^+)e^{ish\Delta}J_h^+(d^+)^*}_{1\to \infty}
&\lesssim \min\{h^{-n},|s h|^{-n/2}\}
\end{align*}
for $t\le s$. Thus, choosing $N\ge 4n+2$ large enough, we obtain
$$
\norm{U_2(t)}_{1\to\infty}\lesssim h^{N-3n/2-2}\left|\int_0^t\<t-s\>^{-n}\min\{h^{-n},|s h|^{-n/2}\}
ds\right|\lesssim \min\{h^{-n},|th|^{-n/2}\},\quad t\le0. 
$$
Taking the embedding $\<x\>^{-n/2-\ep}L^\infty(\R^n)\subset L^2(\R^n)$ into account, by using the dual estimates of \eqref{proposition_IK_7_2} and \eqref{proposition_IK_7_3} instead of \eqref{proposition_IK_7_1}, one can also obtain for $b=\wtilde e_{\com}$ or $\wtilde e^+$ that
$$
\norm{\Op_h(a^+)^*\varphi(H_h)e^{-i(t-s)H_h/h}\Op_h(b)}_{\infty\to \infty}\lesssim h^M\<t-s\>^{-M},\quad t\le s,
$$
which, together with \eqref{proposition_IK_4_1}, implies the desired bounds for the terms $U_3(t)$ and $U_4(t)$. 
\end{proof}

By Theorem \ref{theorem_IK_1} and the $TT^*$- argument, we have obtained \eqref{IK_3} for $0<\lambda\le1$ which, combined with Proposition \ref{proposition_dispersive_1}, concludes the proof of \eqref{IK_2} for $0<\lambda\le1$. 

\subsection{The high energy case}
\label{subsection_high_energy}

The proof of the high energy estimate \eqref{IK_1} basically follows the same line as that of \eqref{IK_3}. Note however that we do not have to decompose $a^\lambda$ into the compact and non-compact parts in contrast to the low energy case since, for $R\ge1$ large enough, $a^\lambda$ is supported in the region
$$
\{(x,\xi)\in \R^{2n}\ |\ |x|> R,\ |\xi|^2\in I\}
$$
with some interval $I\Subset(0,\infty)$. 

As in the low energy case, decomposing $a^\lambda$ as $a^\lambda=a_++a_-$ with some $a_\pm \in S^{0,-\infty}$ supported in $\Gamma^\pm (R,I,1/2)$, we see that \eqref{IK_1} is deduced from the following dispersive estimate
\begin{align}
\label{high_energy_1}
\norm{\Op(a_\pm)^*\varphi(H^\lambda)e^{-itH^\lambda}\Op(a_\pm)}_{1\to \infty}\lesssim |t|^{-n/2},\quad t\neq0,\ \lambda\ge1. 
\end{align}
In what follows, we summarize several propositions from which \eqref{high_energy_1} follows as in the low energy case. The proofs of these propositions are essentially same as those of the corresponding propositions in the low energy case. Indeed, by using Lemma \ref{lemma_functional_1}, Propositions \ref{proposition_functional_2} and \ref{proposition_smoothing_3}, instead of Propositions \ref{proposition_functional_3} and \ref{proposition_smoothing_2}, all the propositions can be obtained by simply adapting the proof of the low energy case with, firstly, replacing $V_h$ by $V^\lambda$ and, then, taking $h=1$. 

%proposition
\begin{proposition}[IK parametrix]
\label{proposition_high_energy_1}
Let $I_1\Subset I_{2}\Subset  I_3\Subset I_4\Subset (0,\infty)$ be relatively compact open intervals and $-1<\sigma_1<\sigma_{2}<\sigma_3<\sigma_4<1$.  Then there exists a large constant $R_{\IK}\ge1$ such that, for all $R\ge R_{\IK}^4$, the following statements are satisfied:  
\begin{itemize}
\item There exist $\{S_\pm^\lambda\}_{\lambda\ge1}\subset  C^\infty(\R^2n;\R)$ such that 
\begin{align*}
|\partial_x^\alpha\partial_\xi^\beta(S_\pm^\lambda(x,\xi)-x\cdot\xi|&\le C_{\alpha\beta}\<x\>^{1-\mu-|\alpha|},\\
|\partial_x^\alpha\partial_\xi^\beta(S_\pm^\lambda(x,\xi)-x\cdot\xi|&\le C_{\alpha\beta}\min\{\<x\>^{1-\mu-|\alpha|},R^{1-\mu-|\alpha|}\},\quad |\alpha|\ge1,
\end{align*}
uniformly in $\lambda\ge1$, and that
\begin{align*}
|\nabla_x S_\pm^\lambda(x,\xi)|^2+V_h(x)=|\xi|^2,\quad (x,\xi)\in \Gamma^\pm(R^{1/4},I_4,\sigma_4).
\end{align*}
\item For any $a_\pm\in S^{0,-\infty}$ supported in $ \Gamma^\pm(R,I_1,\sigma_1)$ and any $N\ge0$, we have
\begin{align*}
e^{-itH^\lambda}\Op(a_\pm)
&=J_\pm^\lambda(c_\pm)e^{it\Delta}J_\pm^\lambda(d_\pm)^*+\sum_{j=1}^4Q_{j,\pm}(t,\lambda),
\end{align*}
where the remainder terms $Q_{j,\pm}$, $1\le j\le 4$, are given by 
\begin{align*}
Q_{1,\pm}(t,\lambda)&=e^{-itH^\lambda}\Op(r_{1,\pm}),\\
Q_{2,\pm}(t,\lambda)&=-i\int_0^t e^{-i(t-s)H^\lambda}J_\pm^\lambda(r_{2,\pm})e^{is\Delta}J_\pm^\lambda(d_\pm)^*ds,\\
Q_{3,\pm}(t,\lambda)&=- i\int_0^t e^{-i(t-s)H^\lambda}\Op(\wtilde e^{\com}_\pm)J_\pm^\lambda(e^{\com}_\pm)e^{is\Delta}J_\pm^\lambda(d_\pm)^*ds,\\
Q_{4,\pm}(t,\lambda)&=- i\int_0^t e^{-i(t-s)H^\lambda}\Op(\wtilde e_\pm)J_\pm^\lambda(e_\pm)e^{is\Delta}J_\pm^\lambda(d_\pm)^*ds
\end{align*}
and the amplitudes satisfy the following properties:
\begin{itemize}
\item $c_\pm,d_\pm,e_\pm,\wtilde e_\pm\in S^{0,-\infty}$, $r_{1,\pm},r_{2,\pm}\in S^{-N,\infty}$ and $e^{\com}_\pm\in C_0^\infty(\R^{2n})$;
\item $\supp c_\pm\subset \Gamma^\pm(R^{1/3},I_3,\sigma_3)$, $\supp d_\pm \subset \Gamma^\pm(R^{1/2},I_2,\sigma_2)$;
\item $\supp e^{\com}_\pm,\supp \wtilde e^{\com}_\pm\subset\Gamma^\pm(R^{1/4},I_4,\sigma_4)\cap\{|x|\le R^{4/9}\}$;
\item $\supp e_\pm,\supp \wtilde e_\pm\subset \Gamma^\pm(R^{1/4},I_4,\sigma_4)\cap \Gamma^\mp(R^{1/4},I_4,-\wtilde\sigma)$ with some $\wtilde \sigma\in (\sigma_2,\sigma_3)$.
\end{itemize}
\end{itemize}
\end{proposition}

Here $J_\pm^\lambda(a)$ denotes the FIO with the phase function $S_\pm^\lambda$ and the amplitude $a$, namely
$$
J_\pm^\lambda(a)f(x)=\frac{1}{(2\pi)^n}\iint e^{i(S_\pm^\lambda(x,\xi)-y\cdot\xi)}a(x,\xi)f(y)dyd\xi. 
$$

%proposition
\begin{proposition}
\label{proposition_high_energy_2}
Let $I\Subset (0,\infty)$ be a relatively compact interval. For sufficiently large $R\ge R_{\IK}$ and any symbols $a,b\in S^{0,-\infty}$ satisfying $|\xi|^2\in I$ on $\supp a$ and on $\supp b$, we have
$$
\norm{J_\pm^\lambda(a)e^{it\Delta}J_\pm^\lambda(b)^*}_{1\to \infty}\lesssim |t|^{-n/2},\quad t\neq0,\ \lambda\ge1.
$$
\end{proposition}

%proposition
\begin{proposition}
\label{proposition_high_energy_3}
Let $d_\pm$, $r_{2,\pm}$, $e^{\com}_\pm$ and $e_\pm$  be as above. Then for all $s\in \R$ and large $N\ge1$, 
$$
\norm{\<D\>^{s}\<x\>^{N/4}J_\pm^\lambda(r_{2,\pm})e^{it\Delta}J_\pm^\lambda(d_\pm)^*\<x\>^{N/8}\<D\>^{s}}\lesssim \<t\>^{-N/8}
$$
uniformly in $\pm t\ge0$ and $\lambda\ge 1$. Moreover, for any $M\ge0$ and $s\in \R$,
$$
\norm{\<D\>^{s}\<x\>^{M}J_\pm^\lambda(\chi_\pm)e^{it\Delta}J_\pm^\lambda(d_\pm)^*\<x\>^{M}\<D\>^{s}}\lesssim \<t\>^{-M},
$$
uniformly in $\pm t\ge0$ and $\lambda\ge1$, where $\chi_\pm=e^{\com}_\pm$ or $e_\pm$. 
\end{proposition}

%{lemma}
\begin{proposition}
\label{proposition_high_energy_4}
Let $R\ge1$ be large enough and $b_\pm\in S^{0,-\infty}$ supported in $\Gamma^\pm(R,I,\sigma)$. Then the following statements are satisfied with constants independent of $t$ and $\lambda\ge1$. 
\begin{itemize}
\item For sufficiently large integer $N$ and all $s\ge0$, 
$$
\norm{\<x\>^{-N}\varphi(H^\lambda)e^{-itH^\lambda}\Op(b_\pm)\<D\>^s}\lesssim \<t\>^{-N/8},\quad \pm t\ge 0. 
$$
\item For any $\chi\in C_0^\infty (\R^n)$ supported in a unit ball $\{|x|<1\}$ and all $M,s\ge0$, 
$$
\norm{\chi(x/R^{1/3})\varphi(H^\lambda)e^{-itH^\lambda}\Op(b_\pm)\<D\>^s}\lesssim \<t\>^{-M},\quad \pm t\ge 0. 
$$
\item If $\chi_\mp\in S^{0,-\infty}$ is supported in $\Gamma^\mp(R^{1/4},I_4,-\wtilde\sigma)$ with some $\wtilde \sigma>\sigma$, then for all $M,s\ge0$
$$
\norm{\<D\>^s\Op(\chi_\mp)^*\varphi(H^\lambda)e^{-itH^\lambda}\Op(b_\pm)\<x\>^M\<D\>^s}\lesssim \<t\>^{-M},\quad \pm t\ge 0. 
$$
\end{itemize}
\end{proposition}

%%%%%%%%%%%%%%%%%%%%%%%%%
\section{Inhomogeneous estimates}
This section is devoted to the proof of the inhomogeneous estimate \eqref{theorem_LP_2_2}. As in case of homogeneous estimates, the proof for the high energy case is almost identical to that for the low energy case. We thus may assume $\lambda\le1$. Since the non-endpoint inhomogeneous estimates follow from the homogeneous estimates and Christ-Kiselev's lemma \cite{ChKi}, we also may assume $n\ge3$ and $(p,q)=(\tilde p,\tilde q)=(2,2^*)$. By  duality, \eqref{theorem_LP_2_2} is a consequence of the sesquilinear estimate
\begin{align}
\label{inhomogeneous_1}
\left|\int_\R\int_0^t\<\varphi^\lambda(H)e^{-i(t-s)H}F(s),G(t)\>dsdt\right|\lesssim \norm{F}_{L^2_tL^{2_*}_x  }\norm{G}_{L^2_tL^{2_*}_x  }.
\end{align}
On the other hand, we have already proved the homogeneous endpoint estimate \eqref{theorem_LP_2_1}, which implies \eqref{inhomogeneous_1} with the time interval $[0,t]$ replaced by $I=[0,\infty),(-\infty,0]$ or $\R$. Therefore, \eqref{inhomogeneous_1} is deduced from the retarded estimate
\begin{align}
\label{inhomogeneous_2}
\left|\iint_{s<t}\<\varphi^\lambda(H)e^{-i(t-s)H}F(s),G(t)\>dsdt\right|\lesssim \norm{F}_{L^2_tL^{2_*}_x  }\norm{G}_{L^2_tL^{2_*}_x  }.
\end{align}
In what follows the notation in the previous section will be used. We then write $\varphi^\lambda (H)e^{-itH}$ as
\begin{align*}
\varphi^\lambda (H)e^{-itH}
&=\wtilde \varphi^\lambda(H)^2\varphi^\lambda (H)e^{-itH}(\wtilde \varphi^\lambda(H)^*)^2\\
&=\wtilde \varphi^\lambda(H)\Big(\D_\mu\Op_h(a_h)^*\D_\mu^*+R^\lambda\Big)\varphi^\lambda (H)e^{-itH}\Big(\D_\mu\Op_h(a_h)\D_\mu^*+(R^\lambda)^*\Big)\wtilde \varphi^\lambda(H). 
\end{align*}
$a_h$ is further decomposed as $a_h=a^++a^-+a^{\com}$, where $a^\pm,a^{\com}\in S^{-0,-\infty}$, $\supp a^\pm\subset \Gamma^\pm(R,I,1/2)$ and $\supp a^{\com} \subset \{|x|<2R,\ |\xi|\le c_1\}$ with some $R\gg1$, $I\Subset (0,\infty)$ and $c_1>0$ being independent of $h$. Then $\varphi^\lambda (H)e^{-itH}$ is a sum of operators
\begin{align*}
Y_1^\pm (t)&=\wtilde \varphi^\lambda(H)\D_\mu \Op_h(a^\pm)^*\D_\mu^*\varphi^\lambda (H)e^{-itH}\D_\mu \Op_h(a^\pm)\D_\mu^*\wtilde \varphi^\lambda(H),\\
Y_2(t)&=\wtilde \varphi^\lambda(H)\D_\mu \Op_h(a^{\com})^*\D_\mu^*\varphi^\lambda (H)e^{-itH}\D_\mu \Op_h(a^{\com})\D_\mu^*\wtilde \varphi^\lambda(H),\\
Y_3(t)&=\wtilde \varphi^\lambda(H)R^\lambda\varphi^\lambda (H)e^{-itH}(R^\lambda)^*\wtilde \varphi^\lambda(H),\\
Z_1^\pm(t)&=\wtilde \varphi^\lambda(H)\D_\mu \Op_h(a^\mp)^*\D_\mu^*\varphi^\lambda (H)e^{-itH}\D_\mu \Op_h(a^\pm)\D_\mu^*\wtilde \varphi^\lambda(H),\\
Z_2^\pm(t)&=\wtilde \varphi^\lambda(H)\D_\mu \Op_h(a^{\com})^*\D_\mu^*\varphi^\lambda (H)e^{-itH}\D_\mu \Op_h(a^\pm)\D_\mu^*\wtilde \varphi^\lambda(H),\\
Z_3^\pm(t)&=Z_2^\pm(-t)^*=\wtilde \varphi^\lambda(H)\D_\mu \Op_h(a^\pm)^*\D_\mu^*\varphi^\lambda (H)e^{-itH}\D_\mu \Op_h(a^{\com})\D_\mu^*\wtilde \varphi^\lambda(H),\\
Z_4^\pm(t)&=\wtilde \varphi^\lambda(H)R^\lambda\varphi^\lambda (H)e^{-itH}\D_\mu \Op_h(a^\pm)\D_\mu^*\wtilde \varphi^\lambda(H),\\
Z_5^\pm(t)&=Z_4^\pm(-t)^*=\wtilde \varphi^\lambda(H)\D_\mu \Op_h(a^\pm)^*\D_\mu^*\varphi^\lambda (H)e^{-itH}(R^\lambda)^*\wtilde \varphi^\lambda(H),\\
Z_6(t)&=\wtilde \varphi^\lambda(H)\D_\mu \Op_h(a^{\com})^*\D_\mu^*\varphi^\lambda (H)e^{-itH}(R^\lambda)^*\wtilde \varphi^\lambda(H),\\
Z_7(t)&=Z_6(-t)^*=\wtilde \varphi^\lambda(H)R^\lambda\varphi^\lambda (H)e^{-itH}\D_\mu \Op_h(a^{\com})\D_\mu^*\wtilde \varphi^\lambda(H),
\end{align*}
where $Y_j,Y_j^\pm$ (resp. $Z_j,Z_j^\pm$) correspond to the diagonal (resp. off-diagonal) terms. To prove \eqref{inhomogeneous_2}, it is sufficient to show
\begin{align}
\label{inhomogeneous_3}
\left|\iint_{s<t}\<W(t-s)F(s),G(t)\>dsdt\right|&\lesssim \norm{F}_{L^2_tL^{2_*}_x  }\norm{G}_{L^2_tL^{2_*}_x  }
\end{align}
for all $W\in \{Y_1^\pm,Y_2,Y_3,Z_1^\pm,...,Z_5^\pm,Z_6,Z_7\}$. Before starting its proof, we make a small but useful remark.

%remark
\begin{remark}
Since $\D_\mu^*=\D(\lambda^{-\mu/2})$ and $R^\lambda=\varphi^\lambda(H)-\D\Op(a^\lambda)^*\D$, by \eqref{PDO_1} and \eqref{dilation_1} and Lemma \ref{lemma_LP_2}, all of operators $\wtilde \varphi^\lambda(H),\D_\mu \Op_h(a^\pm)^*\D_\mu^*$, $\D_\mu \Op_h(a^{\com})^*\D_\mu^*$, $R^\lambda$ and their adjoints are bounded on $L^{2_*}$ uniformly in $h=\lambda^{2/\mu-1}\in (0,1]$. Hence \eqref{theorem_LP_2_1} implies
$$
\left|\iint\<W(t-s)F(s),G(t)\>dsdt\right|\lesssim \norm{F}_{L^2_tL^{2_*}_x  }\norm{G}_{L^2_tL^{2_*}_x  }
$$
for all $W\in \{Y_1^\pm,Y_2,Y_3,Z_1^\pm,...,Z_5^\pm,Z_6,Z_7\}$. 
Therefore, \eqref{inhomogeneous_3} is equivalent to 
\begin{align}
\label{inhomogeneous_4}
\left|\iint_{s>t}\<W(t-s)F(s),G(t)\>dsdt\right|&\lesssim \norm{F}_{L^2_tL^{2_*}_x  }\norm{G}_{L^2_tL^{2_*}_x  }. 
\end{align}In particular, for each $W\in \{Y_1^\pm,Y_2,Y_3,Z_1^\pm,...,Z_5^\pm,Z_6,Z_7\}$, one can fix the sign of $t-s$ for which $W(t-s)$ behaves better than the case when $t-s$ has the opposite sign. Such an observation was previously pointed out by Hassell-Zhang \cite{HaZh}. 
\end{remark}

With this remark at hand, one sees that \eqref{inhomogeneous_3} follows from the following Lemmas \ref{lemma_inhomogeneous_2}--\ref{lemma_inhomogeneous_7}. By density, we may assume $F,G\in \S(\R\times\R^n)$ in the sequel. 

%{lemma}
\begin{lemma}
\label{lemma_inhomogeneous_2}
Let $W=Y_1^\pm$. Then \eqref{inhomogeneous_3} holds. 
\end{lemma}

%proof
\begin{proof}
Using a similar argument as in the previous section based on \eqref{dilation_1} and the change of variable $t\mapsto \lambda^2h^{-1}t$, \eqref{inhomogeneous_3} with $W=Y_1^\pm$ is equivalent to 
\begin{align*}
&\left|\iint_{s<t}\<\Op_h(a^\pm)^*\varphi(H_h)e^{-i(t-s)H_h/h}\Op_h(a^\pm)\wtilde \varphi(H_h)F(s),\wtilde \varphi(H_h)G(t)\>dsdt\right|\\
&\lesssim h^{-1}\norm{F}_{L^2_tL^{2_*}_x  }\norm{G}_{L^2_tL^{2_*}_x  }
\end{align*}
which, by virtue of the $TT^*$-argument by Keel-Tao \cite{KeTa} and the fact that $\wtilde \varphi(H_h)$ is bounded on $L^{2_{*}}$ uniformly in $h$, follows from \eqref{dispersive}. 
\end{proof}

For $W=Y_2$ and $Y_3$, we consider the original estimate \eqref{inhomogeneous_1}.
%proposition
\begin{lemma}
\label{lemma_inhomogeneous_3}
We have
$$
\bignorm{\int_0^tY_2(t-s)F(s)ds}_{L^2_tL^{2^*}_x  }\lesssim \norm{F}_{L^2_tL^{2_*}_x  }.
$$
\end{lemma}

%proof
\begin{proof}
By the same scaling considerations as in Lemma \ref{lemma_inhomogeneous_2}, it suffices to show
\begin{align}
\label{proof_lemma_inhomogeneous_3_1}
\bignorm{\int_0^t\Op_h(a^{\com})^*\varphi (H_h)e^{-i(t-s)H_h/h}\Op_h(a^{\com})F(s)ds}_{L^2_tL^{2^*}_x  }\lesssim h^{-1}\norm{F}_{L^2_tL^{2_*}_x  }.
\end{align}
To this end, we shall first show the following estimate
\begin{align}
\label{proof_lemma_inhomogeneous_3_2}
\bignorm{\int_0^t\Op_h(a^{\com})^*\varphi (H_h)e^{-i(t-s)H_h/h}F(s)ds}_{L^2_tL^{2^*}_x  }\lesssim h^{-1/2}\norm{\<x\>F}_{L^2_tL^{2}_x}.
\end{align}
Set
$$
u(t)=-i\int_0^t\varphi (H_h)e^{-i(t-s)H_h/h}F(s)ds. 
$$
Choosing $\wtilde a^{\com}\in C_0^\infty(\R^{2n})$ satisfying $\wtilde a^{\com}\equiv1$ on $\supp a^{\com}$, we know by Proposition \ref{proposition_PDO_1} that 
$$\Op_h(a^{\com})^*u=\Op_h(\wtilde a^{\com})\Op_h(a^{\com})^*u+h^N\Op_h(r)u$$
 with some $r\in S^{-N,-\infty}$ and $N\gg1$. For the remainder, \eqref{PDO_1} and \eqref{proposition_smoothing_2_2} imply
\begin{equation}
\begin{aligned}
\label{proof_lemma_inhomogeneous_3_3}
\norm{h^N\Op_h(r)u}_{L^2_tL^{2^*}_x  }
&\lesssim h^{N}\norm{\Op_h(r)\<x\>}_{2\to 2^*}\norm{\<x\>^{-1}u}_{L^2_tL^{2}_x  }\\
&\lesssim \norm{\<x\>F}_{L^2_tL^{2}_x  }
\end{aligned}
\end{equation}
Let $\theta_j,\wtilde \theta_j$ be as that in the proof of Proposition \ref{proposition_dispersive_1}. To deal with the main term, we consider 
$$
v_j(t)=\theta_j(t)\Op_h(a^{\com})^*u(t),
$$
which satisfies the Cauchy problem
$$
i\partial_t v_j(t)=h^{-1}H_hv_j(t)+G_j(t);\quad v_j(0)=0,
$$
and thus Duhamel's formula
$$
v_j(t)=-i\int_0^te^{-i(t-s)H_h/h}G_j(s)ds,
$$
where $G_j=i\theta_j'\Op_h(a^{\com})^*u+h^{-1}\theta_j[\Op_h(a^{\com})^*,H_h]u+i\theta_j\Op_h(a^{\com})^*\varphi (H_h)F$. Therefore, since $\wtilde \theta_j\theta_j=\theta_j$, $\Op_h(\wtilde a^{\com})v_j(t)$ is a linear combination of
\begin{align*}
w_{1,j}(t)&=\wtilde\theta_j(t)\Op_h(\wtilde a^{\com})\int_0^te^{-i(t-s)H_h/h}\theta_j'(s)\Op_h(a^{\com})^*u(s)ds,\\
w_{2,j}(t)&=\wtilde\theta_j(t)\Op_h(\wtilde a^{\com})\int_0^te^{-i(t-s)H_h/h}\theta_j(s)h^{-1}[\Op_h(a^{\com})^*,H_h]u(s)ds,\\
w_{3,j}(t)&=\wtilde\theta_j(t)\Op_h(\wtilde a^{\com})\int_0^te^{-i(t-s)H_h/h}\theta_j(s)\Op_h(a^{\com})^*\varphi (H_h)F(s)ds. 
\end{align*}
Let us set
$$
\wtilde w_{1,j}(t)=\wtilde\theta_j(t)\Op_h(\wtilde a^{\com})\int_0^\infty e^{-i(t-s)H_h/h}\theta_j'(s)\Op_h(a^{\com})^*u(s)ds.
$$
Since $|\supp\theta_j|\ll1$, we apply  \eqref{lemma_dispersive_2_1} to the operator $\wtilde\theta_j(t)\Op_h(\wtilde a^{\com})e^{-itH_h/h}$ yielding
\begin{align}
\label{proof_lemma_inhomogeneous_3_4}
\norm{\wtilde w_{1,j}}_{L^2_tL^{2^*}_x}\lesssim h^{-1/2} \norm{\theta_j'(s)\Op_h(a^{\com})^*u}_{L^1_tL^2_x}
\end{align}
Since the exponents of Lebesgue norms with respect to $t$ of the left and right hand sides in \eqref{proof_lemma_inhomogeneous_3_4} are different from each other, Christ-Kiselev's lemma can be applied to see that $w_{1,j}$ also satisfies the same estimate as \eqref{proof_lemma_inhomogeneous_3_4}. Therefore, we have
\begin{align*}
\norm{w_{1,j}}_{L^2_tL^{2^*}_x  }
&\lesssim h^{-1/2}\norm{\theta_j'\Op_h(a^{\com})^*u}_{L^1_tL^2_x}\\
&\lesssim h^{-1/2}\norm{\theta_j'\Op_h(a^{\com})^*u}_{L^2_tL^{2}_x  }\\
&\lesssim h^{-1/2}\norm{\Op_h(a^{\com})^*\<x\>}\norm{\theta_j'\<x\>^{-1}u}_{L^2_tL^{2}_x  }\\
&\lesssim h^{-1/2}\norm{\theta_j'\<x\>^{-1}u}_{L^2_tL^{2}_x}
\end{align*}
uniformly in $j\in \Z$ and $h\in (0,1]$, where we have used H\"older's inequality and the uniform bound $|\supp\theta_j'|\lesssim1$ with respect to $j$ in the second line. 

Similarly, we have
$$
\norm{w_{3,j}}_{L^2_tL^{2^*}_x  }\lesssim h^{-1/2}\norm{\theta_jF}_{L^2_tL^{2}_x  }.
$$

For $w_{2,j}$, we find by the same argument that
$$
\norm{w_{2,j}}_{L^2_tL^{2^*}_x  }\lesssim h^{-3/2}\norm{[\Op_h(a^{\com})^*,H_h]\theta_ju}_{L^2_tL^{2}_x  }\lesssim h^{-1/2}\norm{\theta_j\<x\>^{-1}u}_{L^2_tL^{2}_x  },
$$
where we have used  \eqref{lemma_dispersive_3_2} with sufficiently large $N=N(\mu)\in \N$ in the last line. These bounds for $w_{k,j}$,  the almost orthogonality of $\{\theta_j\}_{j\in \Z}$ and \eqref{proposition_smoothing_2_2} show that
\begin{align*}
&\norm{\Op_h(\wtilde a^{\com})^*\Op_h(a^{\com})^*u}_{L^2_tL^{2^*}_x  }^2\\
&\lesssim h^{-1}\sum_{j\in \Z}\Big(\norm{\theta_j\<x\>^{-1}u}_{L^2_tL^{2}_x  }^2+\norm{\theta_j'\<x\>^{-1}u}_{L^2_tL^{2}_x  }^2+\norm{\theta_jF}_{L^2_tL^{2}_x  }^2\Big)\\
&\lesssim h^{-1}\Big(\norm{\<x\>^{-1}u}_{L^2_tL^{2}_x  }^2+\norm{F}_{L^2_tL^{2}_x  }^2\Big)\\
&\lesssim h^{-1}\norm{\<x\>F}_{L^2_tL^{2}_x  }^2
\end{align*}
which, together with \eqref{proof_lemma_inhomogeneous_3_3}, implies \eqref{proof_lemma_inhomogeneous_3_2}. 

Next, we obtain by \eqref{proof_lemma_inhomogeneous_3_2}, duality and the change of variables $t\mapsto -t$,$s\mapsto -s$ that
\begin{align}
\label{proof_lemma_inhomogeneous_3_5}
\bignorm{\int_0^t\<x\>^{-1}\varphi (H_h)e^{-i(t-s)H_h/h}\Op_h(a^{\com})F(s)ds}_{L^2_tL^{2}_x}\lesssim h^{-1/2}\norm{F}_{L^2_tL^{2_*}_x  }.
\end{align}
Now we set 
$$
\wtilde u(t)=\int_0^t\varphi (H_h)e^{-i(t-s)H_h/h}\Op_h(a^{\com})F(s)ds
$$ 
Then, by repeating the same argument as above with \eqref{proof_lemma_inhomogeneous_3_5} instead of \eqref{proposition_smoothing_2_2}, we have
\begin{align*}
&\norm{\Op_h(a^{\com})^*\wtilde u}_{L^2_tL^{2^*}_x  }^2\\
&\lesssim h^{-1}\sum_{j\in \Z}\Big(\norm{\theta_j\<x\>^{-1}\wtilde u}_{L^2_tL^{2}_x  }^2+\norm{\theta_j'\<x\>^{-1}\wtilde u}_{L^2_tL^{2}_x  }^2\Big)+h^{-2}\sum_{j\in \Z}\norm{\theta_jF}_{L^2_tL^{2_*}_x  }^2\\
&\lesssim h^{-1}\norm{\<x\>^{-1}\wtilde u}_{L^2_tL^{2}_x  }^2+h^{-2}\norm{F}_{L^2_tL^{2_*}_x  }\\
&\lesssim h^{-2}\norm{F}_{L^2_tL^{2_*}_x  }^2,
\end{align*}
which completes the proof of \eqref{proof_lemma_inhomogeneous_3_1}. 
\end{proof}

\begin{lemma}
\label{lemma_inhomogeneous_4}
We have
$$
\bignorm{\int_0^tY_3(t-s)F(s)ds}_{L^2_tL^{2^*}_x  }\lesssim \norm{F}_{L^2_tL^{2_*}_x  }.
$$
\end{lemma}

%proof
\begin{proof}
It follows from the proof of Proposition \ref{proposition_reduction_1} that 
$$\norm{\wtilde \varphi^\lambda(H)R^\lambda\<x\>}_{2\to 2^*}+\norm{\<x\>(R^\lambda)^*\wtilde \varphi^\lambda(H)}_{2_*\to 2}\lesssim1$$ uniformly in $0<\lambda\le1$. This bound allows us to deduce the desired estimate from \eqref{proposition_smoothing_1_2}. 
\end{proof}

%proposition
\begin{lemma}
\label{lemma_inhomogeneous_5}
For $W=Z_1^+$ (resp. $W=Z^-_1$), \eqref{inhomogeneous_3} (resp. \eqref{inhomogeneous_4}) holds. 
\end{lemma}

%proof
\begin{proof}
We shall consider the case $W=Z_1^+$ only, the proof for the other case being analogous. As in case of $Y_1^+$, it suffices to show the following dispersive estimate
\begin{align}
\label{proof_lemma_inhomogeneous_5_1}
\norm{\Op_h(a^-)^*\varphi (H_h)e^{-itH_h/h}\Op_h(a^+)}_{1\to \infty}\lesssim |th|^{-n/2},\quad t>0. 
\end{align}
By Proposition \ref{proposition_IK_3}, we decompose the operator in \eqref{proof_lemma_inhomogeneous_5_1} into five parts
$$
\Op_h(a^-)^*\varphi (H_h)J^+_h(c^+)e^{ith\Delta}J^+_h(d^+)^*,\quad \Op_h(a^-)^*\varphi (H_h)Q_j^+(t,h),\quad 1\le j\le4. 
$$
The estimate for the main term follows from \eqref{proposition_IK_4_1} as in the proof of Theorem \ref{theorem_IK_1}. The desired estimate for the term 
$$
\Op_h(a^-)^*\varphi (H_h)Q_1^+(t,h)=h^N\Op_h(a^-)^*\varphi (H_h)e^{-itH_h/h}\Op_h(r^+_1)
$$ follows from the dual estimate of the incoming case of \eqref{proposition_IK_7_1} and \eqref{PDO_1}. To deal with the term
$$
\Op_h(a^-)^*\varphi (H_h)Q_2^+(t,h)=-ih^{N-1}\int_0^t \Op_h(a^-)^*\varphi (H_h)e^{-i(t-\tau)H_h/h}J_h^+(r^+_2)e^{i\tau h\Delta}J^+_h(d^+)^*d\tau,
$$
we again use the dual estimate of the incoming case of \eqref{proposition_IK_7_1} and also \eqref{proposition_IK_5_1} to see that
\begin{align*}
\norm{\<D\>^s\Op_h(a^-)^*\varphi (H_h)e^{-i(t-\tau)H_h/h}\<x\>^{-N/4}}&\lesssim h^{-n-s-1}\<t-\tau\>^{-N/8}\\
\norm{\<x\>^{N/4}J_h^+(r^+_2)e^{i\tau h\Delta}J^+_h(d^+)^*\<D\>^s}&\lesssim h^{-n-s}\<\tau\>^{-N/8}
\end{align*}
for $0\le \tau\le t$ and all $s\in \R$. Then, by using Sobolev's embedding and integrating over $[0,t]$, we obtain the desired estimate provided that $N$  is large enough. The term $\Op_h(a^-)^*\varphi (H_h)Q_3^+(t,h)$ can be dealt with similarly to the second term by using \eqref{proposition_IK_7_3} instead of \eqref{proposition_IK_7_1}. Finally, to deal with the last term
\begin{align*}
&\Op_h(a^-)^*\varphi (H_h)Q_4^+(t,h)\\
&=-\frac ih\int_0^t \Op_h(a^-)^*\varphi (H_h)e^{-i(t-\tau)H_h/h}\Op_h(\wtilde e^+)J^+_h(e^+)e^{i\tau h\Delta}J^+_h(d^+)^*d\tau ,
\end{align*}
taking into account the support property
\begin{align*}
\supp a^-&\subset \Gamma^-(R,I,1/2)\subset \Gamma^-(R^{1/4},I_4,\sigma_4),\\
 \supp e^+&\subset \Gamma^-(R^{1/4},I_4,-\wtilde \sigma)\subset  \Gamma^-(R^{1/4},I_4,\sigma_4)
\end{align*}
since $1/2<\wtilde \sigma<\sigma_4$, we use  the incoming case of Theorem \ref{theorem_IK_1} and \eqref{proposition_IK_7_3} to obtain
\begin{align*}
\norm{\Op_h(a^-)^*\varphi (H_h)e^{-i(t-\tau)H_h/h}\Op_h(\wtilde e^+)}_{1\to \infty}
&\lesssim \min\{h^{-n},|(t-\tau)h|^{-n/2}\},\\
\norm{J^+_h(e^+)e^{i\tau h\Delta}J^+_h(d^+)^*}_{1\to 1}
&\lesssim h^M\<\tau\>^{-M}
\end{align*}
for $0\le \tau\le t$ and all $M>0$. These two bounds imply the desired dispersive estimate for the last term and we complete the proof of \eqref{proof_lemma_inhomogeneous_5_1}. 
\end{proof}

%proposition
\begin{lemma}
\label{lemma_inhomogeneous_6}
For $W=Z_6,Z_7$ we have
$$
\bignorm{\int_0^tW(t-s)F(s)ds}_{L^2_tL^{2^*}_x  }\lesssim \norm{F}_{L^2_tL^{2_*}_x  }
$$
\end{lemma}

%proof
\begin{proof}
We may prove the lemma for $Z_7$ only. Let 
$$
\wtilde u(t)=\int_0^t\varphi (H_h)e^{-i(t-s)H_h/h}\Op_h(a^{\com})F(s)ds.
$$
As before, by a scaling consideration, it is enough to show
\begin{align}
\label{proof_lemma_inhomogeneous_6_0}
\norm{\wtilde \varphi(H_h)R_h\wtilde u}_{L^2_tL^{2^*}_x  }\lesssim h^{-1}\norm{F}_{L^2_tL^{2_*}_x  }
\end{align}
uniformly in $h=\lambda^{2/\mu-1}\in (0,1]$, where $R_h=\D_\mu^*R^\lambda \D_\mu$. The proof is similar to that of \eqref{proof_lemma_inhomogeneous_3_1}. At first note that $\norm{\<x\>R_h^*}\lesssim1$ uniformly in $h\in (0,1]$. Indeed, by \eqref{dilation_2}, \eqref{proposition_functional_2_1} and unitarity of $\D$ and $\D_\mu$, we have stronger bounds
\begin{align}
\label{proof_lemma_inhomogeneous_6_1_0}
\norm{\<x\>^NR_h^*}\le \norm{\<h^{-1} x\>^N R_h^*}=\norm{\<x\>^N(Q^\lambda+\chi(x)\varphi(H^\lambda))}\le C_N<\infty,\quad h\in (0,1],
\end{align}
for all $N\ge0$. 
This bound and \eqref{proof_lemma_inhomogeneous_3_2} imply
\begin{align}
\label{proof_lemma_inhomogeneous_6_1}
\bignorm{\int_0^t\Op_h(a^{\com})^*\varphi (H_h)e^{-i(t-s)H_h/h}R_h^*\wtilde \varphi(H_h)F(s)ds}_{L^2_tL^{2^*}_x  }\lesssim h^{-1/2}\norm{F}_{L^2_tL^{2}_x}.
\end{align}
Then \eqref{proof_lemma_inhomogeneous_3_5} and the dual estimate of \eqref{proof_lemma_inhomogeneous_6_1} imply
\begin{align}
\label{proof_lemma_inhomogeneous_6_2}
\norm{\<x\>^{-1}\wtilde u}_{L^2_tL^{2}_x  }+\norm{R_h\wtilde u}_{L^2_tL^{2}_x  }\lesssim h^{-1/2}\norm{F}_{L^2_tL^{2_*}_x  }.
\end{align}
Using $\theta_j,\wtilde\theta_j$ given in the proof of Lemma \ref{lemma_inhomogeneous_3}, we set $\wtilde v_j=\theta_j R_h \wtilde u$ which satisfies
$$
\wtilde v_j=-i\int_0^t e^{-i(t-s)H_h/h}\wtilde G_j(s)ds
$$
where $\wtilde G_j=i\theta_j'R_h\wtilde u+\theta_jh^{-1}[R_h,H_h]\wtilde u+i\theta_jR_h\varphi (H_h)\Op_h(a^{\com})F(s)$. For the last term of $\wtilde G_j$, we know by Lemma \ref{lemma_inhomogeneous_4} and the scaling argument as above that
\begin{align}
\label{proof_lemma_inhomogeneous_6_3}
\bignorm{\int_0^t \wtilde \theta_j \wtilde \varphi(H_h)R_he^{-i(t-s)H_h/h}\theta_jR_h\varphi (H_h)\Op_h(a^{\com})F(s)ds}_{L^2_tL^{2^*}_x  }\lesssim h^{-1}\norm{\theta_jF}_{L^2_tL^{2_*}_x  }.
\end{align}
To deal with the second term of $G_j$, we note that by definition of $R_h$,  
$$
h^{-1}[R_h,H_h]=h^{-1}[\varphi(H_h)-\Op_h(a_h)^*,H_h]=-h^{-1}[\Op_h(a_h)^*,H_h].
$$
Since $|x|\gtrsim1$ on $\supp a_h$ (see \eqref{a_h}), it follows from the same proof as that for \eqref{lemma_dispersive_3_2} that $h^{-1}[\Op_h(a_h)^*,H_h]=\Op(c_h)$ with some $c_h\in S^{-1,-\infty}$. Then, by the homogeneous Strichartz estimate \eqref{theorem_LP_2_1}, \eqref{proof_lemma_inhomogeneous_6_3} and  similar computations as in the proof of Lemma \ref{lemma_inhomogeneous_3}, we have
\begin{align*}
\norm{\wtilde \theta_j \wtilde \varphi(H_h)R_h\wtilde u}_{L^2_tL^{2^*}_x  }^2\lesssim h^{-1} \norm{\theta_j'R_h\wtilde u}_{L^2_tL^{2}_x  }+h^{-1}\norm{\theta_j\<x\>^{-1}\wtilde u}_{L^2_tL^{2}_x  }+h^{-2}\norm{\theta_jF}_{L^2_tL^{2_*}_x  }
\end{align*}
which, together with \eqref{proof_lemma_inhomogeneous_6_2} and the almost orthogonality of $\theta_j$ and $\wtilde \theta_j$, implies the desired bound \eqref{proof_lemma_inhomogeneous_6_0}. This completes the proof. 
\end{proof}

%proposition
\begin{lemma}
\label{lemma_inhomogeneous_7}
For $W=Z^+_2,Z^-_3,Z_4^+,Z_5^-$ (resp. $W=Z^-_2,Z^+_3,Z_4^-,Z_5^+$), \eqref{inhomogeneous_3} (resp. \eqref{inhomogeneous_4}) holds. 
\end{lemma}

%proof
\begin{proof}
By duality, it suffices to show the lemma for $Z^+_2$ and $Z^+_4$. We first consider the case with $Z^+_2$. As before, we may show
\begin{equation}
\begin{aligned}
\label{proof_lemma_inhomogeneous_7_1}
&\iint_{s<t}\<\wtilde \varphi(H_h)\Op_h(a^{\com})^*\varphi (H_h)e^{-i(t-s)H_h/h}\Op_h(a^+)\wtilde \varphi(H_h)F(s),G(t)\>dsdt\\
&\lesssim h^{-1}\norm{F}_{L^2_tL^{2_*}_x  }\norm{G}_{L^2_tL^{2_*}_x  }.
\end{aligned}
\end{equation}
Decompose $a^{+}=\wtilde a^{+}+b^{\com}$ with $a^{+}\in S^{0,-\infty},b^{\com}\in S^{-\infty,-\infty}$ satisfying $\supp a^+\subset \Gamma^+(R^2,I,1/2)$ and $\supp b^{\com}\subset \{c_0<|x|<2R^2,\ |\xi|^2\in I\}$. The part associated with $b^{\com}$ can be dealt with the same argument as that for $Y_2$. Moreover, since $|x|\le 2R$ on $\supp a^{\com}$, the same argument as in the proof of \eqref{proposition_IK_7_3} yields that
\begin{align*}
&\norm{\wtilde \varphi(H_h)\Op_h(a^{\com})^*\varphi (H_h)e^{-i(t-s)H_h/h}\Op_h(\wtilde a^+)\wtilde \varphi(H_h)}_{1\to \infty}\\
&\lesssim h^{-n}\norm{\Op_h(a^{\com})^*\varphi (H_h)e^{-i(t-s)H_h/h}\Op_h(\wtilde a^+)}\\
&\lesssim h^{M-n}\<t-s\>^{-M},\quad s\le t,
\end{align*}
for all $M\ge0$ which, together with the $TT^*$-argument, implies \eqref{proof_lemma_inhomogeneous_7_1} with $a^+$ replaced by $\wtilde a^+$. This completes the proof of \eqref{proof_lemma_inhomogeneous_7_1}. 

In case of $W=Z_4^+$, taking the bound \eqref{proof_lemma_inhomogeneous_6_1_0} into account, we decompose $R_h$ as
$$
R_h=B_1\chi(x)+h^NB_2\<x\>^{-N}
$$
where $\norm{B_1}+\norm{B_2}\lesssim1$ uniformly in $h\in (0,1]$,  $\chi\in C_0^\infty(\{|x|<1\})$ and $N\gg n$. Then, using \eqref{proposition_IK_7_1} and \eqref{proposition_IK_7_2}, we similarly obtain
\begin{align*}
&\norm{\wtilde \varphi(H_h)R_h\varphi (H_h)e^{-i(t-s)H_h/h}\Op_h(\wtilde a^+)\wtilde \varphi(H_h)}_{1\to \infty}\\
&\lesssim h^{-n}\norm{\chi(x)\varphi (H_h)e^{-i(t-s)H_h/h}\Op_h(\wtilde a^+)}+h^{N-n}\norm{\<x\>^{-N}\varphi (H_h)e^{-i(t-s)H_h/h}\Op_h(\wtilde a^+)}\\
&\lesssim h^{N-2n-1}\<t-s\>^{-N/8}
\end{align*}
for $s\le t$. Choosing $N\gg n$ large enough and applying the $TT^*$-argument, we obtain \eqref{inhomogeneous_3}. 
\end{proof}

%%%%%%%%%%%%%%%%%%%%%%%%%%%%%%%%%%%%%%%%%%
\section{Proof of Theorem \ref{theorem_2}}
\label{section_theorem_2}

In this section we prove Theorem \ref{theorem_2}. For a given self-adjoint operator $A$ on $L^2$, $U_A$ and $\Gamma_A$ denote the homogeneous and inhomogeneous propagators
$$
U_A=e^{-itA}:L^2_x\to L^\infty_t L^2_x,\quad \Gamma_A[F]=\int_0^te^{-i(t-s)A}F(s)ds:L^1_tL^2_x\to L^\infty_t L^2_x. 
$$
The following space-time weighted $L^2$-estimates play a crucial role in the proof of Theorem \ref{theorem_2}.

%theorem
\begin{proposition}
\label{proposition_singular_1}
Let $\mu\in (0,2)$ and $W\in C^2(\R^n\setminus\{0\})$ satisfy
\begin{align}
\label{proposition_singular_1_1}
W(x)\gtrsim |x|^{-\mu},\quad -x\cdot \nabla W(x)\gtrsim |x|^{-\mu},\ |(x\cdot \nabla)^2W(x)|\lesssim |x|^{-\mu}.
\end{align}
Let $\chi\in C_0^\infty(\R^n)$ and $\rho(x)=\chi(x)|x|^{-\mu/2}$. Then, for $H_1=-\Delta+W(x)$,
$$
\norm{\rho e^{-itH_1}u_0}_{L^2_tL^2_x }\lesssim \norm{u_0}_{L^2},\ \norm{\rho \Gamma_{H_1}[\rho F]}_{L^2_tL^2_x }\lesssim \norm{F}_{L^2_tL^2_x }. 
$$
\end{proposition}

%proof
\begin{proof}
Under the above conditions, $W$ fulfills the conditions in \cite[Example 2.18]{BoMi2}. Then the result is a consequence of \cite[Corollary 2.21]{BoMi2}. 
\end{proof}

%proof
\begin{proof}[Proof of Theorem \ref{theorem_2}]
Let $M_\ell =\norm{|x|^\mu (x\cdot \nabla)^\ell V_S}_\infty$ for $\ell=0,1$ and assume $M_\ell>0$ without loss of generality. It is easy to see that if we choose
$
\ep_*=\min(Z/M_0,\ \mu Z/M_1),
$
then $W(x):=Z|x|^{-\mu}+\ep V_S(x)$ fulfills \eqref{proposition_singular_1_1} provided $\ep\in [0,\ep_*)$.   Choose $\chi\in C_0^\infty(\R^n)$ such that $0\le \chi\le1$ and $\chi\equiv1$ on a unit ball, and decompose $Z|x|^{-\mu}+\ep V_S(x)=V_1(x)+V_2(x)$ where 
$$
V_1(x)=\chi(x)+(1-\chi(x))(Z|x|^{-\mu}+\ep V_S(x)),\quad V_2(x)=\chi(x)(Z|x|^{-\mu}+\ep V_S(x)-1).
$$
It is easy to see that $V_1$ satisfies (H1)--(H3) and $|V_2(x)|\lesssim |\chi(x)||x|^{-\mu}$. 
Let $H_1=H+Z|x|^{-\mu}+\ep V_S(x)$ and $H=-\Delta+V_1(x)$. Decompose $V_2=v_1v_2$ with $v_1=|V_2|^{1/2}$ and $v_2=v_1\sgn V_2$. Note that $|v_1|,|v_2|\lesssim |\chi(x)||x|^{-\mu/2}$ and thus $v_1,v_2\in L^n$ since $\mu<2$. 

Now we use a perturbation method, originated from Rodnianski-Schlag \cite[Section 4]{RoSc} and extended by \cite{BPST2,BoMi2}, which is based on Duhamel's formulas (See \cite[Proposition 4.4]{BoMi2})
\begin{align}
\label{proof_theorem_2_1}
U_{H_1}u_0&=U_Hu_0-i\Gamma_H[V_2U_{H_1}u_0]\\
\label{proof_theorem_2_2}
\Gamma_{H_1}[F]&=\Gamma_H[F]-i\Gamma_H[V_2\Gamma_{H_1}[F]],\\
\label{proof_theorem_2_3}
\Gamma_{H_1}[F]&=\Gamma_H[F]+i\Gamma_{H_1}[V_2\Gamma_H[F]].
\end{align}
We may prove the theorem in the endpoint case only: the other cases follow from complex interpolation. By \eqref{proof_theorem_2_1}, \eqref{proof_theorem_2_2} and Theorem \ref{theorem_1}, it suffices to deal with the inhomogeneous terms $\Gamma_HV_2U_{H_1}$ and $\Gamma_HV_2\Gamma_{H_1}$. By Theorem \ref{theorem_1}, H\"older's inequality and Proposition \ref{proposition_singular_1}, 
\begin{align*}
\norm{\Gamma_H[V_2U_{H_1}u_0]}_{L^2_tL^{2^*}_x  }\lesssim\norm{\Gamma_H}_{L^2_tL^{2_*}_x  \to L^2_tL^{2^*}_x  }\norm{v_1}_{L^n}\norm{v_2U_{H_1}u_0}_{L^2_tL^2_x }\lesssim \norm{u_0}_{L^2}
\end{align*}
which gives us the desired estimate for $\Gamma_HV_2U_{H_1}$. We similarly have
$$
\norm{\Gamma_H[V_2\Gamma_{H_1}[F]]}_{L^2_tL^{2^*}_x  }\lesssim \norm{v_2\Gamma_{H_1}[F]}_{L^2_tL^2_x }.
$$
Then, using \eqref{proof_theorem_2_3}, Theorem \ref{theorem_1}, H\"older's inequality and Proposition \ref{proposition_singular_1}, we see that
\begin{align*}
\norm{v_2\Gamma_{H_1}[F]}_{L^2_tL^2_x }
&\le \norm{v_2\Gamma_H[F]}_{L^2_tL^2_x }+\norm{v_2\Gamma_{H_1}[V_2\Gamma_H[F]]}_{L^2_tL^2_x }\\
&\lesssim (\norm{v_2}_{L^n}+\norm{v_2\Gamma_{H_1}v_1}_{L^2_tL^2_x \to L^2_tL^2_x }\norm{v_2}_{L^n})\norm{\Gamma_H[F]}_{L^2_tL^{2^*}_x  }\\
&\lesssim \norm{F}_{L^2_tL^{2_*}_x  }
\end{align*}
which implies the double endpoint estimate for $\Gamma_HV_2\Gamma_{H_1}$. This completes the proof. 
\end{proof}

%%%%%%%%%%%%%%%%%%%%%%%
\appendix

\section{Proof of the resolvent bound \eqref{proof_proposition_smoothing_2_1}}
\label{appendix_A}
Here we give the proof of \eqref{proof_proposition_smoothing_2_1}. Let $I\Subset (0,\infty)$ be a compact interval such that $\supp \varphi\subset I$ and set $I_\pm=\{z\ |\ \Re z\in I, 0\le \pm \Im z\le 1\}$. When $z\in (I_+\cup I_-)^c$, the spectral theorem implies
\begin{align*}
&\sup_{h\in (0,1]}\sup_{z\in (I_+\cup I_-)^c}\norm{\<x\>^{-\gamma}\varphi(H_h  )(H_h  -z)^{-N}\varphi(H_h  )\<x\>^{-\gamma}}\\
&\lesssim \sup_{(z,\rho)\in (I_+\cup I_-)^c\times \supp\varphi}|\rho-z|^{-N}<\infty.
\end{align*}
To deal with the case when $z\in I_\pm$, we first show that
\begin{align}
\label{proof_proposition_smoothing_2_2}
\sup_{z\in I_\pm}\norm{\<A_h\>^{-\gamma}(H_h  -z)^{-N}\<A_h\>^{-\gamma}}\le C h^{-N},\quad h\in (0,1],
\end{align}
with some $C>0$ independent of $h$, where $A_h=hA$ and $A=-\frac{i}{2}(x\cdot\nabla+\nabla\cdot x)$ is the generator of the dilation $\D$. The proof of \eqref{proof_proposition_smoothing_2_2} is based on a semiclassical version \cite[Theorem 1]{Nak1} of the multiple commutator method by \cite{JMP}. At first, by decomposing $I$ into finitely many intervals, we may assume that $|I|$ is sufficiently small. We also note that, for each fixed $h_0>0$, there exists $C=C(h_0)$ such that \eqref{proof_proposition_smoothing_2_2} with $h\in [h_0,1]$ holds true. Indeed, we obtain by \eqref{dilation_3} that
$$
\<A_h\>^{-\gamma}(H_h  -z)^{-N}\<A_h\>^{-\gamma}=\D_\mu^*\<A_h\>^{-\gamma}(\lambda^{-2}H-z)^{-N}\<A_h\>^{-\gamma}\D_\mu
$$
which, combined with the fact $\<A_h\>^{-\gamma}\lesssim h^{-\gamma}\<A\>^{-\gamma}$ as a bounded operator, implies
$$
\norm{\<A_h\>^{-\gamma}(H_h  -z)^{-N}\<A_h\>^{-\gamma}}\lesssim h^{-2\gamma}\lambda^{2N}\norm{\<A\>^{-\gamma}(H-\lambda^2z)^{-N}\<A\>^{-\gamma}}.
$$
Let $I=[a,b]\Subset (0,\infty)$ and $\lambda_0\in (0,1]$ be such that $h_0=\lambda_0^{2/\mu-1}$. Since 
$$
\lambda^2I_\pm\subset \{Z\ |\ \Re Z\in [\lambda_0^2a,b],\ 0\le \pm \Im Z\le 1\}=:\tilde I^\pm,\quad \lambda \in [\lambda_0,1],
$$
the proof for $h\in [h_0,1]$ is reduced to that of the uniform bound for $\norm{\<A\>^{-\gamma}(H-Z)^{-N}\<A\>^{-\gamma}}$ with respect to $Z\in \tilde I^\pm$. It is easy to see that $H$ satisfies the hypothesis {\it a})--{\it e}) in \cite[Definition 2.1]{JMP} (see \cite[Section 5]{JMP}). Taking the fact that $H$ has no eigenvalue into account, we thus can apply \cite[Theorem 2.2]{JMP} to obtain
$$
\sup_{Z\in \tilde I_\pm}\norm{\<A\>^{-\gamma}(H-Z)^{-N}\<A\>^{-\gamma}}\le C(\lambda_0)
$$
with some $C(\lambda_0)>0$. This implies \eqref{proof_proposition_smoothing_2_2} with $C=h_0^{N-2\gamma}\lambda_0^{2N}C(\lambda_0)$ for $h\in [h_0,1]$. 

Next, we shall check that  all of hypothesis (H1)--(H3) and (H5) in \cite[Section 2]{Nak1} are satisfied for the pair $(H_h,A_h)$ and $h\in (0,h_0]$ with sufficiently small $h_0$. Note that the hypothesis (H4) in \cite[Section 2]{Nak1} is not necessary in the present case (see \cite[Remark after Theorem 1]{Nak1}). At first, $D(A_h)\cap D(H_h)$ is dense in $D(H_h)$ since $C_0^\infty(\R^n)\subset D(A_h)$ and $D(H_h)=\H^2$. By Lemma \ref{lemma_dilation_1}, 
\begin{align}
\label{proof_proposition_smoothing_2_2_0}
|(x\cdot \nabla)^jV_h(x)|\le C_j(\lambda^{2/\mu}+|x|)^{-\mu}\le C_jV_h(x),\quad x\in \R^n,\quad j=0,1,...
\end{align}
with some $C_j>0$. 
Therefore, if we define $B_j$ inductively by $B_0=H_h$ and
\begin{align*}
B_j=[B_{j-1},iA_h]=h^j(-2^jh^2\Delta-(x\cdot\nabla)^j V_h),\quad j\in \N,
\end{align*}
then $B_j$ are well-defined as forms on $D(A_h)\cap D(H_h)$ and extended to bounded operators from $D(H_h)$ to $L^2$. Thus $H_h$ is smooth with respect to $A_h$ in the sense of \cite{JMP}. Since $B_j\lesssim h^jH_h$ by the above computations, it is also easy to see that $\norm{B_j(H_h+i)^{-1}}\le C_jh^j$. This proves (H1)--(H3) in \cite[Section 2]{Nak1}. Next, choosing $\chi\in C_0^\infty(\R^n)$ satisfying $0\le \chi\le1$ and $\chi\equiv1$ on $\{|x|<R_0\}$, we know by \eqref{lemma_dilation_1_4}--\eqref{lemma_dilation_1_6} and \eqref{proof_proposition_smoothing_2_2_0} that  there exists constants $a_0,a_1>0$ such that
\begin{align*}
[H_h,iA_h]&=h(-2h^2\Delta -x\cdot\nabla V_h)\\
&\ge h a_0H_h-ha_1\chi(\lambda^{-2/\mu}x)V_h\\
&\ge h(a_0-a_1\chi(\lambda^{-2/\mu}x) )H_h.
\end{align*}
Here we claim that there exists $h_0>0$  such that
\begin{align}
\label{proof_proposition_smoothing_2_2_1}
\norm{E_{H_h}(I)\chi(\lambda^{-2/\mu} x)H_hE_{H_h}(I)}\le \frac{a_0}{2a_1}
\end{align}
for all $h\in (0,h_0]$, where $E_{H_h}$ is the spectral measure for $H_h$.  We postpone the proof of this claim to Lemma \ref{lemma_A_1} below. It follows from this claim that semiclassical strict Mourre's inequality
\begin{align}
E_{H_h}(I)[H_h  ,iA_h]E_{H_h}(I)\gtrsim hE_{H_h}(I)
\end{align}
holds uniformly in $h\in (0,h_0]$. Hence (H1)--(H3) and (H5)  in \cite[Section 2]{Nak1} are satisfied for $(H_h,A_h)$ and we can apply \cite[Theorem 3]{Nak1} obtaining \eqref{proof_proposition_smoothing_2_2} for $h\in (0,h_0]$. 

To pass \eqref{proof_proposition_smoothing_2_1} from \eqref{proof_proposition_smoothing_2_2}, it remains to show the following uniform bound:  
\begin{align}
\label{proof_proposition_smoothing_2_3}
\norm{\<x\>^{-\gamma}\varphi(H_h)\<A_h\>^\gamma}\lesssim1,\quad h=\lambda^{2/\mu-1}\in (0,1].
\end{align}
We first prove as an intermediate step that
\begin{align}
\label{proof_proposition_smoothing_2_4}
\norm{\<x\>^{-\gamma}\<A_h\>^{2\gamma}\varphi(H_h)\<x\>^{-\gamma}}\lesssim h^{-4\gamma}.
\end{align}
To this end, using \eqref{dilation_2} and the fact that $A_h$ commutes with $\D_\mu$, we see that
\begin{align}
\label{proof_proposition_smoothing_2_5}
\norm{\<x\>^{-\gamma}\<A_h\>^{2\gamma}\varphi(H_h)\<x\>^{-\gamma}}=\norm{\<\lambda^{2/\mu}x\>^{-\gamma}\<A_h\>^{2\gamma}\varphi(\lambda^{-2}H)\<\lambda^{2/\mu}x\>^{-\gamma}}.
\end{align}
Since $\<\lambda^{2/\mu}x\>^{-1}\le \lambda^{-2/\mu}\<x\>^{-1}$ for $\lambda\in (0,1]$ and $\norm{\<A_h\>^\gamma\<A\>^{-\gamma}}\le 1$ for all $h\in (0,1]$ by the spectral theorem, 
the right hand side of \eqref{proof_proposition_smoothing_2_5} is dominated by 
$\lambda^{-4\gamma/\mu}\norm{\<x\>^{-\gamma}\<A\>^{2\gamma}\varphi(\lambda^{-2}H)\<x\>^{-\gamma}}$. 
To deal with this term, we take $\wtilde \varphi\in C_0^\infty(\R)$ so that $\wtilde \varphi\equiv1$ on $(-1,\sup(|\supp\varphi|)+1)$ and write
$$
\<x\>^{-\gamma}\<A\>^{2\gamma}\varphi(\lambda^{-2}H)\<x\>^{-\gamma}=\<x\>^{-\gamma}\<A\>^{2\gamma}\wtilde \varphi(H)\<x\>^{-\gamma}\<x\>^{\gamma}\varphi(\lambda^{-2}H)\<x\>^{-\gamma}
$$
where we have used the fact $\wtilde \varphi(x)\varphi(\lambda^{-2}x)=\varphi(\lambda^{-2}x)$ for $x\in \R$, $0<\lambda\le1$. 
By the same proof as that of Lemma \ref{lemma_functional_1}, we have $\norm{\<x\>^{\gamma}\varphi(\lambda^{-2}H)\<x\>^{-\gamma}}\lesssim \lambda^{-2}$. Moreover, by Remark \ref{remark_functional_1}, $\wtilde \varphi(H)\<x\>^{-\gamma}=\Op(a_\gamma)+Q_\gamma$ with some $a_\gamma\in S^{-N,-\infty}$ and $Q_\gamma$ satisfying $\<x\>^{\gamma}\<D\>^{2\gamma}Q_N\in \mathbb B(L^2)$. Since $\<x\>^{-\gamma}\<A\>^{2\gamma}\<D\>^{-2\gamma}\<x\>^{-\gamma}\in \mathbb B(L^2)$ by Proposition \ref{proposition_PDO_1} and complex interpolation, we thus conclude that 
\begin{align*}
\norm{\<x\>^{-\gamma}\<A\>^{2\gamma}\varphi(\lambda^{-2}H)\<x\>^{-\gamma}}\lesssim \lambda^{-4\gamma/\mu-2},
\end{align*}
which implies \eqref{proof_proposition_smoothing_2_4}. Next we shall show
$$
\sup_{h\in (0,1]}\norm{\<x\>^{-\gamma}\varphi(H_h)\<A_h\>^{2\gamma}\varphi(H_h)\<x\>^{-\gamma}}<\infty
$$
which implies \eqref{proof_proposition_smoothing_2_3}. 
We decompose the operator in the left hand side into two parts
$$
\<x\>^{-\gamma}\varphi(H_h)\chi_R(x)\<A_h\>^{2\gamma}\varphi(H_h)\<x\>^{-\gamma},\quad
\<x\>^{-\gamma}\varphi(H_h)(1-\chi_R(x))\<A_h\>^{2\gamma}\varphi(H_h)\<x\>^{-\gamma},
$$
where $\chi_R$ is as in Section \ref{section_homogeneous}. For the first term, \cite[Theorem 1.1 and (2.4)]{Nak2} implies
$$
\norm{\<x\>^{-\gamma}\varphi(H_h)\chi_R\<x\>^\gamma}\le C_{M,R}h^M
$$
for any $M\in \N$ if $R>0$ is small enough. Choosing $M\ge 4\gamma$, we thus learn by \eqref{proof_proposition_smoothing_2_4} that
$$
\norm{\<x\>^{-\gamma}\varphi(H_h)\chi_R(x)\<A_h\>^{2\gamma}\varphi(H_h)\<x\>^{-\gamma}}\lesssim1.
$$ 
On the other hand, Proposition \ref{proposition_functional_3} yields that 
$
\<x\>^{-\gamma}\varphi(H_h)(1-\chi_R)=\Op(a_h)+Q_h$ with some $a_h\in S^{-\gamma,-\infty}$ and $Q_h$ satisfying $\norm{Q_h\<x\>^{2\gamma}\<hD\>^{2\gamma}}\lesssim h^{2\gamma}
$
for each $R>0$. Therefore, 
\begin{align*}
&\norm{\<x\>^{-\gamma}\varphi(H_h)(1-\chi_R(x))\<A_h\>^{2\gamma}\varphi(H_h)\<x\>^{-\gamma}}\\
&\le\norm{\Op(a_h)\<A_h\>^{2\gamma}\<x\>^{-\gamma}}\norm{\<x\>^\gamma\varphi(H_h)\<x\>^{-\gamma}}+\norm{Q_h\<A_h\>^{2\gamma}}\norm{\varphi(H_h)\<x\>^{-\gamma}}\lesssim 1,
\end{align*}
which completes the proof of \eqref{proof_proposition_smoothing_2_3}. \qed

%{lemma}
\begin{lemma}
\label{lemma_A_1}
For sufficiently small $h_0>0$, \eqref{proof_proposition_smoothing_2_2_1} holds for all $h\in (0,h_0]$. 
\end{lemma}

%proof
\begin{proof}By the spectral theorem, $\norm{H_hE_{H_h}(I)}\le |\sup I|\lesssim1$ uniformly in $h\in (0,1]$. Moreover, \eqref{dilation_3} implies $E_{H_h}(I)\chi(\lambda^{-2/\mu} x)=\D_\mu^* E_{\lambda^{-2}H}(I)\chi(x)\D_\mu$ where $h=\lambda^{2/\mu-1}$. Since $\mathds 1_I(\lambda^{-2}x)=\mathds1_{\lambda^2I}(x)$, we also have $E_{\lambda^{-2}H}(I)=\mathds1_{\lambda^2 I}(H)$. 
Therefore, it suffices to show $\norm{\mathds1_{\lambda^2I}(H)\chi(x)}\to0$ as $\lambda\to0$. 
Taking into account the formula $\mathds1_{\lambda^2I}(H)\chi(x)=\mathds1_{\lambda^2I}(H)\mathds1_{J}(H)\chi(x)$ for $0<\lambda\le1$ where $J=(0,\sup|I|)$, we see that this norm convergence follows from the facts that $\mathds1_{\lambda^2I}(H)\to 0$ strongly as $\lambda\to 0$ since $|\lambda^2I|\le \lambda^2\sup|I|$ and $\mathds1_{J}(H)\chi(x)$ is compact. 
\end{proof}

%%%%%%%%%% Bibliography %%%%%%%%%%%%%%%%%%%

\end{document}